\documentclass[10pt]{amsart}
\usepackage{amssymb, amsthm, amsfonts, amsmath, amscd}
\usepackage{enumitem}
\usepackage[utf8]{inputenc}
\usepackage[margin=1in]{geometry}
\usepackage{graphicx}
\usepackage{tikz}
\usepackage{tikz-cd}
\usepackage{mathrsfs}
\usepackage{hyperref}

\usetikzlibrary{decorations.markings}
\newmuskip\pFqmuskip

\theoremstyle{definition}

\newcommand*\pFq[6][8]{
  \begingroup
  \pFqmuskip=#1mu\relax
  \mathcode`\,=\string"8000
  \begingroup\lccode`\~=`\,
  \lowercase{\endgroup\let~}\pFqcomma
  {}_{#2}F_{#3}{\left[\genfrac..{0pt}{}{#4}{#5};#6\right]}
  \endgroup
}

\newcommand{\pFqcomma}{\mskip\pFqmuskip}
\newcommand\myeq{\mathrel{\stackrel{\makebox[0pt]{\mbox{\normalfont\tiny def}}}{=}}}

\newtheorem{thm}{Theorem}[section]
\newtheorem{lem}[thm]{Lemma}
\newtheorem{df}[thm]{Definition}
\newtheorem{prop}[thm]{Proposition}
\newtheorem{cor}[thm]{Corollary}
\newtheorem{rem}[thm]{Remark}

\newtheorem{defprop}[thm]{Definition-Proposition}

\newcommand{\C}{\mathbb{C}}

\newcommand{\R}{\mathbb{R}}
\newcommand{\Z}{\mathbb{Z}}
\newcommand{\ZZ}{\mathcal{Z}}

\newcommand{\Zp}{\mathbb{Z}_+}
\newcommand{\Zpe}{\mathbb{Z}_+^{\epsilon}}
\newcommand{\N}{\mathbb{N}}

\newcommand{\p}{\mathfrak{p}}
\newcommand{\X}{\mathfrak{X}}
\newcommand{\SSS}{\mathcal{S}}
\newcommand{\TR}{\widetilde{R}}
\newcommand{\TS}{\widetilde{S}}

\newcommand{\GT}{\mathbb{GT}}

\newcommand{\U}{\mathcal{U}}
\newcommand{\Tau}{\mathcal{T}}
\newcommand{\GTp}{\mathbb{GT}^+}
\newcommand{\Y}{\mathfrak{Y}}
\newcommand{\PP}{\mathcal{P}}

\newcommand{\OO}{\mathcal{O}}
\newcommand{\LL}{\mathcal{L}}

\newcommand{\hatx}{\widehat{x}}
\newcommand{\haty}{\widehat{y}}
\newcommand{\hatl}{\widehat{l}}

\newcommand{\adm}{\textrm{adm}}
\newcommand{\ort}{\textrm{ort}}
\newcommand{\princ}{\textrm{princ}}
\newcommand{\compl}{\textrm{compl}}
\newcommand{\degen}{\textrm{degen}}
\newcommand{\Conf}{\textrm{Conf}}
\newcommand{\fin}{\textrm{fin}}
\newcommand{\Res}{\textrm{Res}}

\numberwithin{equation}{section}

\begin{document}

\title{BC type Z-measures and determinantal point processes}

\author{Cesar Cuenca}

\address{Department of Mathematics, MIT, Cambridge, MA, USA}
\email{cuenca@mit.edu}

\date{}

\begin{abstract}

The (BC type) $z$-measures are a family of four parameter $z, z', a, b$ probability measures on the path space of the nonnegative Gelfand-Tsetlin graph with Jacobi-edge multiplicities. We can interpret the $z$-measures as random point processes $\PP_{z, z', a, b}$ on the punctured positive real line $\X = \R_{>0}\setminus\{1\}$. Our main result is that these random processes are determinantal and moreover we compute their correlation kernels explicitly in terms of hypergeometric functions.

For very special values of the parameters $z, z'$, the processes $\PP_{z, z', a, b}$ on $\X$ are essentially scaling limits of Racah orthogonal polynomial ensembles and their correlation kernels can be computed simply from some limits of the Racah polynomials. Thus, in the language of random matrices, we study certain analytic continuations of processes that are limits of Racah ensembles, and such that they retain the determinantal structure. Another interpretation of our results, and the main motivation of this paper, is the representation theory of big groups. In representation-theoretic terms, this paper solves a natural problem of harmonic analysis for several infinite-dimensional symmetric spaces.

\end{abstract}

\maketitle

\setcounter{tocdepth}{1}
\tableofcontents

\section*{Introduction}

A wide range of probabilistic and statistical mechanics models are integrable in the sense that observables of the system admit closed forms that can be analyzed in different limit regimes. The tools that are used to obtain closed form formulas for the observables, or correlation functions, are typically combinatorial or algebraic in nature and very often they originate from representation theory.

Some notable examples of integrable probability ensembles come from Schur measures, \cite{Ok1, OkRe}. The first kind of Schur measures that appeared in the literature are known as $z$-measures and they originated from the problem of harmonic analysis for the infinite symmetric group, see \cite{Ol1} for an introduction. A more sophisticated class of probability measures, known as $zw$-measures, originated from harmonic analysis of the infinite unitary group; various aspects of the $z$- and $zw$-measures were studied in the past decade, see for example \cite{BO1, BO2, BO3}.

All of the theory we have mentioned above is based on representation theory of \textit{type A}, that is, the algebraic machinery behind the scenes comes from the symmetric and unitary groups. In this paper we work with the `type BC' analogues of $zw$-measures that we call \textit{BC type $z$-measures}; they come from harmonic analysis of the infinite symplectic, orthogonal groups and other infinite-dimensional symmetric spaces. In a degenerate case, the BC type $z$-measures are equivalent to a scaling limit of the Racah orthogonal polynomial ensembles, which are at the top of the hierarchy of all discrete orthogonal polynomial ensembles of the classical Askey scheme.

Our analysis of the BC type $z$-measures is based on a map from the space where they are defined to the space of simple point configurations in $\R_{>0}\setminus\{1\}$. Under this map, the BC type $z$-measures define a four-parameter family of stochastic point processes $\PP_{z, z', a, b}$, with infinitely many points, in the punctured positive real line. We began a study on these processes in \cite{C}, where we constructed a continuous-time Markov chain that preserves the $z$-measures. In this paper, the random processes themselves are our subject of study: our main result states that they are \textit{determinantal point processes} and moreover we compute explicit correlation kernels in terms of special functions. The resulting kernel has appeared in the literature before, in \cite{BO1}, and is known as the \textit{hypergeometric kernel} in view of the fact that it can be expressed in terms of the hypergeometric function $_2F_1$. The previous appearance of the hypergeometric kernel was in the problem of harmonic analysis of the infinite-dimensional unitary group (`type A' analogue of our problem), but it is far from obvious that the same kernel should show up in our context; the author cannot offer at this point a conceptual explanation for it.

The principal theme of this paper is the `approximation' of the point processes $\PP_{z, z', a, b}$ on the continuous state space $\R\setminus\{1\}$ by a sequence $\{\LL^{(N)}_{z, z', a, b}\}_{N\geq 1}$ of point processes, with a finite number of points, on a quadratic lattice. For any fixed $N$, the point processes $\LL^{(N)}_{z, z', a, b}$ are determinantal and a correlation kernel for it can be expressed in terms of a finite system of orthogonal polynomials on a quadratic lattice, known as the \textit{Wilson-Neretin polynomials}. In addition, we need to have suitable expressions of the correlation kernels of $\LL^{(N)}_{z, z', a, b}$ for the limit transition $N\rightarrow\infty$. To obtain such expressions, we make use of a discrete Riemann-Hilbert problem, which has been used before in other problems from asymptotic representation theory. Finally, the expressions that admit a limit as $N\rightarrow\infty$ are obtained after massive calculations and use of recent formulas \cite{Mi1} that relate several generalized hypergeometric functions  $_4F_3$.

Below we give a more detailed account on the results of this paper.

\subsection*{BC type $z$-measures}

Let us fix two real parameters $a \geq b \geq -1/2$, and momentarily fix also a positive integer $N\in\N$. In this introduction, we only discuss BC type $z$-measures of level $N$ (or simply $z$-measures of level $N$), which are certain probability measures on the countable space of $N$-positive signatures $\GTp_N = \{\lambda = (\lambda_1, \ldots, \lambda_N) : \lambda_1 \geq \ldots \geq \lambda_N \geq 0, \ \lambda_i \in \Z \ \forall i\}$. The BC type $z$-measures depend on $a, b$ and on two additional complex parameters $z, z'$ subject to certain constraints, see Definition $\ref{Uadm}$ below. For example, the reader can think of the case $z' = \overline{z} \in \C \setminus \R$, or of the degenerate case $z = n\in\N$, $z' > n-1$ which gives rise to the Racah orthogonal polynomial ensemble. Explicitly, $z$-measures of level $N$ are
\begin{equation*}
P_N(\lambda | z, z', a, b) = \frac{1}{Z_N(z, z', a, b)}\cdot\widehat{\Delta}_N(l_1, \ldots, l_N)^2\cdot\prod_{i=1}^N{W(l_i  |  z, z', a, b; N)}, \ \lambda \in \GTp_N,
\end{equation*}
where we denoted by $l_i = \lambda_i + N - i$, $1\leq i\leq N$, the shifted coordinates of $\lambda = (\lambda_1, \ldots, \lambda_N)$; also $\widehat{\Delta}_N(l_1, \ldots, l_N) = \prod_{1\leq i < j\leq N}{\left(\left(l_i + \frac{a + b + 1}{2}\right)^2 - \left(l_j + \frac{a+b+1}{2}\right)^2\right)}$, the weight function is
\begin{equation*}
\begin{gathered}
W(x | z, z', a, b; N) = \left(x + \frac{a+b+1}{2}\right)\frac{\Gamma(x + a + b + 1)\Gamma(x + a + 1)}{\Gamma(x + b + 1)\Gamma(x + 1)}\\
\times\frac{1}{\Gamma(z - x + N)\Gamma(z' - x + N)\Gamma(z + x + N + a + b + 1)\Gamma(z' + x + N + a + b + 1)},
\end{gathered}
\end{equation*}
and $Z_N(z, z', a, b) > 0$ is a normalization constant that makes $P_N$ a probability measure. Observe that the weight function $W(\cdot | z, z', a, b; N)$ is not necessarily nonnegative on $\Z_{\geq 0}$ for all complex values of $z, z'$, but it is for some special paremeters, for example when $z' = \overline{z}$ and $z = n\in\N$, $z' > n - 1$ as before.

In what follows, we construct several point processes associated to the $z$-measures of level $N$. The reader can refer to \cite{DVJ, L} for generalities on point processes; see also \cite[Sec. 4.2]{AGZ} and \cite{B0} for generalities on determinantal point processes.

\subsection*{Point processes in the quadratic half-lattice}

Let $\epsilon = \frac{a + b + 1}{2} \geq 0$ and $\Zpe = \{\epsilon^2, (1 + \epsilon)^2, (2 + \epsilon)^2, \ldots\}$ be a quadratic half-lattice. Let $\Conf_{\fin}(\Zp^{\epsilon})$ be the space of multiplicity-free point configurations on $\Zpe$ with finitely many points. The space $\Conf_{\fin}(\Zp^{\epsilon})$ has a canonical sigma algebra and a probability measure on it determines a stochastic point process on $\Zpe$.

Under any map $\PP^{(N)}: \GTp_N \rightarrow \Conf_{\fin}(\Zpe)$, the pushforwards of $z$-measures of level $N$ yield random point processes $\PP^{(N)} = \PP^{(N)}_{z, z', a, b}$ on $\Zpe$ that depend on the parameters $z, z', a, b$, and that we denote by the same letter $\PP^{(N)}$.
Our goal is to obtain a point process with infinitely many points on the state space $\R_{>0}\setminus\{1\}$ as a natural scaling limit of the processes $\PP^{(N)}$ with finitely many points in the quadratic lattice $\Zpe$ (the scaling is of order $N^2$).
The most obvious map that we can consider is
\begin{equation*}
\OO^{(N)}(\lambda) = \{(\lambda_1 + N - 1 + \epsilon)^2, (\lambda_2 + N - 2 + \epsilon)^2, \ldots, (\lambda_N + \epsilon)^2\}, \ \forall \lambda\in\GTp_N.
\end{equation*}
It turns out that the processes $\OO^{(N)}$ do not have the desired scaling limit as $N$ tends to infinity. The point processes that we need arise from a different map that we now describe.

An $N$-positive signature $\lambda\in\GTp_N$ can be described uniquely by its Frobenius coordinates. If we let $d = d(\lambda)$ be the largest positive integer such that $\lambda_d \geq d$, then the Frobenius coordinates $(p_1, \ldots, p_d \ | \ q_1, \ldots, q_d)$ of $\lambda$ are given by
\begin{equation*}
p_i = \lambda_i - i, \hspace{.1in} q_i = \lambda_i' - i, \hspace{.1in}1\leq i\leq d,
\end{equation*}
where $\lambda_i' = |\{n \in \Z_{>0} : \lambda_n \geq i\}|$ for all $i$. It is clear that $p_1 > \ldots > p_d\geq 0$ and $q_1 > \dots > q_d \geq 0$. Now consider the map $\LL^{(N)} : \GTp_N \rightarrow \Conf_{\fin}(\Zpe)$ given by
\begin{equation*}
\LL^{(N)}(\lambda) = \{(N + p_1 + \epsilon)^2 > \ldots > (N + p_d + \epsilon)^2 > (N - 1 - q_d + \epsilon)^2 > \ldots > (N - 1 - q_1 + \epsilon)^2\}.
\end{equation*}
The point processes $\LL^{(N)}$ have a natural scaling limit that we describe next.

\subsection*{Point processes on the continuous state space $\R_{>0}\setminus\{1\}$}

Let $\X = \R_{>0} \setminus\{1\} = (0, 1) \sqcup (1, \infty)$ and $\Conf(\X)$ be the space of multiplicity free point configurations on $\X$, with its canonical $\sigma$-algebra. The point processes $\LL^{(N)}$ described above are probability measures on the space $\Conf_{\fin}(\Zpe)$. Consider the composition
\begin{eqnarray*}
\mathfrak{j}_N : \Zpe &\longrightarrow& \frac{1}{\left(N + \epsilon - \frac{1}{2}\right)^2}\Zpe \hookrightarrow\X\\
x &\longmapsto& \frac{x}{\left(N + \epsilon - \frac{1}{2}\right)^2},
\end{eqnarray*}
which induces a map $\mathfrak{j}_N: \Conf_{\fin}(\Zpe)\rightarrow\Conf(\X)$, denoted by the same letter (the normalization by $\left(N+\epsilon - \frac{1}{2}\right)^2$ is chosen so that $\frac{\Zpe}{(N+\epsilon - 1/2)^2}\subset\X$). Denote by $\widetilde{\PP}^{(N)}$ the pushforward of $\LL^{(N)}$ under the map $\mathfrak{j}_N$ above; then $\widetilde{\PP}^{(N)}$ is a probability measure on $\Conf(\X)$ and determines a point process on $\X$ that we also denote by $\widetilde{\PP}^{(N)}$. 

Then the point processes $\widetilde{\PP}^{(N)}$ converge to certain point process $\PP = \PP_{z, z', a, b}$ in the sense that all the correlation functions of $\widetilde{\PP}^{(N)}$ converge to the corresponding correlation functions of $\PP$, as $N$ tends to infinity. The process $\PP$ is our hero: the main result of this paper is a complete characterization of the stochastic point process $\PP$ for very general parameters $z, z', a, b$.

Let us remark that the point process $\PP$ has an intrinsic definition, which is independent of the finite point processes $\widetilde{\PP}^{(N)}$. However we are able to thoroughly study $\PP$ because we can realize it as a scaling limit of the processes $\widetilde{\PP}^{(N)}$, see Section $\ref{sec:scalinglimit}$ for details.

\subsection*{The main theorem}

Let $a\geq b\geq -1/2$ and $(z, z')\in\C^2$ be an admissible pair (see Definition $\ref{Uadm}$ below, or simply think of the cases $z' = \overline{z}\in\C\setminus\R$ or $z=n\in\N$, $z'>n-1$ for now). The point processes $\PP = \PP_{z, z', a, b}$ on $\X$ are determinantal. Moreover an explicit correlation kernel $K^{\PP}(x, y)$ for $\PP$ is given by the expressions in $(\ref{KernelP})$ and $(\ref{KernelP2})$ below. For example, if $x, y > 1$, $x\neq y$, then
\begin{equation}\label{eq:kernelintro}
K^{\PP}(x, y) = \sqrt{\psi(x)\psi(y)}\cdot\frac{R(x)S(y) - R(y)S(x)}{x - y},
\end{equation}
where
\begin{eqnarray*}
\psi(x) &=& \frac{\sin(\pi z)\sin(\pi z')}{2\pi^2}x^{b}(x - 1)^{-z-z'},\\
R(x) &=& \left(1 - \frac{1}{x}\right)^{z'}\pFq{2}{1}{z' + b ,,, z'}{z + z' + b}{\frac{1}{x}},\\
S(x) &=& \frac{2}{x}\left(1 - \frac{1}{x}\right)^{z'}\frac{\Gamma(1+z)\Gamma(1+z')\Gamma(1+z+b)\Gamma(1+z'+b)}{\Gamma(1+z+z'+b)\Gamma(2+z+z'+b)}\pFq{2}{1}{z' + b + 1 ,,, z' + 1}{z + z' + b + 2}{\frac{1}{x}}.
\end{eqnarray*}
The kernel $K^{\PP}: \X\times\X \rightarrow \R$ is known as the \textit{hypergeometric kernel}, the reason being that $K^{\PP}$ can be expressed in terms of the hypergeometric functions $_2F_1$, see \cite{BO1} for another problem in which this kernel appears. Observe that the expression for $K^{\PP}$ in $(\ref{eq:kernelintro})$ is only defined for $x \neq y$ due to the singularities on the diagonal. However it admits the following analytic continuation:
\begin{equation*}
K^{\PP}(x, x) = \psi(x)\cdot \left(R'(x)S(x) - R(x)S'(x)\right) \ \forall x>1.
\end{equation*}

\subsection*{Harmonic analysis on big groups}

We follow the exposition of \cite{OkOl1}.

An \textit{infinite symmetric space} $G/K$ is an inductive limit of Riemannian symmetric spaces $G(n)/K(n)$ of rank $n$, with respect to a natural chain of maps $G(1) \rightarrow \dots \rightarrow G(n)\rightarrow G(n+1)\rightarrow\dots$ that is compatible with the inclusions $K(n)\subset G(n)$. Some examples of such infinite symmetric spaces are the following
\begin{enumerate}
	\item $G/K = U(2\infty)/U(\infty)\times U(\infty) =  \lim_{n\rightarrow\infty}{(U(2n)/U(n) \times U(n))}$.
	\item $G/K = O(\infty)\times O(\infty)/O(\infty) = \lim_{n\rightarrow\infty}{(O(\widetilde{n})\times O(\widetilde{n})/ O(\widetilde{n}))}$.
	\item $G/K = Sp(\infty)\times Sp(\infty)/Sp(\infty) = \lim_{n\rightarrow\infty}{(Sp(n)\times Sp(n)/Sp(n))}$.
\end{enumerate}

(In $(2)$, we wrote $\widetilde{n}$ to indicate that the rank of the Riemannian symmetric space in question is $\lfloor \widetilde{n}/2 \rfloor$, and $\widetilde{n}$ can be either $2n$ or $2n+1$; either one leads to the same $G/K$, up to isomorphism.)

The natural problem of noncommutative harmonic analysis for finite or compact groups asks for the decomposition of the regular representation into irreducible representations. In this context, the solution is given by the well known Peter-Weyl theorem. In the infinite-dimensional context, it is not even clear what the problem of harmonic analysis should be. The difficulty is that there is no analogue of the Haar measure. The problem of harmonic analysis for the infinite-dimensional unitary group $U(\infty)$ was posed by G. Olshanski in \cite{Ol}. He used certain completion $\mathfrak{U}$ of $U(\infty)$, which admits an analogue of the Haar measure. We hope to explain the representation theoretic picture for the symmetric spaces listed above elsewhere, but for now let us only describe the problem in terms of spherical functions.

A \textit{spherical function} of the infinite symmetric space $G/K$ is a function $\phi: G \rightarrow \C$ satisfying
\begin{itemize}
	\item $\phi$ is a $K$-bi-invariant function, i.e., $\phi(k_1gk_2) = \phi(g)$ for all $g\in G$, $k_1, k_2\in K$.
	\item $\phi$ is normalized by $\phi(I) = 1$, where $I\in G$ is the unit element.
	\item $\phi$ is positive definite, i.e., for any $g_1, \ldots, g_n\in G$, the matrix $[\phi(g_i g_j^{-1})]_{i, j=1}^n$ is positive definite.
\end{itemize}

It is clear that the space of spherical functions of $G/K$ is convex. The extreme points of the set of spherical functions are the \textit{irreducible (or extreme) spherical functions}. The space of irreducible spherical functions for all three spaces (1), (2), (3) above are isomorphic; that space will be denoted by $\Omega_{\infty}$. In fact, it was shown in \cite{OkOl1} that $\Omega_{\infty}$ can be embedded as the subspace of $\R_+^{\infty}\times\R_+^{\infty}\times\R_+$ consisting of triples $\omega = (\alpha; \beta; \delta)$ satisfying
\begin{equation*}
\begin{gathered}
\alpha_1 \geq \alpha_2 \geq \dots \geq 0, \hspace{.2in} 1 \geq \beta_1 \geq \beta_2 \geq \dots \geq 0,\\
\infty > \delta \geq \sum_{i=1}^{\infty}{(\alpha_i + \beta_i)}.
\end{gathered}
\end{equation*}
Let us denote by $\phi^{\omega}$ the irreducible spherical function corresponding to $\omega\in\Omega_{\infty}$.

In analogy to \cite{Ol} it is possible, for each of the three infinite symmetric spaces above, to construct a family of natural \textit{generalized quasi-regular representations of $G/K$}, each with a distinguished cyclic $K$-invariant vector. The members of the family of generalized quasi-regular representations $\{(T^z, v^z)\}_z$ (we let $v^z$ be the distinguished vector of the representation $T^z$) of $G/K$ are parameterized by complex numbers $z$ satisfying $\Re z > -\frac{b}{2}$, where $b = 0, -\frac{1}{2}, \frac{1}{2}$ for (1), (2) and (3), respectively. Each representation $T^z$ is determined completely by its spherical function $\phi^z$ defined by $\phi^z(\cdot) = (T^z(\cdot)v^z, v^z)$. As it turns out, for our three cases of interest, the spherical functions $\phi^z : G \rightarrow\R$ admit analytic continuations in the variables $z$ and $z' = \overline{z}$ to the domain $\{\Re (z+z') > -b\}$, where $b = 0, -\frac{1}{2}, \frac{1}{2}$ for the spaces (1), (2), (3) above. Furthermore, one can write down a spherical function $\phi^{z, z', a, b} : G\rightarrow\R$ depending on four parameters $z, z', a, b$ (subject to some constraints) that reduces to the representation-theoretic spherical functions of $G/K$ upon specialization of the parameters $a, b$ as follows:
\begin{enumerate}
	\item $(a, b) = (0, 0)$.
	\item $(a, b) = (\pm\frac{1}{2}, -\frac{1}{2})$.
	\item $(a, b) = (\frac{1}{2}, \frac{1}{2})$.
\end{enumerate}

(In (2),  $a = -\frac{1}{2}$ if $\widetilde{n} = 2n$ and $a = \frac{1}{2}$ if $\widetilde{n} = 2n+1$.)

A general result from functional analysis implies that all spherical functions $\phi^{z, z', a, b}$ of $G/K$ are mixtures of the irreducible spherical functions $\{\phi^{\omega}\}_{\omega\in\Omega_{\infty}}$, in the sense that there exists a Borel probability measure $\pi_{z, z', a, b}$ on $\Omega_{\infty}$ such that
\begin{equation*}
\phi^{z, z', a, b} = \int_{\Omega_{\infty}}{\phi^{\omega}d\pi_{z, z', a, b}(\omega)}.
\end{equation*}
The probability measure $\pi_{z, z', a, b}$ on $\Omega_{\infty}$ is the \textit{$z$-measure with parameters $z, z', a, b$} and also known as the \textit{spectral measure of $\phi^{z, z', a, b}$}.

The probability measures $\pi_{z, z', a, b}$ describe how the spherical function $\phi^{z, z', a, b}$ is `decomposed' into irreducible spherical functions. A (far reaching) generalization of the problem of harmonic analysis for the infinite symmetric spaces (1), (2), (3) above is therefore the following:\\

\textbf{Problem:} Describe the spectral measures $\pi_{z, z', a, b}$ on $\Omega_{\infty}$ for the most general set of parameters $z, z', a, b$ possible.\\

The content of the present article solves the problem just posed. In fact, the pushforwards of the measures $\pi_{z, z', a, b}$ under the map
\begin{equation}\label{eqn:mapi}
\begin{gathered}
\mathfrak{i} : \Omega_{\infty} \longrightarrow \Conf(\X)\\
\omega = (\alpha; \beta; \delta) \mapsto \left( \{(1 + \alpha_i)^2\}_{i\geq 1} \sqcup \{(1 - \beta_j)^2\}_{j\geq 1} \right) \setminus \{0, 1\}.
\end{gathered}
\end{equation}
define the point processes $\PP_{z, z', a, b}$ on $\X$, and our main result is a complete characterization of these processes, see Section $\ref{sec:scalinglimit}$ below for more details.

In particular, Theorem $\ref{maintheorem}$ with $b = 0, -\frac{1}{2}, \frac{1}{2}$ solves the problem of harmonic analysis, stated above for the infinite symmetric spaces in (1), (2) and (3), respectively.

The reader is referred to \cite{BO4} for a self-contained survey of the problem of harmonic analysis of the infinite symmetric space $(G, K) = (U(\infty)\times U(\infty), U(\infty))$ and some applications.

\subsection*{Racah orthogonal polynomial ensemble}

We have mentioned that BC type $z$-measures are esentially analytic continuations of scaling limits of Racah orthogonal polynomial ensembles. Let us clarify the meaning of that statement and briefly explain how the main result of this paper can be deduced from limits of Racah polynomials in the degenerate case.

Let $z = k\in\N$ and $z'> k-1$ (similar considerations hold if $z' = k$ and $z > k-1$). For any $N\geq 1$, the corresponding $z$-measure $P_N = P_N(\cdot | n, z', a, b)$ of level $N$ is supported on the set
\begin{equation*}
\GTp_N(k) \myeq \{ \lambda\in\GTp_N : k \geq \lambda_1 \geq \cdots \geq \lambda_N \geq 0 \}.
\end{equation*}
The process $\OO^{(N)} = \OO^{(N)}_{k, z', a, b}$ on $\Zpe$ coming from $P_N$ lives in the finite space $\Z^{\epsilon}_{N+k-1} \myeq \{\epsilon^2, (1+\epsilon)^2, \ldots, (N+k-1+\epsilon)^2\}$. The corresponding $z$-measure $\pi = \pi_{k, z', a, b}$ is supported on the finite dimensional simplex
\begin{equation*}
\Omega_{\infty}(k) \myeq \{\omega\in\Omega_{\infty} : \alpha = \delta = 0, \ \beta_{k+1} = \beta_{k+2} = \dots = 0 \}.
\end{equation*}
The point process $\PP = \PP_{k, z', a, b}$ that corresponds to $\pi$ under the map $\mathfrak{i}$ given in $(\ref{eqn:mapi})$ lives in the open interval $(0, 1)$. Such process is the scaling limit of the processes $\LL^{(N)}_{k, z', a, b}$, as described previously. In this special case, it will also be the scaling limit of some processes $\mathcal{Y}^{(N)} = \mathcal{Y}^{(N)}_{k, z', a, b}$, which come from the pushforwards of $\OO^{(N)}$ under the maps
\begin{equation*}
\begin{gathered}
\Conf(\Z^{\epsilon}_{\leq N+k-1}) \rightarrow \Conf(\Z^{\epsilon}_{\leq N+k-1})\\
X \mapsto \Z^{\epsilon}_{\leq N+k-1} \setminus X.
\end{gathered}
\end{equation*}

Clearly $\mathcal{Y}^{(N)}$ is a point process on $\Z^{\epsilon}_{\leq N+k-1}$ with exactly $k$ points. Simple calculations show that $\mathcal{Y}^{(N)}$ is a Racah orthogonal polynomial ensemble associated to the parameters
\begin{equation}\label{eqn:racah}
\alpha = -k-N, \hspace{.1in} \beta = -z'-N-a, \hspace{.1in} \gamma = b, \hspace{.1in} \delta = a.
\end{equation}
In other words, if we denote $\left(y + \frac{a+b+1}{2}\right)^2$ by $\haty$, for any $y\in\{0, 1, \ldots, N+k-1\}$, then
\begin{equation*}
\begin{gathered}
\textrm{Prob}\left( \mathcal{Y}^{(N)}_{k, z', a, b} = \{\haty_1, \ldots, \haty_k\} \right) = \textrm{const}\cdot\prod_{1\leq i<j\leq k}{(\haty_i - \haty_j)^2}\cdot\prod_{i=1}^k{w_R(y_i(y_i+a+b+1)\ |\ \alpha, \beta, \delta, \gamma)},\\
w_R(y(y+a+b+1)) \myeq \left(y + \frac{\gamma + \delta + 1}{2} \right)\frac{\Gamma(y+\gamma+\delta+1)\Gamma(y+\gamma+1)\Gamma(y+\alpha+1)\Gamma(\beta-\gamma-y)}{\Gamma(y+1)\Gamma(y+\delta+1)\Gamma(y-\alpha+\gamma+\delta+1)\Gamma(-\beta-\delta-y)},
\end{gathered}
\end{equation*}
where $\alpha, \beta, \delta, \gamma$ are related to $k, z', a, b$ via the identification of parameters in $(\ref{eqn:racah})$. The weight function $w_R(y(y+a+b+1)\ |\ \alpha, \beta, \gamma, \delta)$ on the quadratic lattice $\{n(n+a+b+1) : n\in\Zp\}$ corresponds to the classical \textit{Racah orthogonal polynomials} $\mathfrak{R}_n(y(y+a+b+1) | \alpha, \beta, \gamma, \delta)$, see \cite[Ch. 9.2]{KLS} and \cite[Ch. 3]{NSU}. One deduces that $\mathcal{Y}^{(N)}$ is a determinantal process with correlation functions of the form
\begin{equation*}
\begin{gathered}
\textrm{Prob}\left( \{ \haty_1, \ldots, \haty_p \} \subseteq \mathcal{Y}^{(N)}_{k, z', a, b} \right) = \textrm{const}\cdot\mathbf{1}_{\{ p \leq k \}}\cdot\det_{1\leq i, j\leq p}\left[ K^{\mathcal{Y}^{(N)}}(\haty_i, \haty_j) \right],\\
K^{\mathcal{Y}^{(N)}}(\hatx, \haty) = \frac{\sqrt{w_R(\widetilde{x})w_R(\widetilde{y})}}{H_{k-1}}\cdot\frac{\mathfrak{R}_k(\widetilde{x})\mathfrak{R}_{k-1}(\widetilde{y})  -  \mathfrak{R}_{k-1}(\widetilde{x})\mathfrak{R}_k(\widetilde{y})}{\hatx - \haty},
\end{gathered}
\end{equation*}
where $\widetilde{x} = x(x+a+b+1)$, $\widetilde{y} = y(y+a+b+1)$, $\hatx = (x+(a+b+1)/2)^2$, $\haty = (y + (a+b+1)/2)^2$ and $H_{k-1}$ is the squared norm of $\mathfrak{R}_{k-1} = \mathfrak{R}_{k-1}(\cdot | \alpha, \beta, \gamma, \delta)$.

The point process $\PP = \PP_{k, z', a, b}$ has exactly $k$ points and is a scaling limit of the processes $\mathcal{Y}^{(N)}$, all of which have $k$ points. It is easily justified in this case that $\PP$ is a determinantal point process and a calculation of its correlation kernel $K^{\PP}$ would follow from a suitable limit of the Racah polynomials and of the kernel $K^{\mathcal{Y}^{(N)}}$ above.

Alternatively the same kernel $K^{\PP}$ can be read from Theorem $\ref{maintheorem}$ after setting $z = k$ (in that theorem, note $\psi_{>1}(x) = 0$ for $z=k\in\N$, which implies that there is no particles in $(1, \infty)$). The general point processes $\PP_{z, z', a, b}$ that we consider ($z, z'$ not integers) can be considered as analytic continuations of the limits $\PP_{k, z', a, b}$ of Racah ensembles $\mathcal{Y}^{(N)}$, but we need much more work to study them.

\subsection*{Other results}

In addition to the main result described above, other results in this article are:
\begin{itemize}
	\item The point process $\OO^{(N)}$ is a \textit{discrete orthogonal polynomial ensemble}. Moreover we find a correlation kernel $K^{\OO^{(N)}}$ in terms of the generalized hypergeometric function $_4F_3(1)$, see Theorem $\ref{thm:Okernel}$.
	\item The point process $\LL^{(N)}$ is an \textit{$L$-ensemble}. Moreover we find an explicit correlation kernel $K^{\LL^{(N)}}$ in terms of the generalized hypergeometric function $_4F_3(1)$, see Theorem $\ref{thm:Lkernel}$.
\end{itemize}

\subsection*{Organization of the paper}

In the next section, we introduce some terminology that is used throughout the paper. In section $\ref{zprocesses}$, we introduce the BC type $z$-measures of level $N$ for a large set of parameters $z, z', a, b$, and introduce the point processes $\OO^{(N)}$, $\LL^{(N)}$ on the state space $\Zpe$. In section $\ref{sec:OO}$, we prove that $\OO^{(N)}$ is a discrete orthogonal polynomial ensemble and find a correlation kernel. In section $\ref{sec:Lensemble}$, we prove that $\LL^{(N)}$ is an $L$-ensemble. After recalling the theory of the discrete Riemann-Hilbert problem in section $\ref{sec:DRHP}$, we find a correlation kernel for $\LL^{(N)}$ in section $\ref{kernelL}$. The point processes $\LL^{(N)}$ yield a point process $\PP$ with infinitely many particles in a continuous state space, under a suitable scaling limit. We construct these point processes as scaling limits of the point processes $\LL^{(N)}$ in section $\ref{sec:scalinglimit}$. In section $\ref{sec:maintheoremkernel}$, we prove the main theorem of this paper, which is an explicit formula for the correlation kernel of $\PP$. Finally, the appendix contains several technical proofs of various statements throughout the paper.

\subsection*{Acknowledgments}

I would like to thank Alexei Borodin and Grigori Olshanski for many helpful discussions, e-mail correspondence and comments on a draft of this paper.

\section{Notation and Terminology}\label{sec:notation}

We collect here some of the terminology and notation that is used throughout the paper, for the reader's convenience.

\begin{itemize}[leftmargin=*]
	\item The parameters $a, b$ in this paper are real numbers, and we define
\begin{equation*}
\epsilon \myeq \frac{a+b+1}{2}.
\end{equation*}
We assume throughout that $a, b$ satisfy
\begin{equation*}
a \geq b \geq -1/2,
\end{equation*}
which implies $\epsilon \geq 0$. This restriction allows us to use the main results from \cite{OkOl1, OlOs}.

	\item In Section $\ref{zprocesses}$ below, we define several subsets of $\C^2$ from which we take pairs $(z, z')$ as parameters of the $z$-measures and for some proofs involving analytic continuations. Some of these sets are the following \textit{domains} of $\C^2$:
\begin{eqnarray*}
\U &\myeq& \{(z, z')\in\C^2 : \Re (z+z'+b) > -1\},\\
\U_0 &\myeq& \{ (z, z')\in\U : \{z, z', z + b, z' + b \} \cap \{\ldots, -3, -2, -1\} = \emptyset \}.
\end{eqnarray*}
Moreover we have the more complicated set of \textit{admissible parameters} defined by
\begin{equation*}
\U_{\adm} \myeq \{(z, z')\in\U_0 : (z, z'), (z+2\epsilon, z'+2\epsilon)\in\mathcal{Z}\},
\end{equation*}
where $\mathcal{Z}$ is defined below in $(\ref{defnZ})$. The set $\U_{\adm}$ is not a domain of $\C^2$. The following inclusions hold, see Lemma $\ref{UadmsubsetUzero}$ below:
\begin{equation*}
\U_{\adm}\subset\U_0\subset\U.
\end{equation*}
Given any $(z, z')\in\U$, we denote $\Sigma \myeq z+z'+b$. Observe that if $(z, z')\in\U_{\adm}$, then $\Sigma\in (-1, +\infty)$.

	\item We write $\Zp$ and $\N$ for the set of nonnegative and positive integers, respectively. We also write
\begin{equation*}
\Zpe \myeq \{(n + \epsilon)^2 : n\in\Zp\}
\end{equation*}
for the quadratic half-lattice. For any $x\in\Z$, we let
\begin{equation*}
\hatx \myeq (x+\epsilon)^2.
\end{equation*}
Very often, starting in Section $\ref{sec:OO}$, we shall need the splitting $\Zpe = \Z^{\epsilon}_{\geq N}\sqcup\Z^{\epsilon}_{< N}$ into the sets $\Z^{\epsilon}_{\geq N} \myeq \{(N + \epsilon)^2, (N + 1 + \epsilon)^2, \ldots\}$ and $\Z^{\epsilon}_{< N} \myeq \{\epsilon^2, (1 + \epsilon)^2, \ldots, (N-1+\epsilon)^2\}$, for some $N\in\N$.

	\item For any $N\in\N$, the set $\GTp_N$ of \textit{positive signatures of length $N$} (also known as the set of \textit{$N$-positive signatures}) is defined as
\begin{equation*}
\GTp_N \myeq \{\lambda = (\lambda_1, \ldots, \lambda_N) \in \Zp^N : \lambda_1 \geq \cdots \geq \lambda_N \geq 0\}.
\end{equation*}
To each positive $N$-signature $\lambda = (\lambda_1 \geq \cdots \geq \lambda_N)\in\GTp_N$ we can associate the $N$-tuples $l = (l_1 > \cdots > l_N) \in \Zp^N$ and $\hatl = (\hatl_1 > \cdots > \hatl_N) \in (\Zpe)^N$ via the equations
\begin{equation*}
\begin{gathered}
l_i = \lambda_i + N - i, \hspace{.2in} \hatl_i = (l_i + \epsilon)^2, \hspace{.1in} \forall \ 1\leq i\leq N.
\end{gathered}
\end{equation*}
Below we shall often be working with a positive signature of length $N$ denoted by $\lambda\in\GTp_N$. We denote its two associated $N$-tuples given as above by $l = (l_1 > \ldots > l_N)$ and $(\hatl_1 > \ldots > \hatl_N)$.

	\item We use the following simplified notation for products of (ratios of) Gamma functions
\begin{eqnarray*}
\Gamma\left[ a_1, \ a_2, \ldots, \ a_m  \right] &\myeq& \Gamma(a_1)\Gamma(a_2)\cdots\Gamma(a_m),\\
\Gamma\left[ \begin{split}
a_1 &,& a_2 &,& \cdots &,& a_m\\
b_1 &,& b_2 &,& \cdots &,& b_n
\vphantom{\frac12}\end{split}
\right] &\myeq& \frac{\Gamma(a_1)\Gamma(a_2)\cdots\Gamma(a_m)}{\Gamma(b_1)\Gamma(b_2)\cdots\Gamma(b_n)}.
\end{eqnarray*}

	\item If $x = (x_1, \ldots, x_N)$ is any $N$-tuple of complex numbers, then the \textit{Vandermonde determinant} is
\begin{equation*}
\Delta_N(x) = \Delta_N(x_1, \ldots, x_N) \myeq \prod_{1\leq i<j\leq N}{(x_i - x_j)}.
\end{equation*}
\end{itemize}

\section{BC type $z$-measures}\label{zprocesses}

For any $N\in\N$, the set of \textrm{$N$-positive signatures} is $\GTp_N \myeq \{\lambda = (\lambda_1 \geq \ldots \geq \lambda_N)\in\Z^N\}$. Let us begin by considering the following function on $N$-positive signatures $\lambda\in\GTp_N$, which depends on certain parameters $z, z'\in\C$, $a \geq b \geq -1/2$, and is given by
\begin{equation}\label{zmeasureN}
P_N'(\lambda | z, z', a, b) = \left(\Delta_N(\hatl)\right)^2\cdot\prod_{i=1}^N{w(l_i  |  z, z', a, b; N)},
\end{equation}
where $\Delta_N(\hatl) = \prod_{1\leq i<j\leq N}{(\hatl_i - \hatl_j)}$ is the Vandermonde determinant on the variables $\hatl_i = (l_i + \epsilon)^2$, $l_i = \lambda_i + N - i$, $1\leq i\leq N$, and
\begin{equation}\label{weight}
\begin{gathered}
w(x | z, z', a, b; N) = (x + \epsilon)\times \Gamma\left[ \begin{split}
x + 2\epsilon &,& x + a + 1\\
x + 1 &,& x + b + 1
\vphantom{\frac12}
\end{split}
\right]\\
\times\frac{1}{\Gamma[z - x + N, \ z' - x + N, \ z + x + N + 2\epsilon, \ z' + x + N + 2\epsilon]}.
\end{gathered}
\end{equation}
When there is no risk of confussion, we write $w(x | z, z', a, b; N)$ simply as $w(x)$, for simplicity of notation. Observe that the conditions on the real parameters $a, b$ ensure that $\Gamma(x + 2\epsilon)$ and $\Gamma(x + a + 1)$ are well defined for all $x\in\Zp$ and therefore $P_N'(\lambda | z, z', a, b)$ is well defined for all $\lambda\in\GTp_N$.

The following equality
\begin{equation}\label{partitionfunction}
\sum_{\lambda\in\GTp_N}{P_N'(\lambda | z, z', a, b)} = S_N(z, z', a, b),
\end{equation}
where
\begin{equation*}
S_N(z, z', a, b)= \prod_{i=1}^N{\frac{\Gamma[b + z + z' + i, \ a + i, \ i]}{\Gamma[z + i, \ z + b + i, \ z' + i, \ z' + b + i, \ z + z' + a + b + N + i]}},
\end{equation*}
was proved in \cite{OlOs} for all $z, z' \in \{u \in \C : \Re u > - \frac{(1+ b)}{2}\}$ and extended by analytic continuation, in \cite{C}, for all pairs $(z, z')$ in the complex domain
\begin{equation}
\U \myeq \left\{ (z, z') \in \C^2 : \Re(z + z' + b) > -1 \right\}.
\end{equation}

Define also the subdomain $\U_0 \subset \U$ consisting of those pairs $(z, z')$ for which $S_N(z, z', a, b)$ never vanishes; clearly such set is
\begin{equation}\label{subdomain}
\U_0 \myeq \left\{ (z, z')\in\U : \{z, z', z + b, z' + b \} \cap \{\ldots, -3, -2, -1\} = \emptyset \right\}.
\end{equation}
For any pair $(z, z')\in\U_0$, we can therefore define the (complex) measure on $\GTp_N$ given by
\begin{equation}\label{zmeasure}
P_N(\lambda | z, z', a, b) = \frac{P_N'(\lambda | z, z', a, b)}{S_N(z, z', a, b)}.
\end{equation}
From $(\ref{partitionfunction})$, the measure $P_N$ has finite total measure equal to $1$. Next, we find very general conditions on complex pairs $(z, z')$ which guarantee that $\{P_N(\lambda | z, z', a, b)\}_{\lambda\in\GTp_N}$ is a probability measure on $\GTp_N$.

By following common terminology on $zw$-measures, see for example \cite{BO1}, define the sets
\begin{eqnarray}
\ZZ &\myeq& \ZZ_{\princ} \sqcup \ZZ_{\compl} \sqcup \ZZ_{\degen}\label{defnZ}\\
\ZZ_{\princ} &\myeq& \{(z, z')\in\C^2 \setminus \R^2 : z' = \overline{z}\}\nonumber\\
\ZZ_{\compl} &\myeq& \{(z, z')\in\R^2 :  n < z, z' < n + 1 \textrm{ for some }n\in\Z\}\nonumber\\
\ZZ_{\degen} &\myeq& \bigsqcup_{n=1}^{\infty}{\ZZ_{\degen, n}}\nonumber\\
\ZZ_{\degen, n} &\myeq&\{(z, z')\in\R^2 : z = n,\ z' > n - 1, \  \textrm{ or }\ z' = n,\ z > n - 1\}.\nonumber
\end{eqnarray}

The subscripts `princ', `compl' and `degen' stand for \textit{principal, complementary} and \textit{degenerate}.

The reason why the definitions above are important for us is the following statement.
\begin{lem}[{{\cite[Lemma 7.9]{Ol}}}]\label{propposdomain}
Let $(z, z')\in\C^2$, then
\begin{itemize}
	\item $(\Gamma(z - k)\Gamma(z' - k))^{- 1} \geq 0$ for all $k\in\Z$ iff $(z, z')\in\ZZ$.
	\item $(\Gamma(z - k)\Gamma(z' - k))^{- 1} > 0$ for all $k\in\Z$ iff $(z, z')\in\ZZ_{\princ}\sqcup\ZZ_{\compl}$.
	\item If $(z, z')\in\ZZ_{\degen, n}$, then $(\Gamma(z - k)\Gamma(z' - k))^{- 1} = 0$ for $k\in\Z$, $k \geq n$ and $(\Gamma(z - k)\Gamma(z' - k))^{- 1} > 0$ for $k\in\Z$, $k\leq n-1$.
\end{itemize}
\end{lem}

From the lemma above, we can define a set of complex pairs $(z, z')$ such that the numbers $P_N'(\lambda | z, z', a, b)$ are nonnegative real values for all positive $N$-signatures.

\begin{df}\label{Uadm}
The \textit{admisible domain} $\U_{\adm}\subset\C^2$ is defined as the set of complex pairs $(z, z')\in\C^2$ such that $\Re(z + z' + b) > -1$ and both $(z, z')$, $(z + 2\epsilon, z' + 2\epsilon)$ belong to $\ZZ$. A pair $(z, z')\in\U_{\adm}$ is called an \textit{admissible pair}.
\end{df}

Thus for $(z, z')\in\U_{\adm}\cap\U_0$, then $P_N'(\lambda | z, z', a, b) \geq 0$ and not all of them are zero. Then $(\ref{partitionfunction})$ shows $S_N(z, z', a, b) > 0$ and therefore the ratios $(\ref{zmeasure})$ define a probability measure on $\GTp_N$. We actually only need to require $(z, z')\in\U_{\adm}$, which automatically implies $(z, z')\in\U_0$, as the following lemma implies.

\begin{lem}\label{UadmsubsetUzero}
The inclusion $\U_{\adm} \subset \U_0$ holds, where $\U_0$ is defined in $(\ref{subdomain})$ and $\U_{\adm}$ in Definition $\ref{Uadm}$.
\end{lem}
\begin{proof}
Let $(z, z')\in\U_{\adm}$. Then $\Re(z + z' + b) >  - 1$, which implies $(z, z')\in\U$. By the definition of $\U_0$, we simply need to show that none of the four complex numbers $z, z', z + b, z' + b$ is a negative integer. We argue in three cases.

If $(z, z')\in\ZZ_{\princ}$, then $z, z'\notin\R$ which implies that none of $z, z', z + b, z' + b$ is a real number.

If $(z, z')\in\ZZ_{\compl}$, then there exists $n\in\Z$ for which $n < z, z' < n+1$. Coupled with the conditions $b \geq -1/2$ and $\Re(z + z' + b) > -1$, we have $n \geq -1$. Therefore $z, z' > -1$ and so $z, z'$ cannot be negative integers. Moreover $z + b, z' + b > -1 +b \geq -3/2$, so $z+b, z'+b$ cannot take on integer values less than $-2$. If $z + b = -1$, then $n = -1$ in the analysis above, so $-1 < z, z' < 0$. Then $z + z' + b = -1 + z' < -1$, a contradiction. Similarly $z' + b \neq -1$ and we are done with this case.

If $(z, z')\in\ZZ_{\degen, n}$ for some $n\in\N$, then we can assume without loss of generality that $z = n$ and $z' > n - 1$. In particular, $z' > z - 1 = n - 1 \geq 0$, thus implying that $z$ and $z'$ are both positive and not negative integers. Moreover $z' + b > z + b - 1 = n + b - 1 \geq 1 - 1/2 - 1 = -1/2$, thus implying that $z+b$ and $z' + b$ are both larger than $-1/2$ and therefore not negative integers.
\end{proof}

In view of Lemmas $\ref{propposdomain}$ and $\ref{UadmsubsetUzero}$, it follows that the numbers $\{P_N(\lambda|z, z', a, b) : \lambda\in\GTp_N\}$ define a probability measure on $\GTp_N$, whenever $(z, z')\in\U_{\adm}$. We call such probability measures the \textit{BC type $z$-measures of level $N$} (associated to the parameters $z, z', a, b$). Let $\Conf_{\fin}(\Zpe)$ be the space of simple (multiplicity free) finite point configurations on $\Zpe$.
Any map $\GTp_N \rightarrow \Conf_{\fin}(\Zpe)$ pushforwards the BC type $z$-measures of level $N$ into point processes on the discrete state space $\Zpe$. We consider two such maps.

The most natural map $\GTp_N \rightarrow \Conf_{\fin}(\Zpe)$ determines a random point configuration with exactly $N$ points. In fact, consider the map
\begin{equation}\label{Omap}
\begin{array}{r@{}l}
\OO : \GTp_N &{}\longrightarrow \Conf_{\fin}(\Zpe)\\
\lambda = (\lambda_1 \geq \cdots \geq\lambda_N) &{}\longrightarrow \OO(\lambda) = \hatl = (\hatl_1 > \cdots > \hatl_N),
\end{array}
\end{equation}
where we do not write $N$ in the definition of $\OO$ for simplicity, and recall that $\hatl$ is one of the $N$-tuples associated to $\lambda$, see Section $\ref{sec:notation}$. Under the map above, the pushforward of the $z$-measure $P_N$ is a probability measure on simple $N$-point configurations in $\Zpe$, therefore determining a point process on $\Zpe$ that we denote by $\OO^{(N)}$.

Next we describe a different map $\GTp_N \rightarrow \Conf_{\fin}(\Zpe)$ that gives rise to a different point process. Recall that $\lambda\in\GTp_N$ can be described by its \textit{Frobenius coordinates}. If we let $d = d(\lambda)$ be the side-length of the largest square that can be inscribed in the Young diagram of $\lambda$, then the Frobenius coordinates $(p_1, \ldots, p_d \ | \ q_1, \ldots, q_d)$ of $\lambda$ are defined by
\begin{equation}\label{frobenius}
\begin{gathered}
p_i(\lambda) = p_i \myeq \lambda_i - i, \hspace{.2in} q_i(\lambda) = q_i \myeq \lambda_i' - i, \hspace{.2in}\forall i = 1, 2, \ldots, d,
\end{gathered}
\end{equation}
where $\lambda_i'$ is the cardinality of the set $\{j \in \{0, 1, \ldots, N\} : \lambda_j \geq i\}$, or equivalently the length of the $i$-th column in the Young diagram of $\lambda$. Observe that $p_1 > \ldots > p_d \geq 0$ and $q_1 > \ldots > q_d \geq 0$. A \textit{balanced point configuration in $\Zp^{\epsilon}$ with respect to $\Z^{\epsilon}_{<N}$} is a subset of $\Zpe$ with an equal number of points in either of the sets
\begin{equation}\label{Zlessgreater}
\begin{gathered}
\Z^{\epsilon}_{<N} \myeq \{\epsilon^2, (1 + \epsilon)^2, \ldots, (N - 1 + \epsilon)^2\},\\
\Z^{\epsilon}_{\geq N} \myeq \{(N + \epsilon)^2, (N + 1 + \epsilon)^2, \ldots\} = \Zp^{\epsilon} \setminus \Z^{\epsilon}_{<N}.
\end{gathered}
\end{equation}
For example, the empty configuration (the one with no points) is balanced (with respect to $\Z^{\epsilon}_{<N}$), and any balanced point configuration in $\Zpe$ has at most $2N$ points. Consider the map
\begin{equation}\label{Lmap}
\begin{array}{r@{}l}
\LL : \GTp_N &{}\longrightarrow \Conf_{\fin}(\Zpe)\\
\lambda = (p_1, \ldots, p_d | q_1, \ldots, q_d) &{}\longrightarrow \LL(\lambda) = \{(N + p_i + \epsilon)^2 \}_{1\leq i\leq d} \sqcup\{(N - 1 - q_i + \epsilon)^2\}_{1\leq i\leq d},
\end{array}
\end{equation}
where, again, we do not write $N$ in the definition of $\LL$ for simplicity. Evidently $\LL(\lambda)$ is a balanced point configuration with $d = d(\lambda)$ points in each of the sets $\Z_{<N}^{\epsilon}$ and $Z_{\geq N}^{\epsilon}$. The pushforward of the $z$-measure $P_N$, under the map above, determines a probability measure on balanced point configurations in $\Zpe$ that we also denote by $\LL^{(N)}$.

The connection between the point configurations $\OO^{(N)}$ and $\LL^{(N)}$, as well as a summary of our definitions, is given in the following statement.

\begin{defprop}\label{relprocesses}
Let $N\in\N$, and consider parameters $a \geq b\geq -1/2$, $(z, z')\in\U_{\adm}$. The probability measure $P_N = P_N(\cdot | z, z', a, b)$ defined in $(\ref{zmeasure})$ is called the \textit{BC type $z$-measure of level $N$} (or simply \textit{$z$-measure of level $N$}) associated to the parameters $z, z', a, b$.

The pushforwards of $P_N$ under the maps $\OO$ and $\LL$, defined in $(\ref{Omap})$ and $(\ref{Lmap})$ respectively, are point processes on $\Zpe$ that we denote by $\OO^{(N)}_{z, z', a, b}$ and $\LL^{(N)}_{z, z', a, b}$, or simply by $\OO^{(N)}$ and $\LL^{(N)}$. The process $\OO^{(N)}$ is an $N$-point configuration, while $\LL^{(N)}$ is a balanced configuration on $\Zpe$ with respect to $\Z^{\epsilon}_{<N}$.

The following relation holds
\begin{equation}\label{symdifference}
\LL(\lambda) = \OO(\lambda) \bigtriangleup \Z_{<N}^{\epsilon} = (\OO(\lambda) \setminus \Z_{<N}^{\epsilon}) \cup (\Z_{<N}^{\epsilon} \setminus \OO(\lambda)),
\end{equation}
for any $\lambda\in\GTp_N$. We denote this relation by
\begin{equation*}
\LL^{(N)} = (\OO^{(N)})^{\bigtriangleup}.
\end{equation*}
\end{defprop}
\begin{proof}
The only statement to prove is $(\ref{symdifference})$, for any $\lambda\in\GTp_N$. If $\lambda$ is the ``empty'' signature $(0^N)$, then evidently $\OO(\lambda) = \{\epsilon^2, (1 + \epsilon)^2, \ldots, (N - 1 + \epsilon)^2\} = \Z_{<N}^{\epsilon}$, $\LL(\lambda) = \emptyset$, and therefore $(\ref{symdifference})$ is evident. Assume otherwise that $\lambda \neq (0^N)$. Then $d = d(\lambda)$ can be characterized as the largest integer in $\{1, 2, \ldots, N\}$ such that $\lambda_d \geq d$, thus implying that $\lambda_{d+1} \leq d$. In this case, we have that the set $\{\lambda_1 + N - 1, \ldots, \lambda_d + N - d\}$ is a subset of $\{N, N+1, \ldots\}$ of size $d$ and also the set $\{\lambda_{d+1} + N - d - 1, \lambda_{d+2} + N - d - 2, \ldots, \lambda_N\}$ is a subset of $\{0, 1, \ldots, N - 1\}$ of size $N - d$. It is easy to check the equalities of sets
\begin{equation*}
\begin{gathered}
\{\lambda_1 + N - 1, \ldots, \lambda_d + N - d\} = \{N + p_1, \ldots, N + p_d\}\\
\{\lambda_{d+1} + N - d - 1, \lambda_{d+2} + N - d - 2, \ldots, \lambda_N\} \sqcup \{N - 1 - q_1, \ldots, N - 1 - q_d\} = \{0, 1, \ldots, N - 1\},
\end{gathered}
\end{equation*}
where $(p_1, \ldots, p_d | q_1, \ldots, q_d)$ are the Frobenius coordinates of $\lambda$. From the definitions of the sets $\OO(\lambda)$ and $\LL(\lambda)$, identity $(\ref{symdifference})$ is now obvious.
\end{proof}

\section{The discrete orthogonal polynomial ensemble $\OO^{(N)}$}\label{sec:OO}

In this section, we prove that the $N$-point process $\OO^{(N)}$ is a discrete orthogonal polynomial ensemble, and therefore a determinantal point process. The main result of this section is Theorem $\ref{thm:Okernel}$, where we compute explicitly a correlation kernel $K^{\OO^{(N)}}: \Zp^{\epsilon}\times\Zpe \longrightarrow \R$ for the point process $\OO^{(N)}$ in terms of generalized hypergeometric functions. As we shall do hereinafter, we consider four parameters $z, z', a, b$ satisfying $a\geq b\geq -1/2$ and $(z, z')\in\U_{\adm}$.

\subsection{The point process $\OO^{(N)}$ is a discrete orthogonal polynomial ensemble}

Define the (two-sided) quadratic lattice
\begin{equation*}
\Z^{\epsilon} \myeq \{\ldots, (-2 + \epsilon)^2, (-1 + \epsilon)^2, \epsilon^2, (1 + \epsilon)^2, (2 + \epsilon)^2, \ldots\}
\end{equation*}
and the function $W(\cdot | z, z', a, b; N) : \Z^{\epsilon} \rightarrow \R_{\geq 0}$ given by 
\begin{equation}\label{weightW}
\begin{gathered}
W(\hatx | z, z', a, b; N) = (x + \epsilon)\times \Gamma\left[ \begin{split}
x + 2\epsilon &,& x + a + 1\\
x + 1 &,& x + b + 1
\vphantom{\frac12}
\end{split}
\right]\\
\times\frac{1}{\Gamma[z - x + N, \ z' - x + N, \ z + x + N + 2\epsilon, \ z' + x + N + 2\epsilon]},
\end{gathered}
\end{equation}
where $\hatx = (x + \epsilon)^2$, for all $x\in\Z$. We may write $W(\hatx)$, instead of $W(\hatx | z, z', a, b; N)$ if the parameters $z, z', a, b$, and the natural number $N$ are implicit. Observe that the image of the function $W$ is included in $\R_{\geq 0}$ by our assumption on the parameters $z, z', a, b$. Due to the factor $\Gamma(x+1)$ in the denominator of $W(\hatx)$, it follows that $W(\widehat{-n}) = W((-n + \epsilon)^2) = 0$ for all $n\in\N$, and therefore we do not lose any information if we restrict $W$ to $\Zpe \subset \Z^{\epsilon}$; thus we can think of $W(\hatx)$ as a weight function on $\Zpe$.

\begin{prop}\label{Ptildeterminantal}
The point process $\OO^{(N)} = \OO^{(N)}_{z, z', a, b}$ is a discrete orthogonal polynomial ensemble in $\Zp^{\epsilon}$, with weight function $W(\hatx | z, z', a, b; N) = W(\hatx)$, i.e.,
\begin{equation}\label{discretepoly}
\mathbb{P}^{\OO^{(N)}}\left(\{\hatx_1 ,\ldots, \hatx_N\} \right) = \textrm{const}\cdot \left|\Delta_N(\hatx_1, \ldots, \hatx_N)\right|^2\cdot\prod_{i=1}^N{W(\hatx_i | z, z', a, b; N)}
\end{equation}
for any set $\{\hatx_1, \ldots, \hatx_N\}\subset\Zpe$ of size $N$, where $\Delta_N(\hatx_1, \ldots, \hatx_N) = \prod_{1\leq i < j\leq N}{(\hatx_i - \hatx_j)}$ and ``const'' is a suitable normalization constant.
\end{prop}
\begin{proof}
Assume without loss of generality that $\hatx_1 > \ldots > \hatx_N$. By the definition of the map $\OO$ in $(\ref{Omap})$, the pushforward probability measure $\mathbb{P}^{\OO^{(N)}}$ is defined as
\begin{eqnarray*}
\mathbb{P}^{\OO^{(N)}}\left(\{\hatx_1, \ldots, \hatx_N\}\right) &=& P_N\left( (x_1 + 1 - N \geq \ldots \geq x_N)\ |\ z, z', a, b; N\right)\\
&=& \frac{P_N'\left( (x_1 + 1 - N \geq \ldots \geq x_N) \ |\ z, z', a, b; N\right)}{S_N(z, z', a, b)}\\
&=& \frac{(\Delta_N(\hatx_1, \ldots, \hatx_N))^2\cdot\prod_{i=1}^N{w(x_i | z, z', a, b; N)}}{S_N(z, z', a, b)}.
\end{eqnarray*}
Thus $\mathbb{P}^{\OO^{(N)}}$ has the right form in $(\ref{discretepoly})$, for $\textrm{const} = S_N(z, z', a, b)^{-1}$ and because $W(\hatx | z, z', a, b; N) = w(x | z, z', a, b; N)$.
\end{proof}

It follows from well known theory of random matrices \cite{Me} that the discrete orthogonal polynomial ensemble $\OO^{(N)}$ is a determinantal point process, whose correlation kernel $K^{\OO^{(N)}}: \Zpe\times\Zp^{\epsilon}\rightarrow\R$ can be expressed in terms of the orthogonal polynomials in $\Zpe$ with respect to the weight function $W(\hatx)$. Furthermore, by general theory of orthogonal polynomials on a quadratic lattice \cite[Ch. 3]{NSU}, the kernel $K^{\OO^{(N)}}$ should admit explicit formulas in terms of generalized hypergeometric functions. Next we review some background on the orthogonal polynomials in $\Zpe$ with respect to the weight function $W(\hatx)$ and then we prove the main result, Theorem $\ref{thm:Okernel}$, in Section $\ref{mainorthogonal}$ below.

\subsection{Generalized hypergeometric functions}\label{sec:hypergeom}

The formulae we obtain below for the correlation kernel $K^{\OO^{(N)}}$ are in terms of the \textit{generalized hypergeometric functions} $_{p+1}F_p$. These are holomorphic functions on the unit disk that depend on a parameter $p\in\N$, and can be defined by the absolutely convergent series
\begin{equation}\label{hypergeomdef}
\pFq{p+1}{p}{a_1 ,, a_2 ,, \ldots ,, a_{p+1}}{b_1 ,, \ldots ,, b_p}{u} \myeq \sum_{k=0}^{\infty}{\frac{(a_1)_k(a_2)_k\cdots(a_{p+1})_k}{(b_1)_k\cdots(b_p)_k}\frac{u^k}{k!}},
\end{equation}
where $(x)_k \myeq x(x+1)\cdots (x+k-1)$ if $k\geq 1$ and $(x)_0 \myeq 1$. The definition $(\ref{hypergeomdef})$ makes sense for $b_1, \ldots, b_p \notin\Z_{\leq 0}$. When $p = 1$, the function $_2F_1$ is simply referred to as the \textit{hypergeometric function}. 

We shall need the generalized hypergeometric functions evaluated at $u = 1$. In this case, the sum $(\ref{hypergeomdef})$ converges absolutely if $\Re(\sum_{i=1}^p{b_i} - \sum_{i=1}^{p+1}{a_i}) > 0$. When some $a_i$ is a negative integer, the sum in $(\ref{hypergeomdef})$ becomes finite and therefore admits an analytic continuation in $u$ to the complex plane. If $p = 3$, $a_1 = -N$ for some $N\in\Zp$, the argument is $u = 1$, and moreover $\sum_{i=1}^p{b_i}-\sum_{i=1}^{p+1}{a_i} = 1$, then the generalized hypergeometric function
\begin{equation*}
\pFq{4}{3}{-N ,, a_2 ,, a_3 ,, a_4}{b_1 ,, b_2 ,, b_3}{1} = \sum_{k=0}^{N}{(-1)^k\frac{(a_2)_k(a_3)_k(a_4)_k}{(b_1)_k(b_2)_k(b_3)_k}{N \choose k}},
\end{equation*}
is called a \textit{terminating Saalsch\"{u}tzian $_4F_3(1)$ series}. For more details and for many identities relating these functions, the reader may consult the classical book \cite{B}; see also the Appendix for some applications of these identities.

\subsection{Wilson-Neretin polynomials}

We review the theory of the family of orthogonal polynomials in $\Zpe$ with respect to the weight function $W(\hatx)$, which was first worked out in \cite{N}. We follow the presentation in \cite[Sec. 8]{BO2} and \cite[Sec. 6.3]{C}.

Take arbitrary complex numbers $a_1, a_2, a_3, a_4$ and $\alpha$, and consider the weight function
\begin{equation}\label{weightneretin}
\widehat{w}(t \ |\ a_1, a_2, a_3, a_4; \alpha) \myeq \frac{t+\alpha}{\prod_{i=1}^4{\Gamma(a_i + \alpha + t)\Gamma(a_i - \alpha - t)}}, \hspace{.2in} t\in\Z.
\end{equation}

\begin{prop}\label{orthpolyinfo}
The polynomials
\begin{equation}\label{orthogonalpolys}
\begin{gathered}
Q_n((t+\alpha)^2) = \Gamma\left[ \begin{split}
2-a_1-a_2+n &,& 2-a_1-a_3+n &,& 2-a_1-a_4+n\\
2-a_1-a_2 &,& 2-a_1-a_3 &,& 2-a_1-a_4
\vphantom{\frac12}
\end{split}\right]\\
\times\pFq{4}{3}{-n,, n+3-a_1-a_2-a_3-a_4 ,, 1-a_1+t+\alpha ,, 1-a_1-t-\alpha}{2-a_1-a_2 ,, 2-a_1-a_3 ,, 2-a_1-a_4}{1}.
\end{gathered}
\end{equation}
are orthogonal with respect to the weight $\widehat{w}(t)$, and their squared norms are
\begin{equation}\label{squarednormH}
\begin{gathered}
H_n = \sum_{t\in\Z}{Q_n^2((t+\alpha)^2)\widehat{w}(t)} = \frac{\sin(2\pi\alpha)\prod_{1\leq i\neq j\leq 4}{\sin(\pi(a_i+a_j))}}{2\pi^6\sin(\pi(a_1+a_2+a_3+a_4))}\\
\times\frac{n!\prod_{1\leq i\neq j\leq 4}{\Gamma(2-a_i-a_j+n)}}{(3-a_1-a_2-a_3-a_4+2n)\Gamma(3-a_1-a_2-a_3-a_4+n)}
\end{gathered}
\end{equation}
\end{prop}

\begin{rem}
The orthogonality relation for the polynomials $Q_n$ and the closed expression for $H_n$ are valid whenever the infinite sums defining them converge.
\end{rem}
\begin{rem}\label{kn}
The polynomials $Q_n$ are not monic. Indeed the leading coefficient of $Q_n$ is
\begin{equation*}
k_n \myeq (n+3-a_1-a_2-a_3-a_4)_n = \frac{\Gamma(2n+3-a_1-a_2-a_3-a_4)}{\Gamma(n+3-a_1-a_2-a_3-a_4)}.
\end{equation*}
\end{rem}

The orthogonal polynomials of Proposition $\ref{orthpolyinfo}$ can be expressed in terms of the well-known \textit{Wilson polynomials}. In the setting of discrete orthogonal polynomials on a quadratic lattice, they were studied by Neretin, see \cite{N}. We call them the \textit{Wilson-Neretin polynomials}.

Before moving on, let us mention that, although the classical Wilson polynomials and the Wilson-Neretin polynomials above admit the same explicit formula, there are two important differences between them. First, the Wilson-Neretin polynomials form a \textit{finite system} of orthogonal polynomials, meaning that only finitely many orthogonal polynomials exist, and the number of them depends on the number of finite moments of the weight function. In contrast, the Wilson polynomials form an infinite system. Second, the Wilson-Neretin polynomials are orthogonal with respect to a weight function on the \textit{discrete space} $\Zpe$, whereas the Wilson polynomials are orthogonal with respect to a weight function on the positive real line.

\subsection{An explicit correlation kernel $K^{\OO^{(N)}}$}\label{mainorthogonal}

For simplicity of notation in the formulas to follow in this section and in the rest of the paper, let us denote
\begin{equation*}
\Sigma \myeq z + z' + b.
\end{equation*}
The following theorem is the main statement of this section.

\begin{thm}\label{thm:Okernel}
A correlation kernel $K^{\OO^{(N)}}$ for the process $\OO^{(N)}$ is given by
\begin{equation}\label{corrKtil}
K^{\OO^{(N)}}(\widehat{x}, \widehat{y}) = \frac{\sqrt{W(\hatx)W(\haty)}}{h_{N-1}}\cdot\frac{\mathfrak{p}_N(\hatx)\mathfrak{p}_{N-1}(\haty) - \mathfrak{p}_{N-1}(\hatx)\mathfrak{p}_N(\haty)}{\hatx - \haty}, \hspace{.2in}\widehat{x}, \widehat{y}\in\Zp^{\epsilon},
\end{equation}
where the function $W : \Zp^{\epsilon} \longrightarrow \R_{\geq 0}$ is defined by the formula in $(\ref{weightW})$, and
\begin{equation}\label{pNfirst}
\begin{gathered}
\mathfrak{p}_N(\widehat{x}) \myeq \Gamma\left[ \begin{split}
N + a + 1 &,& -z+1 &,& -z'+1 &,& 1-N-\Sigma\\
a + 1 &,& -z-N+1 &,& -z'-N+1 &,& 1-\Sigma
\vphantom{\frac12}\end{split}
\right]\\
\times\pFq{4}{3}{-N,, 1-N-\Sigma ,, x + 2\epsilon ,, -x}{1 + a,, -z-N+1 ,, -z'-N+1}{1},
\end{gathered}
\end{equation}
\begin{equation}\label{pNfirst1}
\begin{gathered}
\mathfrak{p}_{N-1}(\hatx) \myeq \Gamma\left[ \begin{split}
N + a &,& -z &,& -z' &,& -N-\Sigma\\
a + 1 &,& -z-N+1 &,& -z'-N+1 &,& -1-\Sigma
\vphantom{\frac12}\end{split}
\right]\\
\times\pFq{4}{3}{-N+1,, -N-\Sigma ,, x + 2\epsilon ,, -x}{1 + a,, -z-N+1 ,, -z'-N+1}{1},
\end{gathered}
\end{equation}
\begin{equation}\label{hN1}
h_{N-1} \myeq \frac{1}{2}\cdot\Gamma\left[ \begin{split}
&& N+a &,& N &,& \Sigma+1 &,& \Sigma+2\\
N+\Sigma+a+1 &,& N+\Sigma+1 &,& 1+z &,& 1+z' &,& 1+z+b &,& 1+z'+b
\vphantom{\frac12}\end{split}
\right].
\end{equation}
We also have
\begin{equation}\label{corrKtil2}
K^{\OO^{(N)}}(\widehat{x}, \widehat{y}) = \frac{\sqrt{W(\hatx)W(\haty)}}{h_{N-1}}\cdot \frac{\widetilde{\mathfrak{p}}_N(\hatx)\mathfrak{p}_{N-1}(\haty) - \mathfrak{p}_{N-1}(\hatx)\widetilde{\mathfrak{p}}_N(\haty)}{\hatx - \haty}, \hspace{.2in}\widehat{x}, \widehat{y}\in\Zp^{\epsilon},
\end{equation}
where
\begin{equation*}
\widetilde{\mathfrak{p}}_N(\widehat{x}) = \Gamma\left[ \begin{split}
x+1 &,& x + N + 2\epsilon\\
x + 2\epsilon &,& x-N+1
\vphantom{\frac12}\end{split}\right]\pFq{4}{3}{-N ,,, -N-a ,,, z ,,, z'}{\Sigma+1 ,,, x-N+1 ,,, -x-N-a-b}{1}.\nonumber
\end{equation*}
\begin{rem}
Like our notation above suggests, $\mathfrak{p}_{N-1}(\hatx), \mathfrak{p}_N(\hatx)$ are the $(N-1)$-st and $N$-th orthogonal polynomials in the quadratic lattice $\Zpe$ with respect to the weight $W(\hatx)$, at least when $\Re\Sigma > 0$. Also $h_{N-1}$ is the squared norm of $\mathfrak{p}_{N-1}$, namely $h_{N-1} = \sum_{\hatx\in\Zpe}{\mathfrak{p}_{N-1}(\hatx)^2W(\hatx)}$. We obtain the more general orthogonal polynomials $\mathfrak{p}_n$, $n = 0, 1, \ldots, N$, and more general squared norms $h_n$, $n = 0, 1, \ldots, N-1$, both in terms of terminating Saalsch\"{u}tzian $_4F_3(1)$ series, in the proof below.
\end{rem}
\begin{rem}
The expression $(\ref{pNfirst})$ for $p_N(\widehat{x})$ is not well-defined if $\Sigma = 0$ because of the factor $\frac{\Gamma(1 - N - \Sigma)}{\Gamma(1 - \Sigma)} = \frac{1}{(1 - N - \Sigma)(2 - N - \Sigma)\cdots(-\Sigma)}$. Therefore the expression for the kernel $(\ref{corrKtil})$ is not well-defined in this case, and the formula in $(\ref{corrKtil2})$ provides an analytic continuation of $K^{\OO^{(N)}}$.
\end{rem}
\begin{rem}\label{analiticityfns}
For any $x\in\Zp$, the functions $\mathfrak{p}_{N-1}(\hatx), \widetilde{\mathfrak{p}}_N(\hatx)$ are analytic on the variables $z, z'$ on the domain $(z, z')\in\U$, while $\mathfrak{p}_N(\hatx)$ is analytic on the subdomain $(z, z')\in \U\setminus\{\Sigma = 0\}$. Finally, $h_{N-1}$ is analytic on $(z, z')\in\U$, but $h_{N-1}\neq 0$ only in the subdomain $\U_0$.
\end{rem}
\end{thm}
\begin{proof}
The proof is divided into several steps.\\

\textbf{Step 1.}
Due to general theory of random matrices, see \cite[Sec. 5.2]{Me}, the discrete orthogonal polynomial ensemble $\OO^{(N)}$ has kernel $K^{\OO^{(N)}}$ given by the \textit{normalized Christoffel-Darboux kernel}
\begin{equation}\label{CDkernel}
K^{\OO^{(N)}}(\hatx, \haty) = \sqrt{W(\hatx)W(\haty)}\sum_{m=0}^{N-1}{\frac{\p_m(\hatx)\p_m(\haty)}{h_m}},
\end{equation}
where $\p_0 = 1, \p_1, \ldots, \p_{N-1}$ are the orthogonal polynomials in $\Zpe$ with respect to the weight function $W(\hatx)$ (they come from the Gram-Schmidt orthogonalization in the Hilbert space $L^2(\Zp^{\epsilon}, \sum_{\hatx\in\Zp^{\epsilon}} W(\hatx)\delta_{\hatx})$, of the basis $1, \hatx, \hatx^2, \ldots, \hatx^{N-1}$) and $h_m \equiv \sum_{\widehat{t}\in\Zpe}{\left(\p_m(\widehat{t})\right)^2W(\widehat{t})}$. If the $N$-th orthogonal polynomial $\p_N$ exists, then $(\ref{CDkernel})$ can be rewritten as
\begin{equation}\label{CDkernel2}
\frac{\sqrt{W(\hatx)W(\haty)}}{h_{N-1}}\cdot\frac{\p_N(\hatx)\p_{N-1}(\haty) - \p_{N-1}(\hatx)\p_N(\haty)}{\hatx - \haty}.
\end{equation}

In this first step, we find out when these orthogonal polynomials $\p_m$ exist, given the weight $W(\hatx)$. We also find a condition which guarantees $h_{N-1} < \infty$.

From the Gram-Schmidt orthogonalization process, it is easy to see that the $m$-th polynomial $\p_m$ exists if the weight function $W(\hatx)$ has finite $(2m - 1)$ moments, i.e., if $\sum_{\hatx\in\Zpe}{\widehat{x}^{2m - 1}W(\hatx)} < \infty$. Thanks to the identity $\Gamma(z)\Gamma(1 - z) = \frac{\pi}{\sin(\pi z)}$, we can rewrite $W(\hatx)$ as
\begin{equation}\label{Wrewritten}
\begin{gathered}
W(\hatx) = (x + \epsilon)\times\frac{\Gamma(x + 2\epsilon)\Gamma(x + a + 1)}{\Gamma(x + b  +1)\Gamma(x + 1)}\times\frac{\sin(\pi z)\sin(\pi z')}{\pi^2}\\
\times\frac{\Gamma(x + 1 - z - N)\Gamma(x + 1 - z' - N)}{\Gamma(x + z + N + 2\epsilon)\Gamma(x + z' + N + 2\epsilon)}, \ x\in\Zp.
\end{gathered}
\end{equation}
By well-known asymptotics of the Gamma function, we have $W(\hatx)\sim x^{1-2(\Sigma + 2N)}$, as $x\rightarrow\infty$. Therefore $\sum_{\hatx\in\Zp^{\epsilon}}{\widehat{x}^{2m - 1}W(\hatx)} < \infty$ iff $\sum_{x\in\Zp}{x^{2(2m - 1) + 1 - 2(\Sigma + 2N)}} = \sum_{x\in\Zp}{x^{-1 + 4m -2\Sigma -4N}} < \infty$ iff $4m - 2\Re\Sigma - 4N < 0$ iff $\Re\Sigma > 2m - 2N$. In particular, $\p_{N-1}$ exists if $\Re\Sigma > -2$ and $\p_N$ exists if $\Re\Sigma > 0$.

Finally, the squared norm $h_m = (\p_m, \p_m)_{L^2(\Zp^{\epsilon}, W\cdot\mu)}$ is finite if the weight $W(\hatx)$ has a finite $(2m)$-th moment, i.e., $\sum_{\hatx\in\Zp^{\epsilon}}{\hatx^{2m}W(\hatx)} < \infty$. By a similar reasoning as above, $h_m < \infty$ if $\Re\Sigma > 2m - 2N + 1$; in particular, $h_{N-1} < \infty$ whenever $\Re\Sigma > -1$.\\

\textbf{Step 2.} In this step, we prove that equation $(\ref{corrKtil})$ is a correlation kernel $K^{\OO^{(N)}}$ for the determinantal process $\OO^{(N)}$, under the assumption $\Re\Sigma > 0$. The idea is to compute the orthogonal polynomials $\p_n(\hatx)$ in $\Zpe$ with respect to the weight $W(\hatx)$, compute the squared norms $h_n$, and then plug them into equation $(\ref{CDkernel2})$ to find a correlation kernel. Our main tool is Proposition $\ref{orthpolyinfo}$. Under the identification of variables
\begin{equation*}
\begin{gathered}
t = x, \hspace{.1in} \alpha = \epsilon, \hspace{.1in} (t+\alpha)^2 = \hatx,\\
a_1 = 1 - \epsilon, \hspace{.1in} a_2 = b+1-\epsilon,\hspace{.1in} a_3 = z+N+\epsilon,\hspace{.1in} a_4 = z'+N+\epsilon,
\end{gathered}
\end{equation*}
the following equality is clear
\begin{equation}\label{ftow}
W(\hatx) = \widehat{w}(x)\cdot\frac{\pi^2}{\sin(\pi a)\sin\pi(a+b)}, \textrm{ for all } x\in\Z,
\end{equation}
where for simplicity we wrote $\widehat{w}(x)$ instead of $\widehat{w}(x \ |\ 1 - \epsilon, b + 1 - \epsilon, z + N + \epsilon, z' + N + \epsilon; \epsilon)$, as defined in $(\ref{weightneretin})$. Temporarily we can assume that neither $a$ nor $a + b$ are integers, in order for equation $(\ref{ftow})$ to make sense.
 
From Step 1, since we are under the assumption $\Re\Sigma > 0$, all the monic orthogonal polynomials $\p_0 = 1, \p_1, \ldots, \p_{N-1}, \p_N$ with respect to $W(\hatx)$ exist. Thanks to Proposition $\ref{orthpolyinfo}$, we obtain the equalities
\begin{equation}\label{ptoq}
\p_n(\hatx) = \frac{Q_n((x+\epsilon)^2)}{k_n},
\end{equation}
where the polynomials $Q_n((x + \epsilon)^2)$ is given $(\ref{orthogonalpolys})$ and the leading coefficients $k_n$ are given in Remark $\ref{kn}$, respectively. The resulting explicit formula is

\begin{equation}\label{orthpol}
\begin{gathered}
\p_n(\hatx) = \Gamma\left[ \begin{split}
1+a+n &,& 1-z-N+n &,& 1-z'-N+n &,& n+1-2N-\Sigma\\
1+a &,& 1-z-N &,& 1-z'-N &,& 2n+1-2N-\Sigma
\vphantom{\frac12}
\end{split}\right]\\
\times\pFq{4}{3}{-n,, n+1-2N-\Sigma ,, x+2\epsilon ,, -x}{1+a ,, 1-z-N ,, 1-z'-N}{1} \ \forall n \leq N.
\end{gathered}
\end{equation}
This gives the desired formulas $(\ref{pNfirst})$ and $(\ref{pNfirst1})$ when $n = N-1$ and $n = N$, respectively.

From step $1$, $h_n < \infty$ for all $n\leq N-1$ because $\Re\Sigma > -1$.
To obtain an expression for the squared norm $h_n$, use the relations $(\ref{ftow}), (\ref{ptoq})$:
\begin{eqnarray*}
h_n = \sum_{\hatx\in\Zpe}{\left(\p_n(\hatx)\right)^2W(\hatx)} &=& \frac{\pi^2}{k_n^2\sin(\pi a)\sin\pi(a+b)}\sum_{x\in\Zp}{Q_n^2({(x+\epsilon)^2})\widehat{w}(x)}\\
&=& \frac{\pi^2 H_n}{k_n^2\sin(\pi a)\sin\pi(a+b)}.
\end{eqnarray*}
We point out that $W(\hatx) = \widehat{w}(x) = 0$ for $x\in\{-1, -2, \ldots\}$, reason why $H_n$ appears as a semi-infinite sum in this setting and as a two-sided infinite sum in the setting of Proposition $\ref{orthpolyinfo}$. Plugging in the expression for $H_n$ in Proposition $\ref{orthpolyinfo}$ (under the identification of parameters given above), and applying the identity $\Gamma(z)\Gamma(1-z) = \frac{\pi}{\sin(\pi z)}$ several times, we arrive at
\begin{equation}\label{hnorm}
\begin{gathered}
h_n = \frac{1}{2}\cdot\Gamma\left[ \begin{split}
n+1 &,& n+a+1 &,& 2N+\Sigma-2n &,& 2N+\Sigma-2n-1\\
&&2N+\Sigma-n &,& 2N+\Sigma+a-n
\vphantom{\frac12}
\end{split}\right]\\
\times\frac{1}{\Gamma(z+N-n)\Gamma(z'+N-n)\Gamma(z+b+N-n)\Gamma(z'+b+N-n)} \ \forall n\leq N-1.
\end{gathered}
\end{equation}
When $n = N-1$, $(\ref{hnorm})$ becomes the formula for $H_{N-1}$ given in the proposition.

We need to remove the condition that $a, a+b\notin\Z$. All we need to justify is that the orthogonal polynomials $\p_n(\hatx)$ on $\Zpe$ with respect to the weight $W(\hatx)$ are still given by $(\ref{orthpol})$ and their squared norms $h_n$ are still given by $(\ref{hnorm})$ in the case that $a\in\Z$ or $a+b\in\Z$. First of all observe that, for any $\hatx\in\Zpe$, then $W(\hatx)$ is holomorphic in $\{\Re a, \Re b > -1\}$, as a function of $a, b$. The orthogonal polynomials $\p_n(\hatx)$, $n = 0, 1, \ldots$, are determined recursively so that $\p_n(\hatx)$ is the unique monic polynomial in $\hatx$ that is monic, of degree $n$, and the following equations are satisfied
\begin{equation*}
\sum_{\hatx\in\Zpe}{\p_n(\hatx)\p_m(\hatx)W(\hatx)} = 0, \ \forall m = 0, 1, \ldots, n-1.
\end{equation*}
The arguments above show that all equations above are satisfied for parameters $a, b$ in the complex domain $\{\Re\Sigma > 0\}\cap\{\Re a, \Re b > -1\}\cap\{a\notin\Z\}\cap\{a+b\notin\Z\}$. Since the right-hand side of $(\ref{orthpol})$ is analytic, for all $n = 0, 1, \ldots, N-1$, in the larger domain $\{\Re\Sigma > 0\}\cap\{\Re a, \Re b > -1\}$, then the identities above extend for any $a, b$ in the larger domain and so $(\ref{orthpol})$ is the $n$-th orthogonal polynomial with respect to the weight function $W(\hatx)$ even when $a\in\Z$ or $a+b\in\Z$.

By a similar reasoning, we can show that the squared norms $h_n$ are given by $(\ref{hnorm})$ even when $a\in\Z$ or $a+b\in\Z$. The key is to treat $a$ and $b$ as complex parameters and use the equation $\sum_{\hatx\in\Zpe}{\p_n(\hatx)^2W(\hatx)} = h_n$ to do an analytic continuation.\\

\textbf{Step 3.}
In this step, we prove Theorem $\ref{thm:Okernel}$, under the assumption that $\Sigma \neq 0$ (and $\Re\Sigma > -1$). Observe that step 1 shows that the polynomials $p_N(\hatx)$, $0\leq n\leq N-1$, are well-defined. Moreover the squared norms $h_n$, $0\leq n\leq N - 1$ are finite. Thus $(\ref{CDkernel})$ implies that a correlation kernel $K^{\OO^{(N)}}$ is given by
\begin{equation}\label{corkernel1}
K^{\OO^{(N)}}(\hatx, \haty) = \sqrt{W(\hatx)W(\haty)}\cdot\sum_{n=0}^{N-1}{\frac{\p_n(\hatx)\p_n(\haty)}{h_n}}.
\end{equation}
The goal is to have a simplified expression for $(\ref{corkernel1})$, rather than leaving it as a sum of $N$ generalized hypergeometric functions.

In step 1, see $(\ref{CDkernel2})$, we mentioned that if the polynomial $\p_n(\hatx)$ exists, then we have the equality
\begin{equation}\label{cokernel2}
\sum_{n=0}^{N-1}{\frac{\p_n(\hatx)\p_n(\haty)}{h_n}} = \frac{1}{h_{N-1}}\cdot\frac{\p_N(\hatx)\p_{N-1}(\haty) - \p_N(\haty)\p_{N-1}(\hatx)}{\hatx - \haty}.
\end{equation}
In step 1, we also showed that $\p_N$ exists in the domain $\{\Re\Sigma > 0\}$. Thus $(\ref{cokernel2})$ holds if $\Re\Sigma > 0$. In step 2, we moreover found an explicit formula for $\p_N(\hatx)$, see $(\ref{orthpol})$, which shows that $\p_N(\hatx)$ admits an analytic continuation on the domain $(z, z')\in\U_0 \setminus \{\Sigma = 0\}$. Thus, by analytic continuation, $(\ref{cokernel2})$ is true if we think that $\p_N(\hatx)$, $\p_N(\haty)$ are the explicit expressions in the right-hand side of $(\ref{orthpol})$. Therefore we have finished the proof of Theorem $\ref{thm:Okernel}$ under the assumption that $\Sigma \neq 0$.\\

\textbf{Step 4.}
Finally we prove $(\ref{corrKtil2})$, which gives an analytic continuation of Theorem $\ref{thm:Okernel}$ and removes the singularities in $\Sigma = 0$. By use of well known identities involving generalized hypergeometric functions, the expression for $\p_N(\hatx)$ in $(\ref{pNfirst})$ is equal to
\begin{equation}\label{almostpN}
\mathfrak{p}_N(\hatx) = \Gamma\left[ \begin{split}
x + N + 2\epsilon &,& x+1\\
x + 2\epsilon &,& x-N+1
\vphantom{\frac12}
\end{split}
\right]\pFq{4}{3}{-N ,,, -N-a ,,, z ,,, z'}{\Sigma ,,, x-N+1 ,,, -x-N-a-b}{1}.
\end{equation}
Similarly, we have
\begin{equation}\label{almostpN1}
\mathfrak{p}_{N-1}(\hatx) = \Gamma\left[ \begin{split}
x + N + a+b &,& x+1\\
x + 2\epsilon &,& x-N+2
\vphantom{\frac12}
\end{split}
\right]\pFq{4}{3}{-N + 1 ,,, -N-a+1 ,,, z+1 ,,, z'+1}{\Sigma + 2 ,,, x-N+2 ,,, -x-N-a-b+1}{1}.
\end{equation}
Details are spelled out in the Appendix.

By analytic continuation, we are done if we showed the identity
\begin{equation}\label{step5eq}
\mathfrak{p}_N(\hatx)\mathfrak{p}_{N-1}(\haty) - \mathfrak{p}_{N-1}(\hatx)\mathfrak{p}_N(\haty)  = \widetilde{\mathfrak{p}}_N(\hatx)\mathfrak{p}_{N-1}(\haty)  - \mathfrak{p}_{N-1}(\hatx)\widetilde{\mathfrak{p}}_N(\haty) 
\end{equation}
for any $x, y\in\Zp$ and points $(z, z')$ where all the functions $\p_N, \p_{N-1}, \widetilde{\p}_N$ are defined, that is in the complex domain $(z, z')\in\U_0\setminus\{\Sigma = 0\}$. For this purpose, we use the following identity, which is an easy consequence of the series representation for $_4F_3$:
\begin{equation}\label{trivial4F3}
\begin{gathered}
\pFq{4}{3}{a_1 ,,, a_2 ,,, a_3 ,,, a_4}{b_1 ,,, b_2 ,,, b_3}{1} = \pFq{4}{3}{a_1 ,,, a_2 ,,, a_3 ,,, a_4}{b_1+1 ,,, b_2 ,,, b_3}{1}\\
+ \frac{a_1a_2a_3a_4}{b_1(b_1+1)b_2b_3}\pFq{4}{3}{a_1+1 ,,, a_2+1 ,,, a_3+1 ,,, a_4+1}{b_1+2 ,,, b_2+1 ,,, b_3+1}{1}.
\end{gathered}
\end{equation}
The identity above holds as an equality of analytic functions whenever $\Re(\sum_i{b_i} - \sum_i{a_i}) > 0$ and $b_1, b_2, b_3 \notin \Z_{\leq 0}$. In particular, for the choice of parameters $a_1 = -N$, $a_2 = -N-a$, $a_3 = z$, $a_4 = z'$, $b_1 = \Sigma$, $b_2 = x - N + 1$, $b_3 = -x-N-a-b$, and taking equations $(\ref{almostpN})$, $(\ref{almostpN1})$, into account, identity $(\ref{trivial4F3})$ becomes
\begin{equation*}
\widetilde{\mathfrak{p}_N}(\hatx) = \mathfrak{p}_N(\hatx) + \frac{N(N+a)zz'}{\Sigma(\Sigma+1)}\cdot{\mathfrak{p}}_{N-1}(\hatx).
\end{equation*}
The latter equation immediately implies the desired equation $(\ref{step5eq})$, and we are done.
\end{proof}

From the proof above, we obtain the following corollary which will be useful below in Section $\ref{kernelL}$.

\begin{cor}\label{cor:analytic}
If $\Re\Sigma > -1$, then the orthogonal polynomials $\mathfrak{p}_0 = 1, \mathfrak{p}_1, \ldots, \mathfrak{p}_{N-1}$ with respect to the weight function $W:\Zpe\rightarrow\R_{\geq 0}$ in $(\ref{weightW})$ exist. Also, the squared norms $h_0, h_1, \ldots, h_{N-1}$ are finite. Moreover $1/h_0, \ldots, 1/h_{N-1}$, as functions of $(z, z')$, are holomorphic on $\U_0$. For any fixed $\hatx\in\Zpe$, all the functions $\mathfrak{p}_0(\hatx), \ldots, \mathfrak{p}_{N-1}(\hatx)$ are holomorphic on $\U$. Let $(z, z')\in\C^2$ be such that $z+z' > -\frac{1}{2}$, $\{z, z'\}\cap\{\ldots, -2, -\frac{3}{2}, -1, -\frac{1}{2}\} = \emptyset$, and $\hatx\in\Zpe$ be fixed, then $\mathfrak{p}_0(\hatx), \ldots, \mathfrak{p}_{N-1}(\hatx)$ and $1/h_0, \ldots, 1/h_{N-1}$ are all continuous in the real variables $a\in (-1, \infty)$ and $b$ in a real neighborhood of $-\frac{1}{2}$.
\end{cor}
\begin{proof}
The first two statements follow from Step 1 in the proof above. The latter three statements follow from $(\ref{orthpol})$ and $(\ref{hnorm})$. The real neighborhood of $-\frac{1}{2}$ in the last statement can be taken to be the connected component of $\R\setminus \left((\Z - z) \cup (\Z - z')\right)$ to which $-\frac{1}{2}$ belongs.
\end{proof}

\section{The $L$-ensemble $\LL^{(N)}$}\label{sec:Lensemble}

In this section we recall the definition of $L$-ensembles, a special kind of determinantal point processes. We also prove that $\LL^{(N)}$ is an $L$-ensemble. The calculation of the correlation kernel $K^{\LL^{(N)}}$ is delayed until Section $\ref{kernelL}$ below, after we discuss the discrete Riemann-Hilbert problem.

\subsection{$L$-ensembles}\label{Lsubsection}

Let $\X$ be a countable space, $\Conf_{\fin}(\X)$ be the space of simple, finite point configurations on $\X$.
Let  $L$ be an operator on the Hilbert space $\ell^2(\X)$ and let us denote its matrix $L = [L(x, y)]_{x, y\in\X}$ by the same letter. For each $X\in\Conf_{\fin}(\X)$, let $L_X$ be the submatrix of size $|X|\times|X|$ of $L$ consisting of the rows and columns indexed by elements of $X$. The determinants $\det L_X$ are the finite principal minors of $L$. We assume $L$ satisfies:

\begin{enumerate}[label=(\alph*)]
	\item $L$ is a trace class operator.
	\item All finite principal minors $\det L_X$ are nonnegative.
\end{enumerate}

We agree that $\det L_{\emptyset} = 1$. The previous assumptions imply
\begin{equation*}
0 < \sum_{X\in\textrm{Conf}_{\fin}(\X)}{\det L_X} = \det(1 + L) < \infty.
\end{equation*}
We can then associate a finite, multiplicity free point process to $L$ given by
\begin{equation*}
\textrm{Prob}(X) = \frac{\det L_X}{\det(1 + L)}, \hspace{.1in}X\in\textrm{Conf}_{\fin}(\X).
\end{equation*}
Such processes are called {\it $L$-ensembles}. The matrix $L$ is said to be the \textit{$L$-matrix} of the ensemble. It is known that $L$-ensembles are determinantal processes. The following proposition is proved in \cite[Prop. 2.1]{BO3}, see also \cite{BR} for a linear-algebraic proof in the case that $\X$ is finite.

\begin{prop}\label{kernelLensemble}
Let $L$ be an operator on $\ell^2(\X)$ that satisfies (a) and (b) above. Then the $L$-ensemble defined above is a determinantal point process, whose correlation kernel is the matrix of the operator $K = L(1 + L)^{-1}$.
\end{prop}

\subsection{Relation between orthogonal polynomial ensembles and $L$-ensembles}\label{orthLsubsection}
Fix $N\in\N$ and a locally finite (therefore countable) set $\X\subset\R$ with a splitting $\X = \X_I \sqcup \X_{II}$ and where $\X_{II}$ is finite and of size $|\X_{II}| = N$. Let $h_I: \X_I \rightarrow \R$, $h_{II}: \X_{II} \rightarrow \R$ be two maps satisfying 
\begin{enumerate}[label=(\roman*)]
	\item $h_I(x) \geq 0$ for all $x\in\X_{I}$ and $\sum_{x\in\X_I}{\frac{(h_I(x))^2}{1 + x^2}} < \infty$.
	\item $h_{II}(x) > 0$ for all $x\in\X_{II}$.
\end{enumerate}
To the pair of functions $(h_I, h_{II})$, we associate the function $f: \X \rightarrow \R_{\geq 0}$ by
\begin{equation}\label{funcf}
    f(x) =
	\begin{cases}
                  \frac{(h_I(x))^2}{\prod_{y\in\X_{II}}{(x - y)^2}} &\textrm{if } x\in\X_{I},\\
                  \frac{1}{(h_{II}(x))^2\cdot\prod_{\substack{y\in\X_{II}\\ y\neq x}}{(x - y)^2}} &\textrm{if } x\in\X_{II}.
          \end{cases}
\end{equation}

Next we associate a discrete orthogonal polynomial ensemble $\PP_{\ort}$ to the function $f$ and an $L$-ensemble $\PP_L$ to the pairs of functions $(h_I, h_{II})$. Later we make a connection between these point processes.\\

We denote by $\PP_{\ort}$ the $N$-point orthogonal polynomial ensemble on $\X$ with weight function $f:\X\rightarrow \R_{\geq 0}$. Let us check that $f$ indeed defines such an ensemble.

In fact, observe that $f$ is nonnegative and is strictly positive for all points in the set $\X_{II}$ of size $|\X_{II}| = N$ (and therefore there is at least one set in the ensemble with positive probability). We still need to check that the weight $f(x)$ has a finite $(2N - 2)$-nd moment moment. In fact, $\sum_{x\in\X}{x^{2N-2}f(x)} < \infty$ if and only if $\sum_{x\in\X_I}{x^{2N-2}f(x)} < \infty$ because $\X_{II}$ is finite, and by the definition $(\ref{funcf})$, the latter inequality is equivalent to $\sum_{x\in\X_I}{\frac{(h_I(x))^2}{1+x^2}} < \infty$, which is true by the assumption (i).\\

We now construct an $L$-ensemble $\PP_{L}$ associated to $(h_I, h_{II})$. Observe that the splitting $\X = \X_I \sqcup \X_{II}$ allows us to write the Hilbert space decomposition
\begin{equation*}
\ell^2(\X) = \ell^2(\X_I) \oplus \ell^2(\X_{II}).
\end{equation*}
We can then define an operator $L$ on $\ell^2(\X)$ by giving its matrix in the form

\begin{center}
$L = \begin{bmatrix}
    L_{I, I} & L_{I, II}\\
    L_{II, I} & L_{II, II}
\end{bmatrix},$
\end{center}
where $L_{I, I}$ is the matrix of an operator from $\ell^2(\X_I)$ to $\ell^2(\X_I)$, $L_{I, II}$ is the matrix of an operator from $\ell^2(\X_{II})$ to $\ell^2(\X_I)$, etc. In this notation, let us consider the operator $L$ whose matrix is given by
\begin{equation}\label{generalLmatrixdef}
L = \begin{bmatrix}
    0 & A\\
    -A^* & 0
\end{bmatrix}, \ A(x, y) = \frac{h_I(x)h_{II}(y)}{x - y}, \ x\in\X_I, \ y\in\X_{II},
\end{equation}
where $A^*$ is the transpose of $A$. By condition (i), the columns in $A$ are vectors in $\ell^2(\X_I)$; together with the fact that $A$ has finitely many columns, the operator $A: \ell^2(\X_{II}) \longrightarrow \ell^2(\X_I)$ is of trace class. It then follows that $L$ is of trace class. Since the matrix $L$ is skew-Hermitian, then its finite principal minors are the determinants of skew-Hermitian matrices and therefore nonnegative. Hence we can define the point process $\PP_L$ from the operator $L$ by the general setup of Section $\ref{Lsubsection}$ above.

The following proposition that relates the processes $\PP_{\ort}$ and $\PP_L$ follows from a simple computation.

\begin{prop}[\cite{BO1}, Prop. 5.7]\label{complementation}
Let $(\PP_{\ort})^{\Delta}$ be the simple, finite point process in $\X$ such that the probability of $X\in\Conf_{\fin}(\X)$ is the probability of $X\Delta\X_{II}$ with respect to the the point process $\PP_{\ort}$, i.e.,
\begin{equation*}
\mathbb{P}^{\PP_{\ort}^{\Delta}}\left( X \right) = \mathbb{P}^{\PP_{\ort}}\left( X \Delta \X_{II} \right).
\end{equation*}
Then $(\PP_{\ort})^{\Delta}$ and $\PP_L$ have the same distribution.
\end{prop}

\subsection{The point process $\LL^{(N)}$ is an $L$-ensemble}\label{sec:specificL}

Fix $N\in\N$ and $(z, z')\in\U_{\adm}$. Specialize the general setting of the previous section as follows
\begin{equation*}
\begin{gathered}
\X = \Zpe,\hspace{.1in} \X_I = \Z^{\epsilon}_{\geq N},\hspace{.1in} \X_{II} = \Z^{\epsilon}_{<N},\\
f(\hatx) = W(\hatx | z, z', a, b; N) = W(\hatx).
\end{gathered}
\end{equation*}
The general setting then pushes us to define certain functions $h_I: \X_I \rightarrow \R_{\geq 0}$, $h_{II}: \X_{II} \rightarrow \R_{>0}$. Let us change the notation slightly and define functions $\psi_{\geq N}: \Z^{\epsilon}_{\geq N} \rightarrow \R_{\geq 0}$, $\psi_{< N} : \Z_{<N}^{\epsilon} \rightarrow \R_{>0}$, that depend on parameters $z, z', a, b$ and the positive integer $N$, as follows
\begin{equation*}
\begin{gathered}
\psi_{\geq N}(\hatx) = W(\hatx)\cdot\prod_{\haty \in \Z^{\epsilon}_{<N}}{(\hatx - \haty)^2} \ \forall \ \hatx\in\Z^{\epsilon}_{\geq N},\\
\psi_{< N}(\hatx) = \frac{1}{W(\hatx)\cdot\prod_{\substack{\haty \in \Z^{\epsilon}_{<N} \\ y\neq x}}{(\hatx - \haty)^2}} \ \forall \ \hatx\in\Z^{\epsilon}_{<N}.
\end{gathered}
\end{equation*}
More explicitly,
\begin{equation}\label{eqn:psigreaterN}
\begin{gathered}
\psi_{\geq N}(\hatx) = (x + \epsilon)\cdot\frac{\Gamma(x + N + 2\epsilon)^2}{\Gamma(x - N + 1)^2}\cdot\Gamma\left[ \begin{split}
x + 1 &,& x + a + 1\\
x + 2\epsilon &,& x + b + 1
\vphantom{\frac12}
\end{split}
\right]\\
\times\frac{1}{\Gamma\left[ -x + z + N, \ x + z + N + 2\epsilon, \ -x + z' + N, \ x + z' + N + 2\epsilon \right]} \ \forall \ \hatx\in\Z^{\epsilon}_{\geq N},
\end{gathered}
\end{equation}
\begin{equation}\label{eqn:psismallerN}
\begin{gathered}
\psi_{< N}(\hatx) = 4(x + \epsilon)\cdot\frac{1}{\Gamma(N - x)^2\Gamma(x + N + 2\epsilon)^2}\cdot\Gamma\left[ \begin{split}
x + 2\epsilon &,& x + b + 1\\
x + 1 &,& x + a + 1
\vphantom{\frac12}
\end{split}
\right]\\
\times\Gamma\left[-x + z + N, \ x + z + N + 2\epsilon, \ -x + z' + N, \ x + z' + N + 2\epsilon\right] \ \forall \ \hatx\in\Z^{\epsilon}_{<N}.
\end{gathered}
\end{equation}

In the special case here, the functions $h_I, h_{II}$ from the general setup are
\begin{equation*}
h_I = \sqrt{\psi_{\geq N}}, \ h_{II} = \sqrt{\psi_{<N}}.
\end{equation*}

Next, by following the general setting, we can define a discrete orthogonal polynomial ensemble $\PP_{\ort}$ from the weight function $W(\hatx)$ and an $L$-ensemble $\PP_L$ from the pair of functions $(\psi_{\geq N}, \psi_{<N})$. It is clear that $\PP_{\ort}$ is the process $\OO^{(N)}$ studied in Section $\ref{sec:OO}$. Moreover, thanks to the Propositions $\ref{relprocesses}$ and $\ref{complementation}$, we deduce that $\PP_L$ is the $L$-ensemble $\LL^{(N)}$ which is ``complementary'' to $\OO^{(N)}$ with respect to $\Z^{\epsilon}_{<N}$. According to the general setup, we can find the $L$-matrix that corresponds to the $L$-ensemble $\LL^{(N)}$. With respect to the splitting $\Zpe = \Z^{\epsilon}_{\geq N}\sqcup\Z^{\epsilon}_{< N}$, the $L$-matrix of $\LL^{(N)}$ is given by, cf. $(\ref{generalLmatrixdef})$,
\begin{equation}\label{Lmatrixdef}
L_N = \begin{bmatrix}
    0 & A_N\\
    -\left(A_N\right)^* & 0
\end{bmatrix},\textrm{ where } A_N(\hatx, \haty) = \frac{\sqrt{\psi_{\geq N}(\hatx)\psi_{<N}(\haty)}}{\hatx - \haty}, \ \hatx\in\Z^{\epsilon}_{\geq N},\ \haty\in\Z^{\epsilon}_{<N}.
\end{equation}

This section is summarized in the following statement.

\begin{prop}\label{LNensemble}
The point process $\LL^{(N)}$ is an $L$-ensemble whose associated $L$-matrix is $L_N$, defined in $(\ref{Lmatrixdef})$ above. Then a correlation kernel for $\LL^{(N)}$ is given by $K^{\LL^{(N)}} = L_N(1+L_N)^{-1}$.
\end{prop}
\begin{proof}
The first statement follows from the discussion above. The second follows from Proposition $\ref{kernelLensemble}$.
\end{proof}

\section{Discrete Riemann-Hilbert Problem and connection to $L$-ensembles}\label{sec:DRHP}

In this section, we follow the exposition in \cite{B2} faithfully, though the treatment in that paper is more general than ours. Our goal is to lay out a plan, in a more general scenario, on how to compute the entries of the matrix of $K = L_N(1 + L_N)^{-1}$, where $L_N$ is defined in $(\ref{Lmatrixdef})$. We carry out the plan and compute explicit analytic formulas for the entries of $K$ in the next section.

\subsection{Discrete Riemann-Hilbert problem}\label{DRHP}

Let $\SSS\subset\C$ be a subset of $\C$ for which there exists $\epsilon > 0$ such that $|x - y| > \epsilon \ \forall x\neq y$ in $\SSS$, in particular $\SSS$ is locally finite, and the operator
\begin{equation}\label{Top}
(Th)(x) = \sum_{\substack{y\in\SSS\\ y\neq x}}{\frac{h(y)}{x - y}}
\end{equation}
is well-defined on $\ell^2(\SSS)$ and is a bounded operator.

An operator $L$ on $\ell^2(\mathcal{S})$ is said to be \textit{integrable} if there exist an integer $R\in\{2, 3, \ldots\}$, functions $f_1, g_1, \ldots, f_R, g_R\in\ell^2(\mathcal{S})$ such that $\sum_{i=1}^R{f_i(x)g_i(x)} = 0 \ \forall x\in\SSS$, and the matrix of $L$ has the form
\begin{equation}\label{integrableop}
L(x, y) = \begin{cases}
                  \frac{\sum_{i=1}^R{f_i(x)g_i(y)}}{x - y} & \textrm{if } x\neq y\\
                  0 & \textrm{if } x = y.
          \end{cases}
\end{equation}

It is implied that $L$ is a bounded operator in $\ell^2(\SSS)$.

Assume we are given a matrix-valued function
\begin{equation}\label{functionw}
w : \SSS \rightarrow \mathbf{M}_R(\C)
\end{equation}
to be 	called the \textit{jump matrix}. We denote the space of $R\times R$ complex matrices by $\mathbf{M}_R(\C)$. A solution of the \textit{discrete Riemann-Hilbert problem}\footnote{DRHP, for simplicity} $(\SSS, w)$ is a holomorphic map
\begin{equation}\label{functionm}
m : \C\setminus \SSS \rightarrow \mathbf{M}_R(\C)
\end{equation}
with simple poles at the points of $\SSS$ that satisfies the following conditions
\begin{itemize}
	\item $m(\zeta)$ is analytic in $\C\setminus\SSS$.
	\item $m(\zeta)\longrightarrow I_R$, as $\zeta\longrightarrow\infty$. Here $I_R$ is the identity $R\times R$ matrix.
	\item $\Res_{\zeta = x}{m(\zeta)} = \lim_{\zeta\rightarrow x}({m(\zeta)w(x))} \ \forall \ x\in\SSS$.
\end{itemize}
The second bullet above needs to be clarified when $\SSS$ is an infinite set. It means that the convergence holds uniformly in a sequence of expanding contours in the complex plane, say for example in a sequence of circles $C_k = \{\zeta \in \C : |\zeta| = r_k\}$ whose radii tend to infinity: $\lim_{k\rightarrow\infty}{r_k} = +\infty$. Moreover we make the requirement that $\inf_{k\geq 1}{\textrm{dist}(\SSS, C_k)} > 0$, i.e., the distance between the contours and the discrete set $\SSS$ is (uniformly) bounded away from $0$. The following statement gives existence and uniqueness of solutions $m(\zeta)$ and moreover it provides formulas for the entries of $K = L(1+L)^{-1}$.

\begin{prop}[\cite{B2}]\label{borodinDRHP}
In the setup avoe, let $L$ be an integrable operator on $\ell^2(\SSS)$, such that $(1+L)$ is invertible. Consider the (column) vector-valued functions
\begin{equation*}
f(x) = (f_1(x), \ldots, f_R(x))^t, \hspace{.2in}g(x) = (g_1(x), \ldots, g_R(x))^t,
\end{equation*}
with entries in $\ell^2(\SSS)$ and the jump matrix
\begin{equation*}
w(x) = -f(x)g(x)^t \in \mathbf{M}_R(\C).
\end{equation*}
Then there exists a unique solution $m(\zeta)$ to the DRHP $(\SSS, w)$, i.e., a function $m : \C \setminus \SSS \rightarrow \mathbf{M}_R(\C)$ satisfying the three bullets above, where in the second bullet we require uniform convergence in a sequence of expanding contours in $\C$ that are uniformly bounded away from $\SSS$.

Furthermore, the matrix of $K = L(1+L)^{-1}$ has the form

\begin{equation*}
K(x, y) = \left\{
\begin{aligned}
    \frac{(G(y))^tF(x)}{x - y} \hspace{.8in}& \textrm{ if } x\neq y,\\
    (G(x))^t\lim_{\zeta\rightarrow x}{(m'(\zeta)f(x))} \hspace{.2in} & \textrm{ if } x = y,
\end{aligned}\right.
\end{equation*}

where

\begin{equation*}
F(x) = \lim_{\zeta\rightarrow x}{(m(\zeta)f(x))}, \hspace{.1in} G(x) = \lim_{\zeta\rightarrow x}{(m(\zeta))^{-t}g(x)}.
\end{equation*}
\end{prop}

\subsection{DRHP and $L$-ensembles}\label{kernelsortL}

Next we make use of the theory of the DRHP to find explicit formulas for the correlation kernel of the point process $\mathcal{P}_L$ of Section $\ref{orthLsubsection}$; let us recall the setting of that section.

Let $N\in\N$ be a positive integer and $\X\subset\R$ be a locally finite set with a splitting $\X = \X_I\sqcup\X_{II}$, $|\X_{II}| = N$. Assume moreover there exists $\epsilon > 0$ such that $|x - y| > \epsilon \ \forall \ x\neq y$ in $\X$. Let us additionally assume that the operator $T: \ell^2(\X) \rightarrow \ell^2(\X)$ defined in $(\ref{Top})$ is well-defined and is a bounded operator.
Consider a map $h_I:\X_I\rightarrow\R_{\geq 0}$ such that $\sum_{x\in\X_I}{\frac{(h_I(x))^2}{1+x^2}}<\infty$ and a map $h_{II}:\X_{II}\rightarrow\R_{>0}$. In terms of the direct sum decomposition $\ell^2(\X) = \ell^2(\X_I)\oplus\ell^2(\X_{II})$, the matrix
\begin{equation}\label{eq:matrixentries}
L = \begin{bmatrix}
    0 & A\\
    -A^* & 0
\end{bmatrix}, \textrm{ where } A(x, y) = \frac{h_I(x)h_{II}(y)}{x - y}, \ x\in\X_I, \ y\in\X_{II},
\end{equation}
defines an $L$-ensemble that we called $\PP_L$. By virtue of Proposition $\ref{kernelLensemble}$, a correlation kernel for $\PP_L$ is given by $K = L(1 + L)^{-1}$. Next we use Proposition $\ref{borodinDRHP}$ to find an explicit correlation kernel for $\PP_L$ in terms of analytic functions.

To begin with, let us assume the additional constraint $h_I\in\ell^2(\X_I)$; this is a stronger condition than $\sum_{x\in\X_I}{\frac{(h_I(x))^2}{1+x^2}}<\infty$. Note that $|\X_{II}| < \infty$ implies $h_{II}\in\ell^2(\X_{II})$. The operator $L$ defined above is integrable in the sense of Section $\ref{DRHP}$ above. Indeed, if we let
\begin{equation*}
R = 2,
\end{equation*}
\begin{eqnarray*}
f_1(x) =     \begin{cases}
                  -h_I(x) &\textrm{if } x\in\X_I\\
                  0 &\textrm{if } x\in\X_{II}
\end{cases} &\hspace{.2in}&
f_2(x) =	\begin{cases}
                  0 &\textrm{if } x\in\X_I\\
                  h_{II}(x) &\textrm{if } x\in\X_{II}
\end{cases}\\
g_1(x) =	\begin{cases}
                  0 &\textrm{if } x\in\X_I\\
                  -h_{II}(x) &\textrm{if } x\in\X_{II}
\end{cases} &\hspace{.2in}&
g_2(x) =	\begin{cases}
                  h_I(x) &\textrm{if } x\in\X_I\\
                   0 &\textrm{if } x\in\X_{II}
\end{cases}
\end{eqnarray*}
then the entries of $L$ in $(\ref{eq:matrixentries})$ are given by $L(x, y) = \mathbf{1}_{\{x \neq y\}}\cdot\frac{f_1(x)g_1(y) + f_2(x)g_2(y)}{x - y}$. The corresponding jump matrix $w: \X\rightarrow \mathbf{M}_2(\C)$ is
\begin{eqnarray}\label{eqn:jumpmatrix}
w(x)  = \begin{cases}
                  \begin{bmatrix}
    0 & -\left(h_I(x)\right)^2\\
    0 & 0
\end{bmatrix} &\textrm{if } x\in\X_{I},\\
                  \begin{bmatrix}
    0 & 0\\
    -\left(h_{II}(x)\right)^2 & 0
\end{bmatrix} &\textrm{if } x\in\X_{II}.
\end{cases}
\end{eqnarray}

Since $h_I, h_{II}$ are real-valued, then $L^* = -L$, which implies that $-1$ is not in the spectrum of $L$ and $(1+L)$ is invertible. By Proposition $\ref{borodinDRHP}$, we then have that the DRHP $(\X, w)$ has a unique solution
\begin{equation*}
m: \C \setminus \X \rightarrow \mathbf{M}_2(\C).
\end{equation*}
Write the coordinates of the matrix-valued function $m$ as
\begin{equation*}
m = \begin{bmatrix}
    m_{11} & m_{12}\\
    m_{21} & m_{22}
\end{bmatrix} =
\begin{bmatrix}
    R_I & -S_{II}\\
    -S_I & R_{II}
\end{bmatrix}
\end{equation*}
Being $m$ the solution to the DRHP $(\X, w)$ translates in this case to the following points:
\begin{itemize}
	\item $R_I, S_I$ are holomorphic in $\C \setminus \X_{II}$ and $R_{II}, S_{II}$ are holomorphic in $\C \setminus \X_I$.
	\item $R_I(\zeta), R_{II}(\zeta) \longrightarrow 1$ and $S_I(\zeta), S_{II}(\zeta) \longrightarrow 0$ as $\zeta\longrightarrow\infty$. Both limits are uniform in a sequence of expanding contours in the complex plane. The contours in question are uniformly bounded away from $\X$.
	\item The functions $R_I, S_I, R_{II}, S_{II}$ have only simple poles.
	\item $\Res_{\zeta = x}{R_I(\zeta)} = \left(h_{II}(x)\right)^2S_{II}(x)\ \forall x\in\X_{II}$.
	\item $\Res_{\zeta = x}{S_I(\zeta)} = \left(h_{II}(x)\right)^2R_{II}(y) \ \forall x\in\X_{II}$.
	\item $\Res_{\zeta = x}{R_{II}(\zeta)} = \left(h_I(x)\right)^2S_I(x) \ \forall x\in\X_I$.
	\item $\Res_{\zeta = x}{S_{II}(\zeta)} = \left(h_I(x)\right)^2R_I(x) \ \forall x\in\X_I$.
\end{itemize}

Proposition $\ref{borodinDRHP}$ can be rewritten in this case as follows.

\begin{prop}\label{borodinDRHPrewritten}
Assume we are in the setting described above, in particular $h_I\in\ell^2(\X_I)$, and $L$ is defined in $(\ref{eq:matrixentries})$. Let
\begin{equation*}
m =
\begin{bmatrix}
    R_I & -S_{II}\\
    -S_I & R_{II}
\end{bmatrix}
\end{equation*}
be the unique solution to the DRHP $(\X, w)$. Then $(1+L)$ is invertible and the operator $K = L(1+L)^{-1}$ has matrix entries
\begin{equation*}
K(x, y) = \left\{
\begin{aligned}
    h_I(x)h_{I}(y)&\cdot\frac{R_I(x)S_I(y) - S_I(x)R_I(y)}{x-y}, \textrm{ if } x, y \in\X_I, \ x\neq y,\\
    h_I(x)h_{II}(y)&\cdot\frac{R_I(x)R_{II}(y) - S_I(x)S_{II}(y)}{x-y}, \textrm{ if } x\in\X_I,\ y \in\X_{II},\\
    h_{II}(x)h_{II}(y)&\cdot\frac{R_{II}(x)R_I(y) - S_{II}(x)S_I(y)}{x-y}, \textrm{ if } x\in\X_{II},\ y \in\X_I,\\
    h_{II}(x)h_{II}(y)&\cdot\frac{R_{II}(x)S_{II}(y) - S_{II}(x)R_{II}(y)}{x-y},\textrm{ if } x, y\in\X_{II},\ x\neq y,
\end{aligned}\right.
\end{equation*}
and the indeterminacy on the diagonal $x = y$ is resolved by L'H\^opital's rule
\begin{equation*}
K(x, x) = \left\{
\begin{aligned}
(h_I(x))^2\cdot((R_I)'(x)S_I(x) - (S_I)'(x)R_I(x)),& \textrm{ if }x\in\X_I,\\
(h_{II}(x))^2\cdot((R_{II})'(x)S_{II}(x) - (S_{II})'(x)R_{II}(x)),& \textrm{ if }x\in\X_{II}.
\end{aligned}\right.
\end{equation*}
\end{prop}

In view of Proposition $\ref{borodinDRHPrewritten}$, our goal is to find explicit formulas for the functions $R_I, R_{II}, S_I, S_{II}$. A way to find such expressions is by coming up with four functions that satisfy the seven bullets stated above, that relate the structure of their poles and their asymptotics at infinity. Such expressions are available in this general setting, see \cite{BO1}. The following result gives such explicit formulas in terms of the orthogonal polynomials $\{p_k(\hatx)\}_{k \geq 0}$ on $\X$ with respect to the weight function $f: \X\rightarrow\R_{\geq 0}$ in $(\ref{funcf})$. When $h_I\in\ell^2(\X_I)$, then $f$ has finite $(2N)$ moments, so at least the orthogonal polynomials $p_0, \ldots, p_N$ exist and $h_n = \sum_{x\in\X}{(p_n(x))^2f(x)} < \infty$ for all $n = 0, 1, \ldots N$.

\begin{prop}[\cite{BO1}, Prop. 5.8 and (5.12)]\label{RISI}
Under the assumptions above, in particular $h_I\in\ell^2(\X_I)$,
\begin{equation*}
\begin{gathered}
R_I(\zeta) = \frac{p_N(\zeta)}{\prod_{y\in\X_{II}}{(\zeta - y)}}, \hspace{.2in} S_I(\zeta) = \frac{p_{N-1}(\zeta)}{h_{N-1}\prod_{\substack{y\in\X_{II}}}{(\zeta- y)}},\\
R_{II}(\zeta) = 1 + \sum_{y\in\X_I}{\frac{(h_I(y))^2S_I(y)}{\zeta-y}}, \hspace{.2in} S_{II}(x) = \sum_{y\in\X_I}{\frac{(h_I(y))^2R_I(y)}{\zeta-y}}.
\end{gathered}
\end{equation*}
\end{prop}

The following is an immediate consequence of Proposition $\ref{RISI}$. Although the result itself is not used, it was important to guide us in our search of explicit formulas for solutions of a DRHP, see Remark $\ref{rem:borodin}$.

\begin{cor}[\cite{BO1}, Cor. 5.9]\label{RIISII}
Under the assumptions above, for any $x\in\X_{II}$, we have
\begin{equation*}
\begin{gathered}
R_{II}(x) = \frac{p_{N-1}(x)}{h_{N-1}\left(h_{II}(x)\right)^2\prod_{\substack{y\in\X_{II} \\ y\neq x}}{(x - y)}}, \hspace{.2in} S_{II}(x) = \frac{p_N(x)}{\left(h_{II}(x)\right)^2\prod_{\substack{y\in\X_{II} \\ y\neq x}}{(x - y)}}.
\end{gathered}
\end{equation*}
\end{cor}

Lastly let us state a theorem which can be applied \textit{without} the assumption $h_I\in\ell^2(\X)$, though the result still requires the weaker assumption $\sum_{x\in\X_I}{\frac{(h_I(x))^2}{1+x^2}} < \infty$.

\begin{thm}[\cite{BO1}, Thm. 5.10]\label{thm:kernelLgeneral}
Assume we are under the assumptions above, including $\sum_{x\in\X_I}{\frac{(h_I(x))^2}{1+x^2}} < \infty$, $L$ is defined as in $(\ref{eq:matrixentries})$ and $K = L(1+L)^{-1}$. Define $\epsilon:\X\rightarrow\{+1, -1\}$ by
\begin{equation*}
\epsilon(x) = \left\{
\begin{aligned}
\textrm{sgn}\left(\prod_{y\in\X_{II}}{(x-y)}\right),& \textrm{ if }x\in\X_I,\\
\textrm{sgn}\left(\prod_{y\in\X_{II}\setminus\{x\}}{(x-y)}\right),& \textrm{ if }x\in\X_{II}.
\end{aligned}\right.
\end{equation*}
Recall that a function $f:\X\rightarrow\R_{\geq 0}$ was defined from the pair of functions $(h_I, h_{II})$ in the general setting of Section $\ref{orthLsubsection}$. The orthogonal polynomials $p_0 = 1, p_1, \ldots, p_{N-1}$ with respect to the weight function $f$ on $\X$ exist, and $h_n = \sum_{x\in\X}{(p_n(x))^2f(x)} < \infty$ for $n = 0, 1, \ldots, N-1$. Define the map $K':\X\times\X\rightarrow\R$ by
\begin{equation*}
K'(x, y) = \left\{
\begin{aligned}
\sqrt{f(x)f(y)}\cdot\sum_{n=0}^{N-1}{\frac{p_n(x)p_n(y)}{h_n}}, & \textrm{ if }x\in\X_I,\\
\delta_{xy} - \sqrt{f(x)f(y)}\cdot\sum_{n=0}^{N-1}{\frac{p_n(x)p_n(y)}{h_n}},& \textrm{ if }x\in\X_{II}.
\end{aligned}\right.
\end{equation*}
Then
\begin{equation*}
K(x, y) = \epsilon(x)K'(x, y)\epsilon(y) \ \forall x, y \in \X.
\end{equation*}
\end{thm}

\section{Correlation kernel of $\LL^{(N)}$}\label{kernelL}

In Section $\ref{sec:Lensemble}$, we showed $\LL^{(N)}$ is an $L$-ensemble on the state space $\Zp^{\epsilon}$. Proposition $\ref{kernelLensemble}$ implies that $\LL^{(N)}$ is determinantal. In this section, our objective is to find a correlation kernel $K^{\LL^{(N)}}$ of $\LL^{(N)}$. The real struggle will be to find formulas that are suitable for the limit transition $N\rightarrow\infty$. Throughout this section, assume as usual that $(z, z')\in\U_{\adm}$; we do restrict $(z, z')$ to smaller subsets of $\C^2$ in some arguments (or make analytic continuations of $(z, z')$ to domains of $\C^2$ containing $\U_0$), but we will point it out when we do. Observe that $(z, z')\in\U_{\adm}$ implies that $\Sigma = z+z'+b\in (-1, +\infty)$.

\subsection{Strategy of the computation}\label{sec:strategy}

As shown in Proposition $\ref{kernelLensemble}$, a correlation kernel $K^{\LL^{(N)}}$ of the $L$-ensemble $\LL^{(N)}$, whose $L$-matrix $L_N$ is given in $(\ref{Lmatrixdef})$, is
\begin{equation*}
K^{\LL^{(N)}} = L_N\left(1 + L_N\right)^{-1}.
\end{equation*}
In order to give explicit expressions for $K^{\LL^{(N)}}$, we use the general theory of the DRHP from the previous section. Remember also the specific setting of Section $\ref{sec:specificL}$. The state space is $\X = \Zpe$ with the splitting into spaces $\X_I = \Z^{\epsilon}_{\geq N}$ and $\X_{II} = \Z^{\epsilon}_{<N}$. We defined two functions $h_I = \sqrt{\psi_{\geq N}} : \Z^{\epsilon}_{\geq N}\rightarrow\R_{>0}$ and $h_{II} = \sqrt{\psi_{<N}} : \Z^{\epsilon}_{<N}\rightarrow\R_{\geq 0}$. The entries of the matrix $K^{\LL^{(N)}}$ can be expressed in terms of the solution of the DRHP $(\Zpe, w_N)$, where
\begin{eqnarray*}
w_N(\hatx)  = \begin{cases}
                  \begin{bmatrix}
    0 & -\psi_{\geq N}(\hatx)\\
    0 & 0
\end{bmatrix} &\textrm{if } \hatx\in\Z^{\epsilon}_{\geq N},\\
                  \begin{bmatrix}
    0 & 0\\
    -\psi_{<N}(\hatx) & 0
\end{bmatrix} &\textrm{if } \hatx\in\Z^{\epsilon}_{<N}.
\end{cases}
\end{eqnarray*}
is the jump matrix, cf. $(\ref{eqn:jumpmatrix}$). Proposition $\ref{borodinDRHPrewritten}$ expresses the entries of $K^{\LL^{(N)}}$ in terms of four functions which solve the DRHP $(\Zpe, w_N)$. In this section, we redenote the four functions $R_I, R_{II}, S_I, S_{II}$ in question by $R_{\geq N}, R_{<N}, S_{\geq N}, S_{<N}$, respectively.

In order to apply the DRHP method in the general setting of the previous section, we need the assumption that $h_I\in\ell^2(\X_I)$. In our special case, this assumption is equivalent to $\psi_{\geq N}\in\ell^1\left(\Z^{\epsilon}_{\geq N}\right)$. By the identity $\Gamma(z)\Gamma(1 - z) = \frac{\pi}{\sin\pi z}$, the formula $(\ref{eqn:psigreaterN})$ for $\psi_{\geq N}$ can be rewritten as
\begin{equation*}
\begin{gathered}
\psi_{\geq N}(\hatx) = \frac{\sin(\pi z)\sin(\pi z')}{\pi^2}(x + \epsilon)\frac{\Gamma(x + N + 2\epsilon)^2}{\Gamma(x - N + 1)^2}\cdot\Gamma\left[ \begin{split}
x + 1 &,& x + a + 1\\
x + 2\epsilon &,& x + b + 1
\vphantom{\frac12}
\end{split}
\right]\\
\times\Gamma\left[ \begin{split}
x + 1 - z - N &,& x + 1 - z' - N\\
x + z + N + 2\epsilon &,& x + z' + N + 2\epsilon
\vphantom{\frac12}
\end{split}
\right] \ \forall \ \hatx\in\Z^{\epsilon}_{\geq N}.
\end{gathered}
\end{equation*}
By well known asymptotics of the Gamma function, we have $\psi_{\geq N}(\hatx) \sim \frac{\sin\pi z \sin \pi z'}{\pi^2}x^{1 - 2\Sigma}$ and so $\psi_{\geq N}\in\ell^1(\Z^{\epsilon}_{\geq N})$ if $\Sigma > 1$.

Another condition that we need in order to apply the theory of the DRHP is that the operator $T$ on $\ell^2(\Zpe)$ defined in $(\ref{Top})$ is well-defined and a bounded operator. Indeed, let us begin by checking that the sum defining the operator $T$ is always absolutely convergent. Given $h\in\ell^2(\Zpe)$ and $\hatx\in\Zpe$ we have
\begin{equation*}
\left(\sum_{\substack{\haty\in\Zpe \\ \haty\neq\hatx}}{\left|\frac{h(\haty)}{(\hatx - \haty)}\right|}\right)^2 \leq
3\left( \sum_{\substack{\haty\in\Zpe \\ \haty\neq\hatx}}{\frac{h(\haty)^2}{(\hatx - \haty)^2}} \right) \leq 3\sum_{\substack{\haty\in\Zpe \\ \haty \neq \hatx}}{h(\haty)^2} \leq 3\left(\|h\|_{\ell^2(\Zpe)}\right)^2 <
\infty,
\end{equation*}
where the second inequality holds because $|\hatx - \haty| \geq 1$ whenever $x\neq y$ are elements of $\Zpe$. Thus the sum defining $(Th)(\hatx)$ is absolutely convergent for any $\hatx\in\Zpe$. By a similar reasoning we can bound
\begin{equation*}
\|Th\|_{\ell^2(\Zpe)}^2 = \sum_{\hatx\in\Zpe}{((Th)(\hatx))^2} \leq
3\sum_{\hatx\in\Zpe}{ \left( \sum_{\substack{\haty\in\Zpe \\ \haty\neq\hatx}}{\frac{h(\haty)^2}{(\hatx - \haty)^2}} \right) } =
3\sum_{\haty\in\Zpe}{h(\haty)^2 \sum_{\substack{\hatx\in\Zpe \\ \hatx \neq \haty}}{\frac{1}{(\hatx - \haty)^2}} } \leq 
C\|h\|_{\ell^2(\Zpe)}^2,
\end{equation*}
where the interchange of sums is easily justified and $C$ is $3$ times any upper bound for $\sum_{\substack{\hatx\in\Zpe \\ \hatx \neq \haty}}{\frac{1}{(\hatx - \haty)^2}}$, uniform over $\hatx\in\Zpe$, for example $C = 6\sum_{n=1}^{\infty}{\frac{1}{n^2}}$.\\

Therefore, if $(z, z')\in\U_{\adm}\cap\{\Sigma > 1\}$, then the general theory of the DRHP applies and we are able to obtain formulas for the analytic functions $R_{\geq N}, R_{<N}, S_{\geq N}, S_{<N}$. In fact, Proposition $\ref{RISI}$ expresses these four functions in terms of the orthogonal polynomials $\mathfrak{p}_N, \mathfrak{p}_{N-1}$ of weight $W(\hatx)$ that we computed in the proof of Theorem $\ref{thm:Okernel}$. However our job is not yet complete at this point because the formulas in that proposition do not allow us to take the limit as $N$ tends to infinity. Most of the difficulty relies in obtaining, via transformations of generalized hypergeometric functions (and massive computations), formulas for $R_{\geq N}, S_{\geq N}, R_{<N}, S_{<N}$ that are suitable for taking the limit $N\rightarrow\infty$.

Finally there is the task of extending the results for all $(z, z')\in\U_{\adm}$ by removing the assumption $\Sigma > 1$. This is done by first showing via different techniques, namely Theorem $\ref{thm:kernelLgeneral}$, that the formulas involved depend analytically on the parameters $z, z'$ and then by means of an analytic continuation.

Before carrying out the plan described here, it is very convenient to construct a different discrete Riemann-Hilbert problem whose solution $\widetilde{m}(\zeta)$ is related to the solution of our desired DRHP $m(\zeta)$ via the simple (but non-invertible) change of variables $\widetilde{m}(\zeta) = m((\zeta + \epsilon)^2)$.

\subsection{A modified DRHP}\label{sec:modifiedDRHP}

Assume the following conditions
\begin{equation*}
\begin{gathered}
\Sigma = z+z'+b > 1,\\
(a, b) \neq (-1/2, -1/2) \textrm{ (which is equivalent to $\epsilon > 0$)}.
\end{gathered}
\end{equation*}

From Proposition $\ref{borodinDRHP}$ and the discussion above, there exists a unique solution $m:\C \setminus \Zpe \longrightarrow \mathbf{M}_2(\C)$ to the DRHP $(\Zpe, w_N)$. If we let
\begin{equation*}
m = \begin{bmatrix}
    R_{\geq N} & -S_{<N}\\
    -S_{\geq N} & R_{<N}
\end{bmatrix},
\end{equation*}
then the following relations hold and uniquely define $m$, see Section $\ref{kernelsortL}$:
\begin{itemize}
	\item $R_{\geq N}, S_{\geq N}$ are holomorphic in $\C\setminus\Z^{\epsilon}_{<N}$ and $R_{<N}, S_{<N}$ are holomorphic in $\C\setminus\Z^{\epsilon}_{\geq N}$.
	\item $R_{\geq N}(\zeta), R_{<N}(\zeta) \longrightarrow 1$ and $S_{\geq N}(\zeta), S_{<N}(\zeta) \longrightarrow 0$ as $\zeta\longrightarrow\infty$. The convergence is uniform over a sequence of expanding contours in the complex plane. The contours in question are uniformly bounded away from $\Zpe$.
	\item The functions $R_{\geq N}, R_{<N}, S_{\geq N}, S_{<N}$ have only simple poles.
	\item $\Res_{\zeta = x}{R_{\geq N}(\zeta)} = \psi_{<N}(x)S_{<N}(x) \ \forall x\in\Z^{\epsilon}_{< N}$.
	\item $\Res_{\zeta = x}{S_{\geq N}(\zeta)} = \psi_{<N}(x)R_{<N}(x) \ \forall x\in\Z^{\epsilon}_{< N}$.
	\item $\Res_{\zeta = x}{R_{<N}(\zeta)} = \psi_{\geq N}(x)S_{\geq N}(x) \ \forall x\in\Z^{\epsilon}_{\geq N}$.
	\item $\Res_{\zeta = x}{S_{<N}(\zeta)} = \psi_{\geq N}(x)R_{\geq N}(x) \ \forall x\in\Z^{\epsilon}_{\geq N}$.
\end{itemize}
Let us introduce a different DRHP. Define $\Y$ as the set
\begin{equation}\label{Ydef}
\Y \myeq \Zp \sqcup (-\Zp - 2\epsilon) = \{y : y\in\Zp\} \sqcup \{-y-2\epsilon : y\in\Zp\},
\end{equation}
with the splitting
\begin{equation*}
\begin{gathered}
\Y = \Y_{\geq N} \sqcup \Y_{< N},\\
\Y_{\geq N} \myeq \{N, N+1, \ldots\}\sqcup\{-N-2\epsilon. -N-1-2\epsilon, \ldots\},\\
\Y_{< N} \myeq \{0, 1, \ldots, N - 1\}\sqcup\{-2\epsilon, \ldots, -N+1-2\epsilon\}.
\end{gathered}
\end{equation*}

Consider the functions $\widetilde{\psi}_{\geq N}: \Y_{\geq N}\rightarrow\R$, $\widetilde{\psi}_{<N}:\Y_{<N}\rightarrow\R\setminus\{0\}$ defined by
\begin{equation}\label{psitilde}
\widetilde{\psi}_{\geq N}(y) \myeq \frac{\psi_{\geq N}(\widehat{y})}{2(y+\epsilon)}, \hspace{.1in}y\in\Y_{\geq N}, \hspace{.3in}
\widetilde{\psi}_{<N}(y) \myeq \frac{\psi_{<N}(\widehat{y})}{2(y+\epsilon)}, \hspace{.1in}y\in\Y_{<N},
\end{equation}
where $\psi_{\geq N}$, $\psi_{<N}$ are the functions in $(\ref{eqn:psigreaterN})$ and $(\ref{eqn:psismallerN})$, respectively. Observe that $\epsilon > 0$ is important for the definition above, since otherwise $\widetilde{\psi}_{<N}$ would not be defined at $0$.

Additionally consider the jump matrix $\widetilde{w}_N: \Y \rightarrow \mathbf{M}_2(\C)$ defined by
\begin{eqnarray}\label{eqn:weightwtilde}
\widetilde{w}_N(y)  \myeq \begin{cases}
                  \begin{bmatrix}
    0 & -\widetilde{\psi}_{\geq N}(y)\\
    0 & 0
\end{bmatrix} &\textrm{if } y\in\Y_{\geq N},\\
                  \begin{bmatrix}
    0 & 0\\
    -\widetilde{\psi}_{<N}(y) & 0
\end{bmatrix} &\textrm{if } y\in\Y_{<N}.
\end{cases}
\end{eqnarray}

Finally let $\widetilde{m}: \C\setminus\Y \longrightarrow \mathbf{M}_2(\C)$ be the matrix-valued function
\begin{equation*}
\widetilde{m}(\zeta) \myeq  m((\zeta + \epsilon)^2).
\end{equation*}

We claim that $\widetilde{m}$ is a solution to the DRHP $(\Y, \widetilde{w}_N)$. In fact, $\widetilde{m}$ certainly satisfies the conditions regarding the asymptotics at infinity. We need to check the structure of the poles and the formulas for the residues at these poles. Let $x_0$ be a (simple) pole of $m(\zeta)$ and $\Res_{\zeta = x_0}{m(\zeta)} = c_{-1}$. The Laurent expansion of $m(\zeta)$ around $\zeta = x_0$ is therefore of the form
\begin{equation*}
m(\zeta) = \frac{c_{-1}}{\zeta - x_0} + c_0 + c_1(\zeta-x_0) + c_2(\zeta - x_0)^2 + \ldots, \hspace{.1in}c_{-1}\neq 0.
\end{equation*}
Since the solution $m(\zeta)$ to the DRHP has only poles at $\Zpe$, then the pole $x_0$ must be of the form $x_0 = (y_0 + \epsilon)^2$ for some $y_0\in\Zp$. It follows that
\begin{equation}\label{eqn:pole}
\widetilde{m}(\zeta) = m((\zeta+\epsilon)^2) = \frac{c_{-1}}{(\zeta - y_0)(\zeta + y_0 + 2\epsilon)} + c_0 + c_1(\zeta-y_0)(\zeta + y_0 + 2\epsilon) + \ldots, \hspace{.1in}c_{-1}\neq 0.
\end{equation}
It follows that $\widetilde{m}(\zeta)$ has simple poles at points $y_0\in\Zp$ and $-y_0-2\epsilon\in\Zp-2\epsilon$; in other words, $\widetilde{m}(\zeta)$ has simple poles exactly at the points of $\Y$. Observe that we used $\epsilon > 0$ again, since otherwise the pole at $y_0 = 0$ would have been of order $2$. From $(\ref{eqn:pole})$ it also follows that
\begin{equation}\label{residuesrelation}
\Res_{\zeta = y_0}{\widetilde{m}(\zeta)} = -\Res_{\zeta = -y_0 - 2\epsilon}{\widetilde{m}(\zeta)} = \frac{c_{-1}}{2(y_0+\epsilon)} = \frac{\Res_{\zeta = x_0}{m(\zeta)}}{2(y_0+\epsilon)} \ \forall y_0\in\Zp, \ x_0 = (y_0 + \epsilon)^2.
\end{equation}

From $(\ref{residuesrelation})$ and the fact that $m(\zeta)$ is the solution to the DRHP $(\X, w_N)$, it is evident now that $\widetilde{m}(\zeta)$ is a solution to the DRHP $(\Y, \widetilde{w}_N)$, where $\Y$ is as in $(\ref{Ydef})$ and $\widetilde{w}_N: \Y \rightarrow \mathbf{M}_2(\C)$ is defined in $(\ref{eqn:weightwtilde})$. We have the following uniqueness result.

\begin{lem}\label{lem:uniqueness}
In the setting above, there is a unique solution $\widetilde{m}: \C \setminus \Y \rightarrow \mathbf{M}_2(\C)$ of the DRHP $(\Y, \widetilde{w}_N)$ and it is $\widetilde{m}(\zeta) = m((\zeta+\epsilon)^2)$.
\end{lem}
\begin{proof}
The discussion above shows that there exists at least one solution $\widetilde{m}$ to $(\Y, \widetilde{w}_N)$ given by $\widetilde{m}(\zeta) = m((\zeta+\epsilon)^2)$. In \cite[Lemma 4.7]{B2}, it is shown that there exists at most one solution to the DRHP $(\Y, \widetilde{w}_N)$ provided that $\widetilde{w}_N(x)^2 = 0$ for all $x\in\Y$. In fact, $(\widetilde{w}_N(x))^2 = 0$ follows from $(\ref{eqn:weightwtilde})$, and we are done.
\end{proof}

Like we did with $m$, let us describe the matrix entries of $\widetilde{m}$ more explicitly. If we write
\begin{equation}
\widetilde{m} = \begin{bmatrix}
    \widetilde{R}_{\geq N} & -\widetilde{S}_{<N}\\
    -\widetilde{S}_{\geq N} & \widetilde{R}_{<N}
\end{bmatrix},
\end{equation}
then $\TR_{\geq N}(\zeta) = R_{\geq N}((\zeta + \epsilon)^2)$, $\TS_{\geq N}(\zeta) = S_{\geq N}((\zeta + \epsilon)^2)$, $\TR_{< N}(\zeta) = R_{< N}((\zeta + \epsilon)^2)$ and $\TS_{< N}(\zeta) = S_{< N}((\zeta + \epsilon)^2)$. Moreover the following conditions are satisfied and uniquely determine the solution $\widetilde{m}$ of the DRHP $(\Y, \widetilde{w}_N)$:
\begin{itemize}
	\item $\widetilde{R}_{\geq N}, \widetilde{S}_{\geq N}$ are holomorphic in $\C\setminus\Y_{<N }$ and $\widetilde{R}_{<N}, \widetilde{S}_{<N}$ are holomorphic in $\C\setminus\Y_{\geq N}$.
	\item $\widetilde{R}_{\geq N}(\zeta), \widetilde{R}_{<N}(\zeta) \longrightarrow 1$ and $\widetilde{S}_{\geq N}(\zeta), \widetilde{S}_{<N}(\zeta) \longrightarrow 0$ as $\zeta\longrightarrow\infty$. The convergence is uniform in a sequence of expanding contours in the complex plane. The distance between the contours in question and $\Y$ is uniformly bounded away from $0$.
	\item The functions $\widetilde{R}_{\geq N}, \widetilde{R}_{<N}, \widetilde{S}_{\geq N}, \widetilde{S}_{<N}$ have only simple poles.
	\item $\Res_{\zeta = y}{\TR_{\geq N}(\zeta)} = \widetilde{\psi}_{<N}(y)\TS_{<N}(y) \ \forall y\in\Y_{< N}$.
	\item $\Res_{\zeta = y}{\TS_{\geq N}(\zeta)} = \widetilde{\psi}_{<N}(y)\TR_{<N}(y) \ \forall y\in\Y_{< N}$.
	\item $\Res_{\zeta = y}{\TR_{<N}(\zeta)} = \widetilde{\psi}_{\geq N}(y)\TS_{\geq N}(y)\ \forall y\in\Y_{\geq N}$.
	\item $\Res_{\zeta = y}{\TS_{<N}(\zeta)} = \widetilde{\psi}_{\geq N}(y)\TR_{\geq N}(y) \ \forall y\in\Y_{\geq N}$.
\end{itemize}

\subsection{More on generalized hypergeometric functions}

We break the discussion momentarily to discuss some terminology and results on generalized hypergeometric functions needed for the formulas in subsequent sections. The \textit{very-well-poised $_7F_6(1)$ series} is defined by\footnote{There is a more familiar function $W$ in connection to $_7F_6(1)$ generalized hypergeometric series, see \cite[Ch. VII]{B}, but it is not the one used here}
\begin{equation}\label{verypoised}
W(A; C, D, E, F, G) \myeq \pFq{7}{6}{A ,,,, 1+\frac{A}{2} ,,,, C ,,,, D ,,,, E ,,,, F ,,,, G}{\frac{A}{2} ,, 1+A-C ,, 1+A-D ,, 1+A-E ,, 1+A-F ,, 1+A-G}{1}.
\end{equation}
It is a holomorphic function in the domain defined by $\Re (2 + 2A - C - D - E - F - G) > 0$ and $A-C, \ A - D, \ A - E, \ A - F, \ A - G\notin\Z_{<0}$.

Next we review the functions defined and studied in \cite{Mi1}. For any $A, B, C, D, E, F, G\in\C$ such that
\begin{equation}\label{eqn:condition1}
E + F + G - A - B - C - D = 1,
\end{equation}
the \textit{$L$-function}\footnote{The letter $L$ in this context stands for a linear combination of generalized hypergeometric functions and has nothing to do with $L$-ensembles} is the function defined by
\begin{equation}\label{Lfunction}
\begin{gathered}
L\left[ \begin{split}
A ,\ B ,\ C ,\  D\\
E ;\hspace{.1in} F ,\hspace{.1in} G
\vphantom{\frac12}
\end{split}\right] \myeq \frac{\Gamma(1 - E)\cdot\pFq{4}{3}{A ,,, B ,,, C ,,, D}{E ,,, F ,,, G}{1}}{\pi\cdot\Gamma(F)\Gamma(G)\Gamma(1+A-E)\Gamma(1+B-E)\Gamma(1+C-E)\Gamma(1+D-E)}\\
+ \frac{\Gamma(E - 1)\cdot\pFq{4}{3}{1+A-E ,,,1+B-E ,,, 1+C-E ,,, 1+D-E}{2-E ,,, 1+F-E ,,, 1+G-E}{1}}{\pi\cdot\Gamma(1+F-E)\Gamma(1+G-E)\Gamma(A)\Gamma(B)\Gamma(C)\Gamma(D)}.
\end{gathered}
\end{equation}

In one-line notation, the function is denoted by $L(A, B, C, D; E; F, G)$. The \textit{top and bottom arguments} of the $L$-function above are $A, B, C, D$ and $F, G$, respectively. Clearly $L(A, B, C, D; E; F, G)$ is symmetric as a function of its top arguments $A, B, C, D$, and it is also symmetric as a function of its bottom arguments $F, G$.

As a function of the variables $E, F, G, B, C, D$ (and replacing $A = E+F+G-B-C-D-1$), the function is analytic in the domain $\{E\notin\Z\}$. Similar observations hold for the $L$-function as an analytic function on the variables $E$ and $5$ other variables, setting the sixth one using $(\ref{eqn:condition1})$.

In certain domain, the function $L(A, B, C, D; E; F, G)$ can be expressed as a very-well-poised $_7F_6(1)$ series. Explicitly, if $\Re(F - D) > 0$, $E\notin\Z$, $D+G-E\notin\Z_{<0}$ and $(\ref{eqn:condition1})$ is satisfied, then \cite[(2.3)]{Mi1}, \cite[(7.5.3)]{B}, give
\begin{equation}\label{eqn:LWidentity}
\begin{gathered}
L(A, B, C, D; E; F, G) = W(D + G - E; \ G - A, \ G - B, \ G - C, \ D, \ 1 + D - E)\\
\times \frac{\Gamma(1 + D + G - E)}{\pi\cdot\Gamma[ G, 1 + G - E, F - D, 1 + A + D - E, 1 + B + D - E, 1 + C + D - E ]}.
\end{gathered}
\end{equation}
When $A = -n$ for some $n\in\N$, then the $L$-function can be written as a multiple of one terminating Saalsch\"{u}tzian $_4F_3(1)$ series. In fact, $\frac{1}{\Gamma(-n)} = 0$ and $(\ref{Lfunction})$ imply
\begin{equation}\label{eqn:4F3L}
L\left[ \begin{split}
-n ,\ X ,\ Y ,\  Z\\
U ;\hspace{.1in} V ,\hspace{.1in} W
\vphantom{\frac12}
\end{split}\right] = \frac{\Gamma(1 - U)\cdot\pFq{4}{3}{-n ,,, X ,,, Y ,,, Z}{U ,,, V ,,, W}{1}}{\pi \cdot \Gamma[1 - n - U, \ 1 + X - U, \ 1 + Y - U, \ 1 + Z - U, \ V, \ W]}.
\end{equation}
Other important identities of the $L$-function that we shall use are the \textit{$L$-fundamental identities} and the \textit{$L$-incoherent} relation, see Propositions $\ref{prop:fundamental}$ and $\ref{incoherent}$ in the Appendix.

\subsection{Solving the modified DRHP}

We still assume here that $\Sigma > 1$, $(a, b)\neq \left(-\frac{1}{2}, -\frac{1}{2}\right)$. By following the general setup of Section $\ref{kernelsortL}$, cf. Proposition $\ref{RISI}$, let us define
\begin{equation}\label{RSgeqN}
\widetilde{R}_{\geq N}(\zeta) \myeq \frac{\mathfrak{p}_N(\widehat{\zeta})}{\prod_{j=0}^{N-1}{\left(\widehat{\zeta} - (j + \epsilon)^2\right)}}, \
\widetilde{S}_{\geq N}(\zeta) \myeq \frac{\mathfrak{p}_{N-1}(\widehat{\zeta})}{h_{N-1}\cdot\prod_{j=0}^{N-1}{\left(\widehat{\zeta} - (j + \epsilon)^2\right)}}, \hspace{.1in} \widehat{\zeta} = (\zeta + \epsilon)^2,
\end{equation}
where $\mathfrak{p}_{N-1}(\hatx), \mathfrak{p}_N(\hatx)$ are the $(N-1)$-st and $N$-th orthogonal polynomials on $\Zpe$ with respect to the weight function $W(\hatx)$, and $h_{N-1}$ is the squared norm of $\mathfrak{p}_{N-1}$ in $L^2(\Zpe, W)$. They have explicit formulas given in $(\ref{pNfirst})$, $(\ref{pNfirst1})$ and $(\ref{hN1})$, which yield
\begin{equation}\label{RlargerN}
\begin{gathered}
\widetilde{R}_{\geq N}(\zeta) = \Gamma\left[ \begin{split}
N+a+1 &,& -z+1 &,& -z'+1 &,& 1-N-\Sigma\\
a+1 &,& -z-N+1 &,& -z'-N+1 &,& 1-\Sigma
\vphantom{\frac12}
\end{split}\right]\\
\times\Gamma\left[ \begin{split}
\zeta-N+1 &,& \zeta+2\epsilon\\
\zeta+1 &,& \zeta+N+2\epsilon
\vphantom{\frac12}
\end{split}\right]\pFq{4}{3}{-N ,,, 1-N-\Sigma ,,, \zeta+2\epsilon ,,, -\zeta}{1+a ,,, -z-N+1 ,,, -z'-N+1}{1},\\
\end{gathered}
\end{equation}
\begin{equation}\label{SlargerN}
\begin{gathered}
\widetilde{S}_{\geq N}(\zeta) = \frac{1}{h_{N-1}}\cdot\Gamma\left[ \begin{split}
N+a &,& -z &,& -z' &,& -N-\Sigma\\
1+a &,& -z-N+1 &,& -z'-N+1 &,& -1-\Sigma
\vphantom{\frac12}
\end{split}\right]\\
\times\Gamma\left[ \begin{split}
\zeta-N+1 &,& \zeta+2\epsilon\\
\zeta+1 &,& \zeta+N+2\epsilon
\vphantom{\frac12}
\end{split}\right]\pFq{4}{3}{-N+1 ,,, -N-\Sigma ,,, \zeta+2\epsilon ,,, -\zeta}{1+a ,,, -z-N+1 ,,, -z'-N+1}{1}.
\end{gathered}
\end{equation}

As functions of $\zeta$, the formulas defining $\widetilde{R}_{\geq N}, \widetilde{S}_{\geq N}$ are meromorphic with simple poles at the points of $\Y_{<N}$. As functions of $(z, z')$, $\widetilde{R}_{\geq N}$ is holomorphic on $\U_0 \cap \{\Sigma \neq 0\}$ and $\widetilde{S}_{\geq N}$ is holomorphic on $\U_0$.

Next we define the functions $\widetilde{R}_{<N}, \widetilde{S}_{<N}$ as follows, cf. Proposition $\ref{RISI}$ again,
\begin{equation}\label{R2S2}
\widetilde{R}_{<N}(\zeta) \myeq 1 + \sum_{y=N}^{\infty}{\frac{\psi_{\geq N}(\haty)\widetilde{S}_{\geq N}(y)}{(\zeta - y)(\zeta + y + 2\epsilon)}}, \hspace{.2in} \widetilde{S}_{<N}(\zeta) \myeq \sum_{y=N}^{\infty}{\frac{\psi_{\geq N}(\haty)\widetilde{R}_{\geq N}(y)}{(\zeta-y)(\zeta+y+2\epsilon)}}.
\end{equation}
It can be shown that both $\TS_{\geq N}$ and $\TR_{\geq N}$ are bounded on $\Z_{\geq N} = \{N, N+1, \ldots\}$ (see for instance (A1) in the proof of Proposition $\ref{prop:R2S2limit}$ below that shows $\TR_{\geq N}(y) \rightarrow 1$ as $y\rightarrow\infty$). As we are assuming $\Sigma > 1$, then $\psi_{\geq N}\in\ell^1(\Z^{\epsilon}_{\geq N})$, which implies that both of the sums defining $\TR_{<N}$, $\TS_{<N}$ are absolutely convergent and therefore $\TR_{<N}, \TS_{<N}$ are holomorphic, as functions of $\zeta$, on the domain $\C\setminus\Y_{\geq N}$. As functions of $(z, z')$, $\TR_{<N}$ is holomorphic on $\U_0$ (just like $\TS_{\geq N}$) and $\TS_{<N}$ is holomorphic on $\U_0\cap\{\Sigma \neq 0\}$ (just like $\TR_{\geq N}$).

\begin{prop}\label{prop:R2S2limit}
The function $\TR_{<N}$ admits the closed form
\begin{equation}\label{tildeRsmallerN}
\begin{gathered}
\widetilde{R}_{<N}(\zeta)= -\frac{\sin\pi z}{\pi}\cdot\Gamma\left[ \begin{split}
z' - z,\ z+b+1,\ z+1\\
\Sigma+1\hspace{.5in}
\vphantom{\frac12}\end{split}\right]\times\Gamma\left[ \begin{split}
N+\Sigma+a+1 &,& N+\Sigma+1\\
N+z+b+1 &,& N+z+2\epsilon
\vphantom{\frac12}\end{split}\right]\\
\times\Gamma\left[ \begin{split}
N+\zeta+2\epsilon,\hspace{.3in} N-\zeta\\
N+\zeta+z'+2\epsilon,\ N-\zeta+z'
\vphantom{\frac12}\end{split}\right]\pFq{4}{3}{1+z+b ,, -z' ,, \zeta+N+z+2\epsilon , N-\zeta+z}{1+z-z' ,,, N+z+b+1 ,,, N+z+2\epsilon}{1}\\
+ \textrm{\{ a similar expression with $z$ and $z'$ interchanged \}},
\end{gathered}
\end{equation}
whereas if $\Sigma \neq 0$, the function $\TS_{<N}$ admits the closed form
\begin{equation}\label{tildeSsmallerN}
\begin{gathered}
\widetilde{S}_{<N}(\zeta) = -\frac{\sin \pi z}{2\pi}\cdot\Gamma\left[ \begin{split}
z' - z &,& \Sigma\\
z' &,& z' + b
\vphantom{\frac12}\end{split}\right]\times\Gamma\left[ \begin{split}
N+a+1 &,& N+1\\
N+z+b+1 &,& N+z+2\epsilon
\vphantom{\frac12}\end{split}\right]\\
\times\Gamma\left[ \begin{split}
N+\zeta+2\epsilon &,& N - \zeta\\
N+\zeta+z'+2\epsilon &,& N-\zeta+z'
\vphantom{\frac12}\end{split}\right]\pFq{4}{3}{z+b ,, 1 - z' ,, \zeta+N+z+2\epsilon ,, N-\zeta+z}{1+z-z' ,,, N+z+b+1 ,,, N+z+2\epsilon}{1}\\
+ \textrm{\{ a similar expression with $z$ and $z'$ interchanged \}}.
\end{gathered}
\end{equation}
\end{prop}

\begin{rem}
The right-hand side of $(\ref{tildeRsmallerN})$ is evidently holomorphic, as a function of $(z, z')$, on the domain $\U_0\cap\{z-z'\notin\Z\}$, while the right-hand side of $(\ref{tildeSsmallerN})$ is holomorphic on the domain $\U_0\cap\{z-z'\notin\Z\}\cap\{\Sigma \neq 0\}$. The singularities at $z-z'\notin\Z$ can, however, be removed by simply checking that the residues of the simple poles of each of the two terms in $(\ref{tildeRsmallerN})$ and $(\ref{tildeSsmallerN})$ cancel out each other.
\end{rem}

The remaining of this section is devoted to proving Proposition $\ref{prop:R2S2limit}$, but first we need more suitable expressions for $\widetilde{R}_{<N}(\zeta)$ and $\widetilde{S}_{<N}(\zeta)$ given in the lemma below. We write $h_{N-1}(z, z')$ for the formula given in the right-hand side of $(\ref{hN1})$ and $h_N(z-1, z'-1)$ for the same formula with $z$, $z'$ and $N$ replaced by $z-1$, $z'-1$ and $N+1$, respectively. Consider also the meromorphic function
\begin{equation}\label{Gdef}
G(\zeta) = \frac{\sin(\pi z)\sin(\pi z')\sin(\pi(\zeta-\Sigma+N))}{\sin(\pi\Sigma)\sin(\pi(\zeta-z))\sin(\pi(\zeta-z'))}.
\end{equation}

\begin{lem}\label{R2S2rewritten}
The right-hand side of $(\ref{tildeRsmallerN})$ (the closed formula of $\TR_{<N}(\zeta)$) equals
\begin{equation*}
\begin{gathered}
\frac{G(\zeta)}{2h_{N-1}(z, z')}\cdot\Gamma\left[ \begin{split}
&&\zeta+a+1 &,& N - \zeta &,& \zeta+N+2\epsilon\\
\zeta+b+1 &,& z-\zeta+N &,& z'-\zeta+N &,& z+\zeta+N+2\epsilon &,& z'+\zeta+N+2\epsilon
\vphantom{\frac12}
\end{split}\right]\\
\times\Gamma\left[ \begin{split}
\zeta+1 &,& \zeta+N+a+b\\
\zeta+2\epsilon &,& \zeta-N+2
\vphantom{\frac12}
\end{split}\right]\pFq{4}{3}{-N+1 ,,, -N+1-a ,,, 1+z ,,, 1+z'}{\zeta-N+2 ,,, 1-\zeta-N-a-b ,,, \Sigma+2}{1}\\
+ \pi\cdot\Gamma\left[ \begin{split}
-\Sigma &,& \zeta+1 &,& \zeta+a+1 &,& \zeta+N+2\epsilon &,& \zeta+1-z-N &,& \zeta+1-z'-N\\
&& && && \zeta-N+1
\vphantom{\frac12}
\end{split}\right]\nonumber\\
\times L\left[ \begin{split}
-z &,& -z-b &,& \zeta+1-N-z &,& \zeta+N+z'+2\epsilon\\
&&1+z'-z &;& \zeta+1-z &,& \zeta+a+1-z
\vphantom{\frac12}
\end{split}\right],
\end{gathered}
\end{equation*}
while the right-hand side of $(\ref{tildeSsmallerN})$ (the closed formula of $\TS_{<N}(\zeta)$) equals
\begin{equation*}
\begin{gathered}
\frac{G(\zeta)}{2}\cdot\Gamma\left[ \begin{split}
&& \zeta + a + 1 &,& \zeta + N + 2\epsilon &,& N-\zeta\\
\zeta+ b + 1 &,& z-\zeta+N &,& z'-\zeta+N &,& z+\zeta+N+2\epsilon &,& z'+\zeta+N+2\epsilon
\vphantom{\frac12}
\end{split}
\right]\\
\times\Gamma\left[ \begin{split}
\zeta + 1 &,& \zeta+N+2\epsilon\\
\zeta+2\epsilon &,& \zeta-N+1
\vphantom{\frac12}
\end{split}
\right]\pFq{4}{3}{-N ,,, -N-a ,,, z ,,, z'}{\zeta-N+1 ,,, -\zeta-N-a-b ,,, \Sigma}{1}\nonumber\\
+\pi\cdot h_{N}(z-1,z'-1)\cdot\Gamma\left[ \begin{split}
\zeta+1 &,& \zeta + a + 1 &,& \zeta + N + 2\epsilon &,& \zeta+1-z-N &,& \zeta+1-z'-N\\
 && && \zeta-N+1
\vphantom{\frac12}
\end{split}
\right]\nonumber\\
\times \Gamma(-\Sigma+2)\cdot L\left[ \begin{split}
1-z &,& 1-z-b &,& \zeta+1-N-z &,& \zeta+N+z'+2\epsilon\\
&&1+z'-z &;& \zeta+2-z &,& \zeta+a+2-z
\vphantom{\frac12}
\end{split}\right].
\end{gathered}
\end{equation*}
\end{lem}

\begin{rem}
By virtue of $\Gamma(z)\Gamma(1 - z) = \frac{\pi}{\sin\pi z}$, the first term of $\TR_{<N}(\zeta)$ (in Lemma $\ref{R2S2rewritten}$) has the factor
\begin{equation*}
\frac{1}{\sin\pi(\zeta - z)\sin\pi(\zeta - z')}\frac{1}{\Gamma(z-\zeta+N)\Gamma(z'-\zeta+N)} = \Gamma(\zeta+1-z-N)\Gamma(\zeta+1-z'-N).
\end{equation*}
Then, as a function of $\zeta$, the first term of $\TR_{<N}(\zeta)$ has simple poles at $\Y_{\geq N}\cup (-a-\Z_{<0})\cup(z+N-\Z_{<0})\cup(z'+N-\Z_{<0})$. The second term of $\TR_{<N}(\zeta)$ has simple poles at $\{-N-2\epsilon, -N-1-2\epsilon, \ldots\}\cup(-a-\Z_{<0})\cup(z+N-\Z_{<0})\cup(z'+N-\Z_{<0})$. As it turns out, a consequence of Lemma $\ref{R2S2rewritten}$ is that the residues of the first term of $\TR_{<N}$ with respect to each of the simple poles at $(-a-\Z_{<0})\cup(z+N-\Z_{<0})\cup(z'+N-\Z_{<0})$ are the negatives of the residues of the corresponding poles of the second term of $\TR_{<N}$. Then $\TR_{<N}$ is holomorphic on $\C\setminus\Y_{\geq N}$, as it ought to be. A similar remark applies to $\TS_{<N}$. One actually has that $\TS_{<N}$ can be obtained from $\TR_{<N}$ after the change of variables $z\mapsto z-1$, $z'\mapsto z'-1$, $N\mapsto N+1$ and then multiplying both sides by $h_N(z-1, z'-1)$.
\end{rem}

The proof of Lemma $\ref{R2S2rewritten}$ is very technical, and based on several identities involving the $L$-functions. We relegate the proof to the Appendix, and we proceed here with the proof of Proposition $\ref{prop:R2S2limit}$ assuming Lemma $\ref{R2S2rewritten}$ has been proved. We shall need the observation that the formulas for $\widetilde{R}_{<N}(\zeta)$ and $\widetilde{S}_{<N}(\zeta)$ in Proposition $\ref{prop:R2S2limit}$ are invariant under the involution $\zeta \mapsto -\zeta-2\epsilon$. Moreover the formulas for $\widetilde{R}_{\geq N}(\zeta)$ and $\widetilde{S}_{\geq N}(\zeta)$ in $(\ref{RlargerN})$ and $(\ref{SlargerN})$, respectively, are also invariant under $\zeta\mapsto -\zeta-2\epsilon$.

\begin{lem}\label{lem:involutionpreserving}
$\widetilde{R}_{<N}(\zeta) = \widetilde{R}_{<N}(-\zeta-2\epsilon)$, $\widetilde{S}_{<N}(\zeta) = \widetilde{S}_{<N}(-\zeta-2\epsilon)$, $\widetilde{R}_{\geq N}(\zeta) = \widetilde{R}_{\geq N}(-\zeta-2\epsilon)$, and $\widetilde{S}_{\geq N}(\zeta) = \widetilde{S}_{\geq N}(-\zeta-2\epsilon)$.$\hfill\square$
\end{lem}

We are finally ready to prove Proposition $\ref{prop:R2S2limit}$.

\begin{proof}[Proof of Proposition $\ref{prop:R2S2limit}$]
By the discussion in Section $\ref{sec:modifiedDRHP}$, especially the uniqueness statement in Lemma $\ref{lem:uniqueness}$, the functions $\widetilde{R}_{<N}, \widetilde{R}_{\geq N}, \widetilde{S}_{<N}$ and $\widetilde{S}_{\geq N}$ are uniquely defined by the seven conditions at the end of Section $\ref{sec:modifiedDRHP}$. In the remaining of the proof, we simply have to check that these conditions hold. In other words, we need to study the structure of residues of these four meromorphic functions and their asymptotics at infinity. Let us verify the necessary conditions only for the functions $\widetilde{R}_{<N}$ and $\widetilde{R}_{\geq N}$, since the analysis for the functions $\widetilde{S}_{<N}$ and $\widetilde{S}_{\geq N}$ is similar. In other words, we need to prove:\\

(A1) $\widetilde{R}_{\geq N}(\zeta)$ is meromorphic and it has only simple poles at the points of $\Y_{<N}$.

(B1) $\widetilde{R}_{\geq N}(\zeta) \rightarrow 1$ uniformly as $\zeta\rightarrow\infty$ on expanding contours whose distance from $\Y$ is bounded away from $0$.

(C1) $\Res_{\zeta = y}{\widetilde{R}_{\geq N}(\zeta)} = \widetilde{\psi}_{<N}(y)\widetilde{S}_{<N}(y)\ \forall y\in\Y_{< N}$.

(A2) $\widetilde{R}_{< N}(\zeta)$ is meromorphic and it has only simple poles at the points of $\Y_{\geq N}$.

(B2) $\widetilde{R}_{< N}(\zeta) \rightarrow 1$ uniformly as $\zeta\rightarrow\infty$ on expanding contours whose distance from $\Y$ is bounded away from $0$.

(C2) $\Res_{\zeta = y}{\widetilde{R}_{<N}(\zeta)} = \widetilde{\psi}_{\geq N}(y)\widetilde{S}_{\geq N}(y) \ \forall y\in\Y_{\geq N}$.\\

From formula $(\ref{RlargerN})$, it is clear that $\widetilde{R}_{\geq N}$ is meromorphic as a function of $\zeta$ and all its poles are simple and at the points of $\Y_{<N}$, which proves (A1).

To prove (B1), we need an identity for terminating Saalsch\"{u}tzian $_4F_3(1)$ series, which implies that $\widetilde{R}_{\geq N}(\zeta)$ can be rewritten as
\begin{equation}\label{R1N}
\widetilde{R}_{\geq N}(\zeta) = \pFq{4}{3}{-N ,, -N-a ,, z ,, z'}{\Sigma ,, \zeta-N+1 ,, -\zeta-N-a-b}{1} = \sum_{k=0}^{N-1}{\frac{(-N)_k(-N-a)_k(z)_k(z')_k}{(\Sigma)_k(\zeta-N+1)_k(-\zeta-N-a-b)_k k!}}.
\end{equation}
In fact, the equality between $(\ref{RlargerN})$ and $(\ref{R1N})$ is equivalent to the equality $(\ref{pNfirst}) = (\ref{almostpN})$, which is proved in the Appendix. In the series above, it is evident that each term, except the first one (corresponding to $k = 0$) converges to $0$ as $\zeta\rightarrow\infty$. Therefore $\lim_{\zeta\rightarrow\infty}{\widetilde{R}_{\geq N}(\zeta)} = 1$ and moreover the limit is uniform in \textit{any} sequence of expanding contours in $\C$. Thus (B1) is proved.

From the definition $(\ref{psitilde})$, $\widetilde{\psi}_{< N}(\zeta) = -\widetilde{\psi}_{< N}(-\zeta-2\epsilon)$; together with Lemma $\ref{lem:involutionpreserving}$, we have $\widetilde{\psi}_{<N}(\zeta)\widetilde{S}_{<N}(\zeta) = -\widetilde{\psi}_{<N}(-\zeta-2\epsilon)\widetilde{S}_{<N}(-\zeta-2\epsilon)$ and also $\widetilde{R}_{\geq N}(\zeta) = \widetilde{R}_{\geq N}(-\zeta-2\epsilon)$. Thus (C1) reduces to prove $\Res_{\zeta = y}{\widetilde{R}_{\geq N}(\zeta)} = \widetilde{\psi}_{<N}(y)\widetilde{S}_{<N}(y) \ \forall y\in\{0, 1, \ldots, N-1\}$. Because of expression $(\ref{RlargerN})$ for $\widetilde{R}_{\geq N}(\zeta)$ and
\begin{equation*}
\Res_{\zeta = y}{\frac{\Gamma(\zeta - N + 1)}{\Gamma(\zeta + 1)}} = \Res_{\zeta = y}{\frac{1}{\zeta (\zeta - 1)\cdots(\zeta-N+1)}} = \frac{(-1)^{N-y+1}}{\Gamma(y+1)\Gamma(N-y)} \ \forall y\in\{0, 1, \ldots, N-1\},
\end{equation*}
we have that $\Res_{\zeta = y}{\widetilde{R}_{\geq N}(\zeta)}$ equals
\begin{equation}\label{eq:res1}
\begin{gathered}
\frac{(-1)^{N-y+1}\cdot\Gamma(y+2\epsilon)}{\Gamma(y+1)\Gamma(N-y)\Gamma(y+N+2\epsilon)}\times\Gamma\left[ \begin{split}
N+a+1 &,& 1-z &,& 1-z &,& 1-N-\Sigma\\
a+1 &,& 1-z-N &,& 1-z'-N &,& 1-\Sigma
\vphantom{\frac12}
\end{split}
\right]\\
\times\pFq{4}{3}{-N ,,, 1-N-\Sigma ,,, y+2\epsilon ,,, -y}{1+a ,,, 1 - z - N ,,, 1 - z' - N}{1} \ \forall y\in\{0, 1, \ldots, N-1\}.
\end{gathered}
\end{equation}
To prove (C1), we must check that the expression above equals $\widetilde{\psi}_{<N}(y)\widetilde{S}_{<N}(y)$. To evaluate $\widetilde{S}_{<N}(y)$, use Lemma $\ref{R2S2rewritten}$. Observe that the second sum in the expression for $\widetilde{S}_{<N}(y)$ vanishes due to the factor $\Gamma(y+1)/\Gamma(y-N+1)$. Moreover $G(y) = (-1)^{N-y+1}$ follows from the definition $(\ref{Gdef})$ and $\sin \pi(\eta - M) = (-1)^{M+1}\sin \pi\eta \ \forall M\in\Z$. By using also the explicit formula $(\ref{eqn:psismallerN})$ for $\psi_{<N}(\haty)$, we can obtain
\begin{equation}\label{eq:res2}
\widetilde{\psi}_{<N}(y)\TS_{<N}(y) = \frac{\psi_{<N}(\haty)}{2(y+\epsilon)}\TS_{<N}(y) = \frac{(-1)^{N-y+1}\cdot\pFq{4}{3}{-N ,,, -N-a ,,, z ,,, z'}{y-N+1 ,,, -y-N-a-b ,,, \Sigma}{1}}{\Gamma(y-N+1)\Gamma(N-y)}.
\end{equation}
We indeed have $(\ref{eq:res1}) = (\ref{eq:res2})$, since it is equivalent to $(\ref{RlargerN})=(\ref{R1N})$, and the latter was already justified.\\

Now we proceed to prove (A2), (B2) and (C2). The formula in $(\ref{tildeRsmallerN})$ for $\widetilde{R}_{<N}(\zeta)$ shows that it is a meromorphic function whose set of poles is $\Y_{\geq N}$ and they are simple poles. Then (A2) is proved.

To prove (B2), observe that $\TR_{<N}(\zeta) = \TR_{<N}(-\zeta-2\epsilon)$ of Lemma $\ref{lem:involutionpreserving}$ allows us to reduce the statement to $\lim_{\zeta\rightarrow\infty}{\TR_{<N}(\zeta)} = 1$ for $\zeta\rightarrow\infty$ in an expanding sequence of \textit{half-contours} in the closed halfspace $\{z\in\C : \Re z \geq -\epsilon\}\subset\C$ (for instance, in a sequence of semicircles whose radii tends to infinity). For the actual asymptotics, we use the expression of Lemma $\ref{R2S2rewritten}$. Let us begin with the first term of $\TR_{<N}(\zeta)$. First, the terminating Saalsch\"{u}tzian $_4F_3(1)$ series in that first term tends to $1$ as $\zeta\rightarrow\infty$, that is,
\begin{equation*}
\lim_{\zeta\rightarrow\infty}{\pFq{4}{3}{-N+1 ,,, -N+1-a ,,, 1+z ,,, 1+z'}{\zeta-N+2 ,,, 1-\zeta-N-a-b ,,, \Sigma+2}{1}} = 1
\end{equation*}
uniformly for $\zeta$ in any sequence of expanding in $\C$ (this is similar to (B1) above). Second, by using
\begin{equation*}
\begin{gathered}
\frac{1}{\sin(\pi(\zeta - z))\sin(\pi(\zeta - z'))\Gamma(z-\zeta+N)\Gamma(z'-\zeta+N)} = \frac{\Gamma(\zeta+1-z-N)\Gamma(\zeta+1-z'-N)}{\pi^2}\\
\Gamma(N-\zeta) = \frac{\pi}{\Gamma(\zeta - N + 1)\sin(\pi(\zeta-N+1))}
\end{gathered}
\end{equation*}
the factor before the terminating Saalsch\"{u}tzian $_4F_3(1)$ series in the first term of $\TR_{<N}(\zeta)$ equals
\begin{equation}\label{eqn:firstRsmallerN}
\begin{gathered}
\Gamma\left[ \begin{split}
\zeta+a+1 &,& \zeta+N+2\epsilon &,& \zeta+1-z-N &,& \zeta+1-z'-N &,& \zeta+1 &,& \zeta+N+a+b\\
\zeta+b+1 &,& \zeta - N + 1&,& z+\zeta+N+2\epsilon &,& z'+\zeta+N+2\epsilon &,& \zeta+2\epsilon &,& \zeta-N+2
\vphantom{\frac12}
\end{split}\right]\\
\times\frac{\sin(\pi z)\sin(\pi z')}{2 \sin(\pi\Sigma) h_{N-1}(z, z')}\frac{\sin(\pi(\zeta-\Sigma))}{\sin(\pi(\zeta+1))}.
\end{gathered}
\end{equation}
From well known asymptotics of Gamma functions, the first line of $(\ref{eqn:firstRsmallerN})$ is $\zeta^{-2(\Sigma+1)}(1 + O(\zeta^{-1}))$, so it converges to $0$ as $\zeta\rightarrow\infty$, because $\Sigma > -1$. Now let us look at the absolute value of the second line of $(\ref{eqn:firstRsmallerN})$; it is upper bounded by
\begin{equation*}
const\times \left| \frac{\sin(\pi(\zeta - \Sigma))}{\sin\pi\zeta} \right| = const\times \left| \cos(\pi\Sigma) - \cot(\pi\zeta)\sin(\pi\Sigma) \right| \leq const\times\left( 1 + |\cot(\pi\zeta)| \right),
\end{equation*}
where the constant $const$ does not depend on $\zeta$. It is known that $|\cot(\pi\zeta)|$ is uniformly bounded for $\zeta\rightarrow\infty$ in any expanding sequence of contours of $\C$ that are uniformly bounded away from $\Z$. It follows that the first term of $\TR_{<N}(\zeta)$ in Lemma $\ref{R2S2rewritten}$ converges to $0$ as $\zeta\rightarrow\infty$ in the sense we want.

Let us now prove that the second term of $\TR_{<N}(\zeta)$ in Lemma $\ref{R2S2rewritten}$ converges to $1$ uniformly as $\zeta\rightarrow\infty$ in a sequence of expanding half-contours of $\{z\in\C : \Re z\geq -\epsilon\}\subset\C$. Since $\epsilon > -a-1$, then $\Re(\zeta+a+1) > 0$ for all $\zeta\in\{z\in\C : \Re z\geq -\epsilon\}$. Then we can apply identity $(\ref{eqn:LWidentity})$ with parameters $A = -z-b, \ B = \zeta+1-N-z, \ C = \zeta+N+z'+2\epsilon, \ D = -z, \ E = 1+z'-z, \ F = \zeta+1-z. \ G = \zeta+a+1-z$. It follows that the second term of $\TR_{<N}(\zeta)$ in Lemma $\ref{R2S2rewritten}$ equals
\begin{equation}\label{eqn:secondtermR}
\begin{gathered}
\Gamma\left[ \begin{split}
\zeta-z-z'+1 &,& \zeta-z-N+1 &,& \zeta-z'-N+1 &,& \zeta+1\\
\zeta-z-z'-N+1 &,& \zeta-z+1 &,& \zeta-z'+1 &,& \zeta-N+1
\vphantom{\frac12}
\end{split}\right]\\
\times\pFq{7}{6}{\zeta-z-z' ,,, 1+\frac{(\zeta-z-z')}{2} ,,, N ,,, -z ,,, -z' ,,, -N-\Sigma-a ,,, \zeta+b+1}{\frac{(\zeta-z-z')}{2} ,,, \zeta-z-z'-N+1 ,,, \zeta-z+1 ,,, \zeta-z'+1 ,,, \zeta+N+2\epsilon ,,, -\Sigma}{1}.
\end{gathered}
\end{equation}
By well known asymptotics of Gamma functions, the first line of $(\ref{eqn:secondtermR})$ is asymptotically $1 + O(\zeta^{-1})$. Let $v_k(\zeta)$ be the $k$-th term of the $_7F_6$ function in the second line of $(\ref{eqn:secondtermR})$, for all $k\in\Z_{\geq 0}$. Clearly $v_0(\zeta) \equiv 1$, while for any $k > 0$ we have
\begin{equation*}
v_k(\zeta) = const\times\frac{(\zeta+2k-z-z')(\zeta-z-z'+1)_{k-1}(\zeta+b+1)_k}{(\zeta-z-z'-N+1)_k(\zeta-z+1)_k(\zeta-z'+1)_k(\zeta+N+2\epsilon)_k},
\end{equation*}
where $const$ is a positive real value not depending on $\zeta$. Then if $k$ is fixed, $v_k(\zeta)\rightarrow 0$ uniformly for $\zeta$ in any expanding sequence of contours bounded away from $\{z+z'+\Z\}\cup\{z+\Z\}\cup\{z'+\Z\}\cup\{-2\epsilon+\Z\}$. By making some uniform estimates on $|v_k(\zeta)|$, that we omit, we can interchange the order between the limit and sum signs, thus obtaining
\begin{equation*}
\lim_{\zeta\rightarrow\infty}{\pFq{7}{6}{\zeta-z-z' ,,, 1+\frac{(\zeta-z-z')}{2} ,,, N ,,, -z ,,, -z' ,,, -N-\Sigma-a ,,, \zeta+b+1}{\frac{(\zeta-z-z')}{2} ,,, \zeta-z-z'-N+1 ,,, \zeta-z+1 ,,, \zeta-z'+1 ,,, \zeta+N+2\epsilon ,,, -\Sigma}{1}} = 1,
\end{equation*}
in the desired sense. Thus the second term of $\TR_{<N}(\zeta)$ converges to $1$ as $\zeta\rightarrow\infty$ in $\{z\in\C : \Re z \geq -\epsilon\}$ (and therefore in $\C$ as explained above), concluding the proof of (B2).

As we did for (C1), we deduce that (C2) reduces to proving $\Res_{\zeta = y}{\widetilde{R}_{<N}(\zeta)} = \widetilde{\psi}_{\geq N}(y)\widetilde{S}_{\geq N}(y),\ \forall y\in\{N, N+1, \ldots\}$. To calculate the left-hand side, use Lemma $\ref{R2S2rewritten}$. The second term in the expression for $\widetilde{R}_{<N}(\zeta)$ has no residue at $\zeta = y$, but the first term does because of the factor $\Gamma(N - \zeta)$. By using
\begin{equation*}
\Res_{\zeta = y}{\Gamma(N - \zeta)} = \frac{(-1)^{y - N + 1}}{\Gamma(y - N + 1)},
\end{equation*}
and $G(y) = (-1)^{y - N + 1} \ \forall y\in\Z$, we have that $\Res_{\zeta = y}{\widetilde{R}_{<N}(\zeta)}$ equals
\begin{equation}\label{eq:residue1}
\begin{gathered}
\frac{1}{2h_{N-1}\Gamma\left[\zeta-N+1, z-y+N, z'-y+N\right]}\cdot\Gamma\left[ \begin{split}
y+a+1 ,\ y+N+2\epsilon ,\ y+1 ,\ y+N+a+b\\
y+b+1 ,\hspace{.3in} y+2\epsilon ,\hspace{.3in} y-N+2
\vphantom{\frac12}
\end{split}
\right]\\
\times\frac{1}{\Gamma(z+y+N+2\epsilon)\Gamma(z'+y+N+2\epsilon)}
\pFq{4}{3}{-N+1 ,,, -N+1-a ,,, 1+z ,,, 1+z'}{y-N+2 ,,, 1-y-N-a-b ,,, \Sigma+2}{1}.
\end{gathered}
\end{equation}
We need to show that the expression above is equal to $\widetilde{\psi}_{\geq N}(y)\TS_{\geq N}(y) = \frac{\psi_{\geq N}(\haty)\TS_{\geq N}(y)}{2(y+\epsilon)}$; from $(\ref{SlargerN})$ and $(\ref{eqn:psigreaterN})$, the latter equals
\begin{equation}\label{eq:residue2}
\begin{gathered}
\frac{1}{2h_{N-1}}\cdot\Gamma\left[ \begin{split}
y+a+1 &,& y+N+2\epsilon &,& N+a &,& -z &,& -z' &,& -N-\Sigma\\
y+b+1 &,& y-N+1 &,& 1+a &,& -z-N+1 &,& -z'-N+1 &,& -1-\Sigma
\vphantom{\frac12}
\end{split}
\right]\\
\times\frac{1}{\Gamma(z-y+N)\Gamma(z'-y+N)\Gamma(z+y+N+2\epsilon)\Gamma(z'+y+N+2\epsilon)}\\
\times\pFq{4}{3}{-N+1 ,,, -N-\Sigma ,,, y+2\epsilon ,,, -y}{1+a ,,, -z-N+1 ,,, -z'-N+1}{1}.
\end{gathered}
\end{equation}
The equality $(\ref{eq:residue1}) = (\ref{eq:residue2})$ is equivalent to the identity $(\ref{RlargerN})=(\ref{R1N})$ after the change of variables $N\mapsto N-1, z\mapsto z+1, z'\mapsto z'+1$. Hence (C2) is proved and we are done.
\end{proof}

\begin{rem}\label{rem:borodin}
Observe that we proved the formulas for the explicit expressions of $\widetilde{R}_{<N}(\zeta), \widetilde{S}_{<N}(\zeta)$ by checking the residues at the points of $\Zp^{\epsilon}$ and the uniform asymptotics at infinity along expanding contours. However, we did not explain how we came up with these explicit formulas. I owe Alexei Borodin the following suggestion of how to find expressions for $\widetilde{R}_{<N}, \widetilde{S}_{<N}$. From his earlier work on the DRHP, see e.g. \cite{B00}, it is expected that both of these functions satisfy certain second-order difference equations. Due to Corollary $\ref{RIISII}$, we know that $\widetilde{R}_{<N}(x), \widetilde{S}_{<N}(x)$ satisfy the difference equation of Wilson polynomials at points $x\in\Z^{\epsilon}_{<N}$. Taking a leap of faith, we assumed that the functions $\widetilde{R}_{<N}(\zeta), \widetilde{S}_{<N}(\zeta)$ satisfy the difference equations of Wilson polynomials for a general argument. The paper \cite{R} finds (even in a more general $q$-setting) two linearly independent solutions to the difference equation of the Wilson polynomials, known as the \textit{Wilson polynomials of the first and second kinds}. Then we found expressions for $\widetilde{R}_{<N}$ and $\widetilde{S}_{<N}$ as linear combinations of the Wilson polynomials of the first and second kind, with certain residue and asymptotics at infinity constraints. The expressions found in this way are those in Lemma  $\ref{R2S2rewritten}$. After massive computations, we were able to transform these expressions into those of Proposition $\ref{prop:R2S2limit}$.
\end{rem}

\subsection{An explicit correlation kernel $K^{\LL^{(N)}}$}

The culmination of our efforts is the following correlation kernel of $\LL^{(N)}$.

\begin{thm}\label{thm:Lkernel}
Let $a\geq b\geq -\frac{1}{2}$, $(z, z')\in\U_{\adm}$, and $\LL^{(N)}$ be the point process with associated parameters $z, z', a, b$, that was defined in Section $\ref{zprocesses}$. The point process $\LL^{(N)}$ is an $L$-ensemble. Therefore $\LL^{(N)}$ is a determinantal point process and $K^{\LL^{(N)}} = L_N(1+L_N)^{-1}$, where $L_N$ is the matrix in $(\ref{Lmatrixdef})$, is a correlation kernel for $\LL^{(N)}$. The matrix element $K^{\LL^{(N)}}(\hatx, \haty)$ can be expressed as follows
\begin{equation}\label{eqn:KLN}
K^{\LL^{(N)}}(\hatx, \haty) = \left\{
\begin{aligned}
    \sqrt{\psi_{\geq N}(\hatx)\psi_{\geq N}(\haty)}\cdot\frac{\TR_{\geq N}(x)\TS_{\geq N}(y) - \TS_{\geq N}(x)\TR_{\geq N}(y)}{\hatx - \haty}\ & \forall \ \hatx, \haty \in\Z^{\epsilon}_{\geq N},\ \hatx\neq \haty,\\
    \sqrt{\psi_{\geq N}(\hatx)\psi_{<N}(\haty)}\cdot\frac{\TR_{\geq N}(x)\TR_{<N}(y) - \TS_{\geq N}(x)\TS_{<N}(y)}{\hatx - \haty}\ & \forall \ \hatx \in\Z^{\epsilon}_{\geq N},\ \haty \in\Z^{\epsilon}_{< N},\\
    \sqrt{\psi_{< N}(\hatx)\psi_{\geq N}(\haty)}\cdot\frac{\TR_{< N}(x)\TR_{\geq N}(y) - \TS_{< N}(x)\TS_{\geq N}(y)}{\hatx - \haty}\ & \forall \ \hatx \in\Z^{\epsilon}_{< N},\ \haty \in\Z^{\epsilon}_{\geq N},\\
    \sqrt{\psi_{<N}(\hatx)\psi_{<N}(\haty)}\cdot\frac{\TR_{< N}(x)\TS_{<N}(y) - \TS_{<N}(x)\TR_{<N}(y)}{\hatx - \haty}\ & \forall \ \hatx, \haty \in\Z^{\epsilon}_{< N},\ \hatx\neq\haty,
\end{aligned}\right.
\end{equation}
where the functions $\psi_{\geq N}, \psi_{<N}, \TR_{\geq N}, \TS_{\geq N}, \TR_{<N}$ and $\TS_{<N}$ are given in $(\ref{eqn:psigreaterN})$, $(\ref{eqn:psismallerN})$, $(\ref{RlargerN})$, $(\ref{SlargerN})$, $(\ref{tildeRsmallerN})$ and $(\ref{tildeSsmallerN})$, respectively. The indeterminacy in the diagonal $\hatx = \haty$ is resolved by L'H\^opital's rule:
\begin{equation}\label{eqn:KLN2}
K^{\LL^{(N)}}(\hatx, \hatx) = \left\{
\begin{aligned}
\psi_{\geq N}(\hatx)\cdot\left(\frac{(\TR_{\geq N})'(x)\TS_{\geq N}(x) - (\TS_{\geq N})'(x)\TR_{\geq N}(x)}{2(x + \epsilon)} \right)\ & \forall \ \hatx\in\Z^{\epsilon}_{\geq N},\\
\psi_{<N}(\hatx)\cdot\left( \frac{(\TR_{<N})'(x)\TS_{<N}(x) - (\TS_{<N})'(x)\TR_{<N}(x)}{2(x + \epsilon)} \right)\ & \forall \ \hatx\in\Z^{\epsilon}_{<N}.
\end{aligned}\right.
\end{equation}
\end{thm}

\begin{rem}
The formulas for the functions $\TR_{\geq N}, \TS_{<N}$ in $(\ref{RlargerN}), (\ref{tildeSsmallerN})$, are ill-defined if $\Sigma = 0$. In the proof below, for any $\hatx, \haty\in\Zpe$, it is shown that $K^{\LL^{(N)}}(\hatx, \haty)$, as a function of $(z, z')$, can be analytically continued to the domain $\U$ and therefore the formula above for $K^{\LL^{(N)}}(\hatx, \haty)$ has a well-defined value for any $(z, z')\in\U_{\adm}$, even when $\Sigma = 0$.
\end{rem}

\begin{proof}
The first statements were proved in Proposition $\ref{LNensemble}$. Next we show that $K^{\LL^{(N)}}(\hatx, \haty)$, $\hatx, \haty\in\Zpe$, is given by $(\ref{eqn:KLN}), (\ref{eqn:KLN2})$ for $\Sigma \neq 0$, and that each of the expressions of the right-hand side of $(\ref{eqn:KLN}), (\ref{eqn:KLN2})$ can be analytically continued as a function of $(z, z')$ to the set $\{\Sigma = 0\}$.

Assume for now the additional restrictions $\Sigma > 1$ and $(a, b) \neq (-\frac{1}{2}, -\frac{1}{2})$. Then the discussion in Section $\ref{sec:strategy}$ and Proposition $\ref{borodinDRHPrewritten}$ shows how to write the entries $K^{\LL^{(N)}}(\hatx, \haty)$ in terms of the unique solution $m = \left( \begin{smallmatrix} R_{\geq N}&-S_{<N}\\ -S_{\geq N}&R_{<N} \end{smallmatrix} \right)$ to the DRHP $(\Zpe, w_N)$, where $w_N$ is defined in $(\ref{eqn:jumpmatrix})$. The discussion in Section $\ref{sec:modifiedDRHP}$, especially Lemma $\ref{lem:uniqueness}$, shows how $m$ is related to the unique solution $\widetilde{m} = \left( \begin{smallmatrix} \widetilde{R}_{\geq N}& -\widetilde{S}_{<N}\\ -\widetilde{S}_{\geq N}&\widetilde{R}_{<N} \end{smallmatrix} \right)$ to the DRHP $(\Y, \widetilde{w}_N)$, where $\widetilde{w}_N$ is defined in $(\ref{eqn:weightwtilde})$. Finally, we have explicit formulas available for $\widetilde{R}_{\geq N}, \widetilde{S}_{\geq N}, \widetilde{R}_{< N}$ and $\widetilde{S}_{< N}$ in $(\ref{RlargerN})$, $(\ref{SlargerN})$, $(\ref{tildeRsmallerN})$ and $(\ref{tildeSsmallerN})$, respectively. All of these four formulas are well-defined for any $(z, z')\in\U_0\cap\{\Sigma \neq 0\}$, in particular for any $(z, z')\in\U_{\adm}\cap\{\Sigma>1\}$. The theorem then follows if $\Sigma > 1$ and $(a, b) \neq (-\frac{1}{2}, -\frac{1}{2})$. In the rest of the proof, we get rid of these constraints.

The key to getting rid of the constraints is Theorem $\ref{thm:kernelLgeneral}$, which gives $K^{\LL^{(N)}}(\hatx, \haty) = \epsilon(\hatx)K'(\hatx, \haty)\epsilon(\haty)$, for any $\hatx, \haty\in\Zpe$, where
\begin{equation*}
K'(\hatx, \haty) = \left\{
\begin{aligned}
\sqrt{W(\hatx)W(\haty)}\sum_{n=0}^{N-1}{\frac{\mathfrak{p}_n(\hatx)\mathfrak{p}_n(\haty)}{h_n}} & \textrm{ if } \ \hatx\in\Z^{\epsilon}_{\geq N},\\
\delta_{x y} - \sqrt{W(\hatx)W(\haty)}\sum_{n=0}^{N-1}{\frac{\mathfrak{p}_n(\hatx)\mathfrak{p}_n(\haty)}{h_n}} & \textrm{ if} \ \hatx\in\Z^{\epsilon}_{<N},
\end{aligned}\right.
\end{equation*}
and where $W:\Zpe\rightarrow\R_{\geq 0}$ is the weight function defined in $(\ref{weightW})$, $\mathfrak{p}_0 = 1, \ldots, \mathfrak{p}_{N-1}$ are the orthogonal polynomials on $\Zpe$ with respect to $W$, and $h_0, \ldots, h_{N-1}$ are the squared norms of these polynomials.

First let us we get rid of the restriction $(a, b) \neq (-\frac{1}{2}, -\frac{1}{2})$. Let $(z, z')\in\C^2$ be such that $z+z' > \frac{3}{2}$ and $\{z, z'\}\cap\{\ldots, -\frac{3}{2}, -1, -\frac{1}{2}\} = \emptyset$. From $(\ref{weightW})$, for any $\hatx\in\Zpe$ the function $W(\hatx)$ is continuous in $a, b\in (-1, \infty)$. Together with Corollary $\ref{cor:analytic}$, for any $\hatx, \haty\in\Zpe$ it follows that $K'(\hatx, \haty)$ is continuous in the real variables $a\in (-1, \infty)$ and $b$ in some neighborhood $J_b' \subset \R$ of $-\frac{1}{2}$ (the neighborhood $J_b'$ is such that $b\in J_b'$ implies $\{z+b, z'+b\}\cap\{\ldots, -2, -1\} = \emptyset$). Then for any $\hatx, \haty\in\Zpe$, the function $K^{\LL^{(N)}}(\hatx, \haty) = \epsilon(x)K'(\hatx, \haty)\epsilon(y)$ is also continuous in $a\in (-1, \infty)$ and $b$ in the neighborhood $J_b'$ of $-\frac{1}{2}$.

On the other hand, let us look at the continuity of $\psi_{\geq N}, \psi_{<N}, \TR_{\geq N}, \TS_{\geq N}, \TR_{<N}$ and $\TS_{<N}$ with respect to the real variables $a, b$. For any $\hatx\in\Z^{\epsilon}_{\geq N}$, it is clear that $\psi_{\geq N}(\hatx)$ is continuous for $a, b\in (-1, \infty)$. Now let $I_a\subset \R$, $J_b''\subset J_b'\subset\R$ be neighborhoods of $-\frac{1}{2}$ such that $a\in I_a$, $b\in J_b''$ imply $\{z+a+b+1, z'+a+b+1\}\cap\{\ldots, -2, -1\}=\emptyset$. With this definition, it is clear that for any $\hatx \in \Z^{\epsilon}_{< N}$, the function $\psi_{<N}(\hatx)$ is continuous for $a \in I_a$ and $b\in J_b''$. For any $x\in\Y_{\geq N}$, the function $\TR_{\geq N}(x)$ is continuous for $a, b \in (-1, \infty)$, while $\TS_{\geq N}(x)$ is continuous for $a\in (-1, \infty)$ and $b\in J_b'$ (note the latter is because $1/h_{N-1}$ shows up in the definition of $\TS_{\geq N}$). Let us reduce $J_b''$ even further: let $J_b\subset J_b''$ be a neighborhood of $-\frac{1}{2}$ such that $b\in J_b$ implies $\Sigma = z+z'+b > 1$. Then for any $x\in\Y_{<N}$, the functions $\TR_{<N}(x)$ and $\TS_{<N}(x)$ are continuous for $a\in I_a$ and $b\in J_b$. Therefore the right-hand side of $(\ref{eqn:KLN})$ is continuous for $a\in I_a$, $b\in J_b$.

Thus we have already proved that both sides of the equality $(\ref{eqn:KLN})$ are continuous for $a\in I_a$, $b\in J_b$ and that the equality holds for any $a\geq b\geq -\frac{1}{2}$, $(a, b)\neq (-\frac{1}{2}, -\frac{1}{2})$ and $(z, z')\in\U_{\adm}\cap\{\Sigma > 1\}$ (note that $\U_{\adm}$ depends on $b$ too). By construction, if $(z, z')\in\U_{\adm}^{b = -1/2}$ is such that $z+z' > \frac{3}{2}$ and $\{z, z'\}\cap\{\ldots, -\frac{3}{2}, -1, -\frac{1}{2}\} = \emptyset$, then any sequence $\{a_n\}_{n\in\N}\subset I_a$, $\{b_n\}_{n\in\N}\subset J_b$ is such that $(z, z')\in\U_{\adm}^{b_n}\cap\{z+z'+b_n > 1\}$, where $\U_{\adm}^{b_n}$ is $\U_{\adm}$ defined with $b_n$ instead of $b$. Consider sequences $\{a_n\}_{n\in\N}\subset I_a$, $\{b_n\}_{n\in\N}\subset J_b$ such that $a_n > b_n > -\frac{1}{2} \ \forall n\in\N$ and $\lim_{n\rightarrow\infty}{a_n} = \lim_{n\rightarrow\infty}{b_n} = -\frac{1}{2}$. We have shown before that equality $(\ref{eqn:KLN})$ holds if $(a_n, b_n)$ replaces $(a, b)$. Then sending $n$ to infinity, we conclude that $(\ref{eqn:KLN})$ holds also for $(a, b) = (-\frac{1}{2}, -\frac{1}{2})$. Observe that we assumed that $z+z' > \frac{3}{2}$ (or $\Sigma > 1$ with $b = -\frac{1}{2}$), so we still need to remove the constraint $\Sigma > 1$ for all $a \geq b \geq -\frac{1}{2}$. By a similar reasoning, we can show that $(\ref{eqn:KLN2})$ holds for $(a, b) = (-\frac{1}{2}, -\frac{1}{2})$ and $(z, z')\in\U_{\adm}\cap\{\Sigma > 1\}$.

Let us finally get rid of the constraint $\Sigma > 1$. Consider the domain $\mathcal{D} \myeq \{ (z, z')\in\U_0 : \{z+2\epsilon, z'+2\epsilon\}\cap\{\ldots, -3, -2, -1\} = \emptyset \}$. Then clearly $\U_{\adm}\subset\mathcal{D}\subset\U_0\subset\U$. Let us look at the holomorphicity of some functions of $(z, z')$. For any $\hatx\in\Z^{\epsilon}_{\geq N}$, $(\ref{eqn:psigreaterN})$ shows that $\psi_{\geq N}(\hatx)$ is entire on $\C^2$. For any $\hatx\in\Z^{\epsilon}_{<N}$, $(\ref{eqn:psismallerN})$ shows $\psi_{<N}(\hatx)$ is holomorphic on $\mathcal{D}$.
We observed earlier that for any $x\in\Y_{\geq N}$, $\TS_{\geq N}(x)$ is holomorphic on $\U_0$ and $\TR_{\geq N}(x)$ is holomorphic on $\U_0\cap\{\Sigma \neq 0\}$. For any $x\in\Y_{<N}$, $\TR_{<N}(x)$ is holomorphic on $\U_0$ and $\TS_{<N}(x)$ is holomorphic on $\U_0\cap\{\Sigma \neq 0\}$. Thus it follows that for any $\hatx, \haty\in\Zpe$ fixed, the corresponding formula on the right-hand side of $(\ref{eqn:KLN})$ is holomorphic on the domain $\mathcal{D}\cap\{\Sigma \neq 0\}$. It also follows that the right-hand side of $(\ref{eqn:KLN2})$ is holomorphic on $\mathcal{D}\cap\{\Sigma \neq 0\}$.

It is clear, from its explicit formula, that for any $\hatx\in\Zpe$ the function $W(\hatx)$ is entire on $\C^2$. Along with Corollary $\ref{cor:analytic}$, for any $\hatx, \haty\in\Zpe$ the function $K'(\hatx, \haty)$ is holomorphic on $\U$; therefore $K^{\LL^{(N)}}(\hatx, \haty) = \epsilon(\hatx)K'(\hatx, \haty)\epsilon(\haty)$ admits an analytic continuation to $\U$. Since we showed that the right-hand side of $(\ref{eqn:KLN})$ gives $K^{\LL^{(N)}}(\hatx, \haty)$ for any $(z, z')\in\U_{\adm}\cap\{\Sigma > 1\}$ and also that the right-hand side of $(\ref{eqn:KLN})$ is holomorphic on $\mathcal{D}\cap\{\Sigma \neq 0\}$, then it follows that the equality $(\ref{eqn:KLN})$ holds for all $(z, z')\in\mathcal{D}\cap\{\Sigma \neq 0\}$, and in particular for all $(z, z')\in\U_{\adm}\cap\{\Sigma \neq 0\}$. Since $\U_{\adm}\subset\mathcal{D}$, we have also shown that the formulas on the right of $(\ref{eqn:KLN})$ admit analytic continuations to the singular set $\{\Sigma = 0\}$. One similarly obtains that $(\ref{eqn:KLN2})$ holds for all $(z, z')\in\U_{\adm}\cap\{\Sigma \neq 0\}$, and the right-hand side of $(\ref{eqn:KLN2})$ has an analytic continuation to $\{\Sigma = 0\}$.
\end{proof}

\section{The continuous point process $\PP$ on $\R_{>0}\setminus\{1\}$}\label{sec:scalinglimit}

In this section, we describe point processes $\PP = \PP_{z, z', a, b}$ on the continuous state space
\begin{equation}\label{puncturedhalfline}
\X = \R_{>0}\setminus\{1\},
\end{equation}
which will turn out to be a scaling limit of the point processes $\LL^{(N)}$ on the discrete state space $\Zpe$. We do so in a more general setting, but first we need to introduce several notions and results from \cite{OkOl1, OlOs}; see also the exposition in \cite[Sec. 3]{C}.

Since $a\geq b\geq -\frac{1}{2}$, in particular $a, b > -1$, we can define the Jacobi polynomials $\{\mathfrak{P}_k(x | a, b)\}_{k\geq 0}$ as the orthogonal polynomials on $[-1, 1]$ with respect to the weight $(1 - x)^a(1 + x)^b$. For any $N\geq 1$, the multivariate Jacobi polynomials $\{\mathfrak{P}_{\mu}(x_1, \ldots, x_N | a, b)\}_{\mu \in \GTp_N}$ are defined by the determinantal formula
\begin{equation*}
\mathfrak{P}_{\mu}(x_1, \ldots, x_N | a, b) \myeq \frac{\det_{1\leq i, j\leq N}{[\mathfrak{P}_{\mu_i + N - i}(x_j | a, b)]}}{\prod_{1\leq i < j\leq N}{(x_i - x_j)}}, \ x_1, \ldots, x_N\in [-1, 1].
\end{equation*}

It is known that $\{\mathfrak{P}_{\mu}(x_1, \ldots, x_N | a, b)\}_{\mu \in \GTp_N}$ is a basis of the Hilbert space $L^2([-1, 1]^N, \mathfrak{m}(dx))$, where $\mathfrak{m}(dx)$ is the absolutely continuous measure whose density with respect to the Lebesgue measure is $|\Delta(x_1, \ldots, x_N)|^2\prod_{i=1}^N{((1 - x_i)^a(1 + x_i)^b)}$. For convenience we also set $\GTp_0 \myeq \{ \emptyset \}$ and $\mathfrak{P}_{\emptyset} \myeq 1$. For any $N\geq 0$, $\mu\in\GTp_N$, $\lambda\in\GTp_{N+1}$, define the real numbers $\Lambda^{N+1}_N(\lambda, \mu)$ by the equalities
\begin{equation}\label{eqn:markovdef}
\frac{\mathfrak{P}_{\lambda}(x_1, \ldots, x_N, 1 | a, b)}{\mathfrak{P}_{\lambda}(1^{N+1} | a, b)} = \sum_{\mu\in\GTp_N}{\Lambda^{N+1}_N(\lambda, \mu)\frac{\mathfrak{P}_{\mu}(x_1, \ldots, x_N | a, b)}{\mathfrak{P}_{\mu}(1^N | a, b)}}.
\end{equation}

The numbers $\Lambda^{N+1}_N(\lambda, \mu)$ generally depend on $a, b$, but we omit such dependence in the notation, for simplicity. Certain explicit formulas for $\Lambda^{N+1}_N(\lambda, \mu)$ are available, but we will not need them. Moreover $\Lambda^{N+1}_N(\lambda, \mu) \geq 0 \ \forall N\geq 0, \mu\in\GTp_N, \lambda\in\GTp_{N+1},$ and $\sum_{\mu\in\GTp_N}{\Lambda^{N+1}_N(\lambda, \mu)} = 1 \ \forall N\geq 0, \lambda\in\GTp_{N+1}$. Thus the matrix $\left[\Lambda^{N+1}_N(\lambda, \mu) \right]$ of format $\GTp_{N+1}\times\GTp_N$ is a stochastic matrix, that we call the matrix of \textit{cotransition probabilities of the BC-Gelfand-Tsetlin graph}.

If $\Lambda^{N+1}_N(\lambda, \mu)\neq 0$, then we write $\mu\prec_{BC}\lambda$, see \cite[Ch. 3]{C} for an easy combinatorial characterization of the partial ordering $\prec_{BC}$. The \textit{BC-Gelfand-Tsetlin graph} is the graph whose set of vertices is
\begin{equation*}
\GTp \myeq \sqcup_{N\geq 0}{\GTp_N},
\end{equation*}
whose edges connect only signatures of adjacent lengths and $\mu\in\GTp_N$ is connected to $\lambda\in\GTp_{N+1}$ if and only if $\mu\prec_{BC}\lambda$. Moreover the BC-Gelfand-Tsetlin graph has edge-multiplicities given by the real numbers $\Lambda^{N+1}_N(\lambda, \mu)$. Observe that the graph therefore depends on the real parameters $a, b$.

We say that the sequence $\{P_N\}_{N\geq 0}$ is a \textit{coherent sequence of (resp. probability) measures} if each $P_N$ is a finite (resp. probability) measure on $\GTp_N$, and moreover the following relations hold
\begin{equation*}
P_N(\mu) = \sum_{\lambda\in\GTp_{N+1}}{\Lambda^{N+1}_N(\lambda, \mu)P_{N+1}(\lambda)}, \ \forall N\geq 0, \ \forall \lambda \in \GTp_{N+1}, \ \forall \mu \in \GTp_N.
\end{equation*}

For us, the BC type $z$-measures are the most important examples of coherent sequences of probability measures, as the following proposition shows.

\begin{prop}\label{prop:coherent}
For any $a\geq b\geq -1/2$, $(z, z')\in\U_0$, the sequence $\{P_N\}_{N\geq 0}$ of $z$-measures of finite levels associated to $z, z', a, b$ is a coherent sequence of measures. If $(z, z')\in\U_{\adm}$, the sequence $\{P_N\}_{N\geq 0}$ of $z$-measures of finite levels is a coherent sequence of probability measures.
\end{prop}
\begin{proof}
The first statement is proved in \cite[Proof of Thm. 6.4]{C}. The second statement follows from Lemma $\ref{propposdomain}$ and the definition of $\U_{\adm}$.
\end{proof}

A sequence $\tau = (\tau^{(0)}, \tau^{(1)}, \tau^{(2)}, \ldots)$ such that $\tau^{(n)}\in\GTp_n$, $\tau^{(n)} \prec_{BC} \tau^{(n+1)}$ for all $n\geq 0$, is called a \textit{path in the BC-Gelfand-Tsetlin graph}\footnote{Simply a \textit{path}, for short}. The set of all paths is denoted by $\Tau$. There is an evident inclusion $\Tau \hookrightarrow \prod_{n\geq 0}{\GTp_n}$. Equip the set $\prod_{n\geq 0}{\GTp_n}$ with the product topology and $\Tau$ with its inherited topology. The topological space $\Tau$ will be called the \textit{path-space of the BC-Gelfand-Tsetlin graph}.

We equip $\Tau$ with its Borel $\sigma$-algebra. The Borel $\sigma$-algebra on $\Tau$ can be obtained in the following equivalent way. For any finite path $\phi = (\phi^{(0)} \prec_{BC} \phi^{(1)} \prec_{BC} \dots \prec_{BC} \phi^{(n)})$ in the BC-Gelfand-Tsetlin graph, the \textit{cylinder set} $S_{\phi}$ is defined by $S_{\phi} \myeq \{\tau\in\Tau : \tau^{(0)} = \phi^{(0)}, \ldots, \tau^{(n)} = \phi^{(n)}\}$. Then the Borel $\sigma$-algebra of $\Tau$ is the $\sigma$-algebra generated by the cylinder sets $S_{\phi}$, where $\phi$ varies over all finite paths in the BC-Gelfand-Tsetlin graph.

Let $P$ be a finite measure on $\Tau$, but not necessarily a probability measure. Let $\lambda\in\GT_N$ be any signature of positive length $N\geq 1$. Consider all the cylinder sets $S_{\phi}$, where $\phi = (\phi^{(0)} \prec_{BC} \phi^{(1)} \prec_{BC} \dots \prec_{BC} \phi^{(N)})$ varies over all finite paths of length $N$ with $\phi^{(N)} = \lambda$. We say that \textit{$P$ is a central measure at $\lambda$} if $P(S_{\phi})$ is proportional to the product of cotransition probabilities $\Lambda^N_{N-1}(\phi^{(N)}, \phi^{(N-1)})\cdots\Lambda^1_0(\phi^{(1)}, \phi^{(0)})$. In other words, conditioned on $\phi^{(N)} = \lambda$, the probability of a cylinder set $S_{\phi}$ is proportional to the product of cotransition probabilities along the path of $\phi$ in the BC-Gelfand-Tsetlin graph. A probability measure $P$ on $\Tau$ is said to be a \textit{central measure} if it is a central measure at $\lambda$, for any $\lambda\in\GTp_N$, $N\geq 1$.

Next let us make a connection between the objects introduced so far. Let $M(\Tau)$ be the space of signed, finite and \textit{central} measures on $\Tau$. It is known to have a natural Banach structure (the norm is given by the total variation of a measure) and therefore a canonical topology. Let $M_{\textrm{prob}}(\Tau)\subset M(\Tau)$ be the simplex consisting of central probability measures on $\Tau$. For any $P \in M_{\textrm{prob}}(\Tau)$, consider its family of pushforward measures with respect to the natural projection maps
\begin{eqnarray*}
\textrm{Proj}_N: \Tau \subset \prod_{n\geq 0}{\GTp_n} \rightarrow \GTp_N,
\end{eqnarray*}
that is, let $P_N = \left(\textrm{Proj}_N\right)_*P$ for each $N\geq 0$. Almost by definition, it follows that $\{P_N\}_{N\geq 0}$ is a coherent sequence of probability measures. The map
\begin{equation}\label{eqn:bijection}
P \longrightarrow \{ P_N = (\textrm{Proj}_N)_*P \}_{N\geq 0}
\end{equation}
is a bijection between the set of central probability measures on $\Tau$ and the set of coherent probability measures in the BC-Gelfand-Tsetlin graph. Both sets above are convex, and the map $(\ref{eqn:bijection})$ is an isomorphism of convex sets.

Since $M_{\textrm{prob}}(\Tau)$ is a convex set, it is natural to ask for its set of extreme points. We mentioned that $M_{\textrm{prob}}(\Tau)$ is also a topological space, so we can refine the question by asking for a characterization of the (topological) space of its extreme points. We denote the space of extreme points of $M_{\textrm{prob}}(\Tau)$ by $\Omega_{\infty} = \Omega_{\infty}(a, b)$ and call it the \textit{boundary of the BC-Gelfand-Tsetlin graph}. The following fundamental result characterizes $\Omega_{\infty}$ completely; observe the remarkable fact that $\Omega_{\infty}$ is independent of $a, b$.

\begin{thm}[\cite{OkOl1}, Thm. 1.3]\label{thm:BCboundary}
The boundary $\Omega_{\infty}$ of the BC-Gelfand-Tsetlin graph can be embedded as the subspace
\begin{equation*}
\Omega_{\infty} \subset \R_+^{\infty}\times\R_+^{\infty}\times\R_+
\end{equation*}
of triples $\omega = (\alpha; \beta; \delta)$ satisfying the inequalities
\begin{equation*}
\begin{gathered}
\alpha = (\alpha_1 \geq \alpha_2 \geq \alpha_3 \geq \dots \geq 0), \ \beta = (\beta_1 \geq \beta_2 \geq \beta_3 \geq \dots \geq 0), \ 1 \geq \beta_1,\\
\infty > \delta \geq \sum_{i=1}^{\infty}{(\alpha_i + \beta_i)}.
\end{gathered}
\end{equation*}
\end{thm}

We denote elements of $\Omega_{\infty}$ by the Greek letter $\omega$. We denote by $P^{\omega} \in M_{\textrm{prob}}(\Tau)$ the extreme measure on $\Tau$ that corresponds to $\omega\in\Omega_{\infty}$. The following theorem shows that any $P\in M_{\textrm{prob}}(\Tau)$ is a combination of the extreme probability measures $P^{\omega}$.

\begin{thm}[consequence of \cite{Ol}, Thm. 9.2]\label{thm:spectral}
Let $P$ be any central probability measure on $\Tau$. Then there exists a unique probability measure $\pi$ on $\Omega_{\infty}$ such that
\begin{equation*}
P = \int_{\Omega_{\infty}}{P^{\omega}d\pi(\omega)},
\end{equation*}
or in other words,
\begin{equation*}
P(S_{\phi}) = \int_{\Omega_{\infty}}{P^{\omega}(S_{\phi}) d\pi(\omega)},
\end{equation*}
for any finite path $\phi = (\phi^{(0)} \prec_{BC} \phi^{(1)} \prec_{BC} \dots \prec_{BC} \phi^{(n)})$, where $S_{\phi}$ is the associated cylinder set. The probability measure $\pi$ on $\Omega_{\infty}$ is said to be the \textit{spectral measure of $P$}.
\end{thm}

Now the goal is to solve the ``spectral problem'' for certain measures $P$. In other words, given certain central probability measure $P$ on $\Tau$, the question is now to describe its spectral measure $\pi$ on $\Omega_{\infty}$. Since $\Omega_{\infty}$ is an infinite-dimensional space, one cannot write down a density function for $\pi$. Instead the key to the problem lies in the theory of point processes and correlation functions. To go further, begin by defining the measurable map
\begin{equation}\label{eqn:mapi}
\begin{gathered}
\mathfrak{i} : \Omega_{\infty} \longrightarrow \Conf(\X)\\
\omega = (\alpha; \beta; \delta) \mapsto \left( \{(1 + \alpha_i)^2\}_{i\geq 1} \sqcup \{(1 - \beta_j)^2\}_{j\geq 1} \right) \setminus \{0, 1\}.
\end{gathered}
\end{equation}

Recall that $\Conf(\X)$ above stands for the space of simple, locally finite point configurations on $\X$.

Given a probability measure $P$ on $\Tau$, we now aim to describe its spectral measure $\pi$ on $\Omega_{\infty}$ by studying its pushforward probability measure $\PP = \mathfrak{i}_*\pi$ on $\Conf(\X)$. The probability measure $\PP$ defines a point process on $\X$ that we denote by the same symbol $\PP$. Observe that the map $\mathfrak{i}$ hides some information about the coordinates $i = 1, 2, \ldots$ with $\alpha_i = 0$, $\beta_i = 0$, $\beta_i = 1$ (namely it hides the number of coordinates $k$ for which $\alpha_k = 0$, etc.), and it also hides $\delta$. However, this issue is not very significant for the problem of harmonic analysis that is our main motivation, see for example the discussion in \cite[Sec. 9]{BO1}.

The problem now is to compute the correlation functions of the point process $\PP$. This is still not an easy question, but we have an strategy to solve it by approximating $\PP$ by point processes with finite number of points.

For a central probability measure $P$ on $\Tau$, let $\{P_N\}_{N\geq 0}$ be the associated coherent sequence of probability measures, i.e., the image of $P$ under the bijective map $(\ref{eqn:bijection})$. Then for each $N\geq 0$, $P_N$ is a probability measure on $\GTp_N$. For each $N\geq 1$, consider the injective map
\begin{equation*}
\begin{gathered}
\mathfrak{k}_N : \GTp_N \hookrightarrow \Omega_{\infty}\\
\lambda = (p_1, \ldots, p_d | q_1, \ldots, q_d) \mapsto \mathfrak{k}_N(\lambda) = \omega = (\alpha; \beta; \delta),\\
\alpha_i = \frac{p_i + \frac{1}{2}}{N+\epsilon-\frac{1}{2}} \ \forall i = 1, \ldots, d; \ \alpha_i = 0 \ \forall i > d,\\
\beta_i = \frac{q_i + \frac{1}{2}}{N+\epsilon-\frac{1}{2}} \ \forall i = 1, \ldots, d; \ \beta_i = 0 \ \forall i > d,\\
\delta = \frac{|\lambda|}{N+\epsilon-\frac{1}{2}} = \frac{\lambda_1 + \ldots + \lambda_N}{N+\epsilon-\frac{1}{2}}.
\end{gathered}
\end{equation*}
The pushforwards $(\mathfrak{k}_N)_*P_N$ are probability measures on $\Omega_{\infty}$. By analogy to the ``A-type case'', we expect that $(\mathfrak{k}_N)_*P_N$ weakly converges to $P$ as $N$ goes to infinity.

Further let $\widetilde{\PP}^{(N)}$ be the pushforward of $(\mathfrak{k}_N)_*P_N$ under the map $\mathfrak{i} : \Omega_{\infty} \rightarrow \Conf(\X)$, that is, $\widetilde{\PP}^{(N)} = \mathfrak{i}_*\left((\mathfrak{k}_N)_*P_N\right)$. If we believe in the weak convergence $\mathbb{P}-\lim_{N\rightarrow\infty}{ (\mathfrak{k}_N)_*P_N } = P$, then we also expect that the pushforwards $\widetilde{\PP}^{(N)}$ converge in certain sense to $\PP$. In fact, the proposition below shows that the latter convergence holds in the sense that the correlation functions of $\widetilde{\PP}^{(N)}$ converge to those of $\PP$ as $N$ goes to infinity, in certain special case that will nonetheless be applicable to our situation.

Let us give a different description of the point processes $\widetilde{\PP}^{(N)}$. Recall the maps $\LL^{(N)}$ defined in Section $\ref{zprocesses}$ by
\begin{equation*}
\begin{gathered}
\LL^{(N)} : \GTp_N \rightarrow \Conf(\Zpe)\\
\lambda = (p_1, \ldots, p_d | q_1, \ldots, q_d) \mapsto \LL^{(N)}(\lambda) = \{(N + p_i + \epsilon)^2 \}_{1\leq i\leq d} \sqcup\{(N - 1 - q_i + \epsilon)^2\}_{1\leq i\leq d}.
\end{gathered}
\end{equation*}
Let $\PP^{(N)}$ be the pushforward measures $\PP^{(N)} = ( \LL^{(N)})_*P_N$. Observe that if $P_N = P_N(\cdot | z, z', a, b)$ is the $z$-measure of level $N$ associated to $z, z', a, b$, then the pushforward measure $\PP^{(N)} = (\LL^{(N)})_*P_N(\cdot | z, z', a, b)$ is the process $\LL^{(N)}_{z, z', a, b}$ already defined in Section $\ref{zprocesses}$.

Also consider the map
\begin{equation*}
\begin{gathered}
\mathfrak{j}_N : \Zpe \rightarrow \frac{1}{\left(N+\epsilon - \frac{1}{2}\right)^2}\Zpe \subset \X\\
x \mapsto \frac{x}{\left( N+\epsilon-\frac{1}{2} \right)^2}
\end{gathered}
\end{equation*}
and the induced map $\mathfrak{j}_N : \Conf(\Zpe) \rightarrow \Conf(\X)$ between spaces of point configurations. It is clear that the pushforward of $P_N$ under the composition map $\GTp_N \xrightarrow{\LL^{(N)}} \Conf(\Zpe) \xrightarrow{\mathfrak{j}_N} \Conf(\X)$ is the probability measure $\widetilde{\PP}^{(N)}$ on $\Conf(\X)$ described previously, i.e., $\widetilde{\PP}^{(N)} = (\mathfrak{j}_N)_*(\LL^{(N)})_*P_N = (\mathfrak{j}_N)_*\PP^{(N)}$.

For the special case dealt with in the statement below, assume that each $\PP^{(N)}$ is a determinantal process with correlation kernel $K^{\PP^{(N)}}$ (with respect to the counting measure on $\Zpe$). The pushforwards $\widetilde{\PP}^{(N)} = (\mathfrak{j}_N)_*\PP^{(N)}$ can be seen as point processes on the discrete space $\frac{1}{(N + \epsilon - \frac{1}{2})^2}\Zpe$, rather than on the continuous space $\X$. Consider the measure $\mu$ on $\frac{1}{(N+\epsilon-\frac{1}{2})^2}\Zpe$ that assigns a mass of $1/N$ to each point, i.e., $\mu\left( \frac{(n + \epsilon)^2}{(N + \epsilon - \frac{1}{2})^2} \right) = \frac{1}{N} \ \forall n\in\Zp$. With respect to this reference measure, the point process $\widetilde{\PP}^{(N)}$ is a determinantal process on $\frac{1}{(N+\epsilon-\frac{1}{2})^2}\Zpe$ with correlation kernel
\begin{equation*}
K^{\widetilde{\PP}^{(N)}}(x, y) = N\cdot K^{\PP^{(N)}}\left((N+\epsilon-\frac{1}{2})^2 x, (N+\epsilon-\frac{1}{2})^2 y \right).
\end{equation*}

Since all processes $\widetilde{\PP}^{(N)}$ are determinantal processes, the ``convergence'' of $\widetilde{\PP}^{(N)}$ to $\PP$ that we had in mind leads us to think that $\PP$ is also determinantal and that the correlation kernels $K^{\widetilde{\PP}^{(N)}}$ must converge to certain function $K^{\PP} : \X\times\X\rightarrow\R$, which is a correlation kernel for $\PP$.

To state the exact form of the proposition, we  need one last piece of notation. For any $x\in\X$, let $\widehat{x_N}$ be the point of $\Zpe$ that is closest to $\left((N + \epsilon - \frac{1}{2})x\right)^2$ (if there are two points in $\Zpe$ that are at the same distance from $\left((N + \epsilon - \frac{1}{2})x\right)^2$, then $\widehat{x_N}$ is either of them).

\begin{prop}\label{thm:scalinglimit}
Let $\left\{\PP^{(N)}\right\}_{N\geq 1}$ be point processes on $\Zpe$ and $\PP$ a point process on $\X$, as above. Assume that each $\PP^{(N)}$ is a determinantal process with correlation kernel $K^{\PP^{(N)}}$ (with respect to the counting measure on $\Zpe$). Let $\widetilde{K}^{\PP}:\X\times\X\rightarrow\R$ be a map such that the following conditions are satisfied:

\begin{enumerate}
\item The map $\widetilde{K}^{\PP}$ is continuous.

\item $\left. K^{\PP^{(N)}}\right|_{\Z^{\epsilon}_{<N}\times \Z^{\epsilon}_{<N}}$ and $\left. K^{\PP^{(N)}}\right|_{\Z^{\epsilon}_{\geq N}\times \Z^{\epsilon}_{\geq N}}$ are Hermitian symmetric, for each $N\geq 1$.

\item The following limit holds
\begin{equation*}
\lim_{N\rightarrow\infty}{N\cdot K^{\PP^{(N)}}(\widehat{x_N}, \widehat{y_N})} = \widetilde{K}^{\PP}(x^2, y^2)
\end{equation*}
uniformly for $(x, y)$ belonging to compact subsets of $\X\times\X$.
\end{enumerate}
Then the point process $\PP$ is determinantal and a correlation kernel of $\PP$ (with respect to the Lebesgue measure on $\X$) is
\begin{equation}\label{eqn:finalkernel}
K^{\PP}(x, y) \myeq \frac{\widetilde{K}^{\PP}(x, y)}{2x^{1/4}y^{1/4}}.
\end{equation}
\end{prop}

We do not provide here a proof of Proposition $\ref{thm:scalinglimit}$. Theorem 9.2. in \cite{BO1} is the `type A' analogue of the proposition above and its proof in \cite[Ch. 9]{BO1} can be repeated with suitable modifications, but without any new ideas. We therefore leave the proof of Proposition $\ref{thm:scalinglimit}$ as an exercise to the reader. We only mention here why the exact form in $(\ref{eqn:finalkernel})$ is the correlation kernel of $\PP$ (the appearance of $x^{1/4}$ and $y^{1/4}$ may seem strange at first).

Let $k\in\N$ be arbitrary, denote by $\rho_k^{(N)}$ and $\rho_k$ the $k$-th correlation functions of $\widetilde{\PP}^{(N)}$ and $\PP$, respectively. The process $\widetilde{\PP}^{(N)}$ is taken to be on the state space $\X\cap \frac{\Zpe}{(N+\epsilon-\frac{1}{2})^2}$ and the reference measure for its correlation functions is taken to be $\mu$ such that $\mu\left( \frac{(n+\epsilon)^2}{(N+\epsilon-\frac{1}{2})^2} \right) = 1/N$. Consequently the correlation function of $\widetilde{\PP}^{(N)}$ is
\begin{equation*}
\rho_k^{(N)}(x_1, \ldots, x_k) = \det_{1\leq i, j\leq k}{\left[K^{\widetilde{\PP}^{(N)}}(x_i, x_j)\right]} = \det_{1\leq i, j\leq k}{\left[N \cdot K^{\PP^{(N)}}\left( (N+\epsilon-\frac{1}{2})^2 x_i, (N+\epsilon-\frac{1}{2})^2 x_j \right)\right]},
\end{equation*}
by our assumptions.

The strategy of \cite[proof of Thm. 9.2]{BO1} enables us to show that $\rho_k$ exists and moreover
\begin{equation*}
\lim_{N\rightarrow\infty}{\left< F, \rho_k^{(N)} \right>} = \left< F, \rho_k \right>,
\end{equation*}
for any continuous, compactly supported function $F:\X^k \rightarrow \R$. By simple algebraic manipulations and change of variables, we have
\begin{equation*}
\begin{gathered}
\left< F, \rho_k^{(N)} \right> = \sum_{x_1, \ldots, x_k\in\X\cap\frac{\Zpe}{(N+\epsilon-\frac{1}{2})^2}}{F(x_1, \ldots, x_k)\rho_k^{(N)}(x_1, \ldots, x_k)\mu(x_1)\cdots\mu(x_k)}\\
= \frac{1}{N^k}\sum_{x_1, \ldots, x_k\in\X\cap\frac{\Zpe}{(N+\epsilon-\frac{1}{2})^2}}{F(x_1, \ldots, x_k)\det\left[ K^{\widetilde{\PP}^{(N)}}(x_i, x_j) \right]}\\
= \frac{1}{N^k}\sum_{y_1, \ldots, y_k\in\X\cap\frac{\Zp + \epsilon}{N+\epsilon-\frac{1}{2}}}{F(y_1^2, \ldots, y_k^2)\det\left[ N\cdot K^{\PP^{(N)}}\left(((N+\epsilon-\frac{1}{2})y_i)^2, ((N+\epsilon-\frac{1}{2})y_j)^2\right) \right]}\\
\xrightarrow{N\rightarrow\infty} \int\dots\int_{\X^k}{F(y_1^2, \ldots, y_k^2)\det\left[ \widetilde{K}^{\PP}(y_i^2, y_j^2) \right]dy_1\dots dy_k}\\
= \int\dots\int_{\X^k}{F(x_1, \ldots, x_k)\det\left[ \widetilde{K}^{\PP}(x_i, x_j) \right]\frac{dx_1}{2\sqrt{x_1}}\dots \frac{dx_k}{2\sqrt{x_k}}}
\end{gathered}
\end{equation*}
and finally $\det_{1\leq i, j\leq k}\left[ \widetilde{K}^{\PP}(x_i, x_j) \right] = 2^k\sqrt{x_1\cdots x_k}\det_{1\leq i, j\leq k}\left[ K^{\PP}(x_i, x_j) \right]$.\\

To conclude this section, we apply Proposition $\ref{thm:scalinglimit}$ to the coherent sequence of $z$-measures $\{P_N\}_{N \geq 1}$, see Proposition $\ref{prop:coherent}$, to obtain the desired stochastic point processes $\PP$ on $\X = \R_{>0}\setminus\{1\}$, depending on parameters $z, z', a, b$.

\begin{df}
Let $a\geq b\geq -\frac{1}{2}$, $(z, z')\in\U_{\adm}$ and $\{P_N (\cdot | z, z', a, b)\}_{N\geq 0}$ be the associated $z$-measures of finite levels. Let $P_{z, z', a, b}$ be the associated element of $M_{\textrm{prob}}(\Tau)$, given by the bijective map $(\ref{eqn:bijection})$ above. Denote by $\pi_{z, z', a, b}$ to the spectral measures on $\Omega_{\infty}$ corresponding to $P_{z, z', a, b}$, and given by Theorem $\ref{thm:spectral}$. Then we define $\PP_{z, z', a, b}$ as the pushforward of $\pi_{z, z', a, b}$ with respect to the map $\mathfrak{i}:\Omega_{\infty}\rightarrow\Conf(\X)$ in $(\ref{eqn:mapi})$. If $z, z', a, b$ are implicit in the context, then we denote $\PP_{z, z', a, b}$ simply by $\PP$.
\end{df}

In the next section, we show that the assumptions of Proposition $\ref{thm:scalinglimit}$ are satisfied (with $\PP^{(N)}$ replaced by $\LL^{(N)}_{z, z', a, b}$ and $\PP$ replaced by $\PP_{z, z', a, b}$). As a consequence, it will follow that $\PP_{z, z', a, b}$ is determinantal and we are able to compute a correlation kernel of it as a limit of the correlation kernels $K^{\LL^{(N)}}$.

\section{Main result: An explicit correlation kernel for the point process $\PP$}\label{sec:maintheoremkernel}

\subsection{Statement of the main theorem}

The main theorem of this paper is the following.

\begin{thm}\label{maintheorem}
Let $a \geq b \geq -\frac{1}{2}$, $(z, z')\in\U_{\adm}$, and $\PP = \PP_{z, z', a, b}$ be the corresponding point process on $\R_{>0} \setminus \{1\}$. Then $\PP$ is determinantal and a correlation kernel $K^{\PP}(x, y)$ can be explicitly given in terms of hypergeometric functions. Explicitly, in terms of the decomposition $\R_{>0} \setminus \{1\} = \X_{>1}\sqcup\X_{<1},\ \X_{>1} = (1, \infty), \ \X_{<1} = (0, 1)$, $K^{\PP}(x, y)$ can be expressed as
\begin{equation}\label{KernelP}
K^{\PP}(x, y) = \left\{
\begin{aligned}
    \sqrt{\psi_{>1}(x)\psi_{>1}(y)}\cdot\frac{R_{>1}(x)S_{>1}(y) - S_{>1}(x)R_{>1}(y)}{x - y},& \textrm{ if } x, y \in\X_{>1}, \ x\neq y,\\
    \sqrt{\psi_{>1}(x)\psi_{<1}(y)}\cdot\frac{R_{>1}(x)R_{<1}(y) - S_{>1}(x)S_{<1}(y)}{x - y},& \textrm{ if } x\in\X_{>1},\ y \in\X_{< 1},\\
    \sqrt{\psi_{<1}(x)\psi_{>1}(y)}\cdot\frac{R_{<1}(x)R_{>1}(y) - S_{<1}(x)S_{>1}(y)}{x - y},& \textrm{ if } x\in\X_{<1},\ y \in\X_{>1},\\
    \sqrt{\psi_{<1}(x)\psi_{<1}(y)}\cdot\frac{R_{<1}(x)S_{<1}(y) - S_{<1}(x)R_{<1}(y)}{x - y}, &\textrm{ if } x, y\in\X_{<1},\ x\neq y,
\end{aligned}\right.
\end{equation}
where $\psi_{>1} : (1, \infty)\rightarrow \R_{\geq 0}$, $\psi_{<1} : (0, 1)\rightarrow \R_{\geq 0}$ are defined by
\begin{eqnarray*}
\psi_{>1}(x) &=& \frac{\sin(\pi z)\sin(\pi z')}{2\pi^2}x^{- b}(x - 1)^{-z-z'},\\
\psi_{<1}(x) &=& 2x^{b}(1 - x)^{z+z'},
\end{eqnarray*}
and $R_{>1}, S_{>1} : (1, \infty) \rightarrow \C$, $R_{<1}, S_{<1} : (0, 1) \rightarrow \C$, are defined as follows
\begin{eqnarray*}
R_{>1}(x) &=& \left(1 - \frac{1}{x}\right)^{z'}\pFq{2}{1}{z' + b ,,, z'}{z + z' + b}{\frac{1}{x}},\\
S_{>1}(x) &=& \frac{2}{x}\left(1 - \frac{1}{x}\right)^{z'}\Gamma\left[ \begin{split}
z+1,\ z'+1,\ z+b+1,\ z'+b+1\\
z+z'+b+1,\hspace{.3in} z+z'+b+2
\vphantom{\frac12}\end{split}
\right]\pFq{2}{1}{z' + b + 1 ,, z' + 1}{z + z' + b + 2}{\frac{1}{x}},\\
R_{<1}(x) &=& -\frac{\sin \pi z}{\pi}\cdot\Gamma\left[ \begin{split}
z'-z ,\ z+b+1 ,\ z+1\\
z+z'+b+1 \hspace{.2in}
\vphantom{\frac12}
\end{split}\right]\left( 1 - x \right)^{-z'}\pFq{2}{1}{-z' ,,, z+b+1}{1+z-z'}{1 - x}\\
&&-\frac{\sin \pi z'}{\pi}\cdot\Gamma\left[ \begin{split}
z-z' ,\ z'+b+1 ,\ z'+1\\
 z+z'+b+1 \hspace{.3in}
\vphantom{\frac12}
\end{split}\right]( 1 - x )^{-z}\pFq{2}{1}{-z ,,, z'+b+1}{1+z'-z}{1 - x},\\
S_{<1}(x) &=& -\frac{\sin \pi z}{2\pi}\cdot\Gamma\left[ \begin{split}
z'-z &,& z+z'+b\\
z' &,& z'+b
\vphantom{\frac12}
\end{split}\right]\left( 1 - x \right)^{-z'}\pFq{2}{1}{1-z' ,,, z+b}{1+z-z'}{1 - x}\\
&&-\frac{\sin \pi z'}{2\pi}\cdot\Gamma\left[ \begin{split}
z-z' &,& z+z'+b\\
z &,& z+b
\vphantom{\frac12}
\end{split}\right]\left( 1 - x \right)^{-z}\pFq{2}{1}{1-z ,,, z'+b}{1+z'-z}{1 - x}.
\end{eqnarray*}

Observe that the expressions for the kernel $K^{\PP}$ in $(\ref{KernelP})$ have poles in the diagonal $x = y$. The kernel admits the following analytic continuation via L'H\^opital's rule:
\begin{equation}\label{KernelP2}
K^{\PP}(x, x) = \left\{
\begin{aligned}
    \psi_{>1}(x)\cdot \left(R_{>1}'(x)S_{>1}(x) - S_{>1}'(x)R_{>1}(x)\right),& \textrm{ if } x \in\X_{>1},\\
    \psi_{<1}(x)\cdot \left(R_{<1}'(x)S_{<1}(x) - S_{<1}'(x)R_{<1}(x)\right), &\textrm{ if } x\in\X_{<1}.
\end{aligned}\right.
\end{equation}
\end{thm}


\begin{rem}
The function $R_{>1}$ is holomorphic (as function of $(z, z')$) on the domain $\U\cap\{\Sigma \neq 0\}$, whereas $R_{<1}$ is holomorphic on $\U_0\cap\{ z - z' \notin \Z \}$. The function $S_{>1}$ is holomorphic on $\U_0$, whereas $S_{<1}$ is holomorphic on $\U\cap\{z-z' \notin \Z\}\cap\{\Sigma \neq 0\}$. Then the functions $R_{>1}, R_{<1}, S_{>1}, S_{<1}$ (and consequently $K^{\PP}(x, y)$) are all holomorphic on $\U_0\cap\{\Sigma \neq 0\}\cap\{z-z' \notin \Z\}$. In the proof below, for any $x, y\in\X$, it is shown that $K^{\PP}(x, y)$ as a function of $(z, z')$ can be analytically continued to the domain $\U_0$ and therefore $K^{\PP}(x, y)$ has a well-defined value for any $(z, z')\in\U_{\adm}$, even when $\Sigma = 0$ and when $z-z'\in\Z$.
\end{rem}

\begin{rem}
The kernel $K^{\PP}$ in Theorem $\ref{maintheorem}$ is known as the \textit{hypergeometric kernel} and it appeared first in \cite{BO1}. We make the exact comparison with the formula of \cite{BO1} in the next remark. Observe the surprising fact that $K^{\PP}$ is independent of the parameter $a$. This phenomenon also occurs in a different context where some natural Markov processes on $\Omega_{\infty}$ preserving the $z$-measures lose the parameter $a$ in a limit transition, see \cite{C}.
\end{rem}

\begin{rem}
Consider the functions $\overline{\psi}_{>1}(x) = 2\psi_{>1}(x)$, $\overline{\psi}_{<1}(x) = \psi_{<1}(x)/2$, $\overline{R}_{>1} = R_{>1}$, $\overline{R}_{<1} = R_{<1}$, $\overline{S}_{>1} = S_{>1}/2$, $\overline{S}_{<1} = 2S_{<1}$ and the kernel $\overline{K}^{\PP}$ defined just like $K^{\PP}$, but with the functions $\overline{\psi}_{>1}, \ldots$ (with a bar over it) instead of the functions $\psi_{>1}, \ldots$. It is clear that $\overline{K}^{\PP}(x, y) = K^{\PP}(x, y)\ \forall x, y\in\X$.

By virtue of the identity, see \cite[2.4.(1)]{B},
\begin{equation*}
(1-z)^{-A}\pFq{2}{1}{A ,,, B}{C}{-z/(1-z)} = \pFq{2}{1}{A ,,, C - B}{C}{z},
\end{equation*}
the formulas for $\overline{R}_{>1}, \overline{S}_{>1}$ have alternative expressions. For example,
\begin{eqnarray*}
\overline{R}_{>1}(x) &=& \left( 1 - \frac{1}{x} \right)^{-b} \pFq{2}{1}{z+b ,,, z'+b}{z + z' + b}{\frac{1}{1-x}}.
\end{eqnarray*}
With the new expressions for $\overline{R}_{>1}, \overline{S}_{>1}$, it is clear that the functions $\overline{\psi}_{>1}, \overline{\psi}_{<1}, \overline{R}_{>1}, \overline{R}_{<1}, \overline{S}_{>1}, \overline{S}_{<1}$ are very similar to the ones appearing in the hypergeometric kernel describing the $zw$-measures, see \cite[Thm. 10.1]{BO2}. In fact, they are actually the \textit{same} functions after the formal identification of variables
\begin{eqnarray*}
\textrm{Hypergeometric kernel in \cite{BO1}}&\leftrightarrow&\textrm{Hypergeometric kernel above}\\
z &\leftrightarrow&z\\
z' &\leftrightarrow&z'\\
w &\leftrightarrow&0\\
w' &\leftrightarrow&b\\
x &\leftrightarrow& x - \frac{1}{2}.
\end{eqnarray*}
A pedrestian explanation for the similarity of the functions $R_{>1}, S_{>1}$ is that Askey-Lesky polynomials (which come up in the analysis of $zw$-measures) and Wilson-Neretin polynomials have the same asymptotics in certain limit regime. Moreover the similarity of the functions $R_{<1}, S_{<1}$, at least when $\Re \Sigma > 0$, comes from the fact that they can be expressed as certain Stieltjes-like transforms of $R_{>1}, S_{>1}$ if one takes a suitable limit of the expressions $(\ref{R2S2})$ and \cite[(8.3)]{BO2}.

Another (informal) explanation that we reproduce here is due to Alexei Borodin and Grigori Olshanski. If we consider the kernel $K^{\PP}$ above as  an operator on $H = L^2(\R_{>0} \setminus \{1\}, dx) = L^2((0, 1), dx)\oplus L^2((1, \infty), dx)$, then it is plausible that we can write it in the form $K^{\PP} = L(1 + L)^{-1}$, for some operator $L$ on $H$, with matrix of the form
\begin{equation*}
L = \begin{bmatrix}
    0 & A\\
    -A^* & 0
\end{bmatrix}.
\end{equation*}
(See \cite{Ol2} for a proof in an easier model coming from the infinite symmetric group.) Moreover, since $K^{\LL^{(N)}} = L_N(1 + L_N)^{-1}$ and $K^{\LL^{(N)}}$ converges to $K^{\PP}$, it is plausible that the desired operator $L$ is a limit of the operators $L_N$, as $N$ tends to infinity. Since the matrix elements of $L_N$ involve only elementary functions, such a limit transition is very simple. Similar remarks apply to the `type A' situation, especially the fact that the corresponding matrix elements of $L_N$ (in their discrete point processes) involve only elementary functions and they have a simple limit transition. Thus the presence of the hypergeometric kernel in both situations can be attributed to the equality of limits of the $L_N$ operators on the corresponding discrete spaces.
\end{rem}

\subsection{Limit of the kernel $K^{\LL^{(N)}}$ as $N$ tends to infinity}

Given Proposition $\ref{thm:scalinglimit}$ and the expressions for $K^{\PP}(x, y)$ in the statement of Theorem $\ref{maintheorem}$, the difficult part is in proving the uniform convergence on compact subsets, $\lim_{N\rightarrow\infty}{N\cdot K^{\LL^{(N)}}}(\widehat{x_N}, \widehat{y_N}) = (2x^{1/4}y^{1/4})\cdot K^{\PP}(x, y)$, which is the second item of Proposition $\ref{thm:scalinglimit}$.

Let us state three propositions which show that each of the functions $\psi_{\geq N}, \psi_{<N}, R_{\geq N}, R_{<N}, S_{\geq N}, S_{<N}$ converges, under suitable normalization, uniformly on compact subsets to the corresponding functions $\psi_{>1}, \psi_{<1}, R_{>1}, R_{<1}, S_{>1}, S_{<1}$.

\begin{prop}\label{psilimit}
Let $(z, z')\in\C^2$. The following limits
$$\lim_{N\rightarrow\infty}{N^{2\Sigma - 1}\psi_{\geq N}(\widehat{x_N})} = 2x\cdot\psi_{>1}(x^2), \hspace{.1in}\lim_{N\rightarrow\infty}{N^{- 2\Sigma-1}\psi_{<N}(\widehat{x_N})} = 2x\cdot\psi_{<1}(x^2),$$
hold uniformly for $x$ in compact subsets of $(1, \infty)$ and $(0, 1)$ respectively, where $\widehat{x_N}$ is the element of $\Zpe$ that is closest to $\left((N + \epsilon - \frac{1}{2})x\right)^2$ (if there are two elements of $\Zpe$ equally close to $\left((N + \epsilon - \frac{1}{2})x\right)^2$ then $\widehat{x_N}$ is either of them). The functions $\psi_{\geq N}, \psi_{<N}$ are as defined in $(\ref{eqn:psigreaterN})$ and $(\ref{eqn:psismallerN})$.
\end{prop}

\begin{prop}\label{RS1limit}
Let $(z, z')\in\U_0\cap\{\Sigma\neq 0\}$. Let $I\subset (1, \infty)$ be any compact subset. For any $\delta > 0$, define
\begin{equation*}
I_{\delta} = \{\zeta\in\C  : \Re\zeta\in I, \ |\Im\zeta| \leq \delta\}.
\end{equation*}
Then there exists $\delta_0 > 0$ small enough such that for any $\zeta\in I_{\delta_0}$, we have
\begin{equation*}
\lim_{N\rightarrow\infty}{\widetilde{R}_{\geq N}(N\zeta)} = R_{>1}(\zeta^2), \hspace{.1in}\lim_{N\rightarrow\infty}{N^{-2\Sigma}\TS_{\geq N}(N\zeta)} = S_{>1}(\zeta^2),
\end{equation*}
and the convergence is uniform on $I_{\delta_0}$. The functions $\TR_{\geq N}, \TS_{\geq N}$ are defined in $(\ref{RlargerN})$ and $(\ref{SlargerN})$.
\end{prop}

\begin{prop}\label{RS2limit}
Let $(z, z')\in\U_0\cap\{\Sigma \neq 0\}\cap\{z-z'\notin\Z\}$. Let $J\subset (0, 1)$ be any compact subset. For any $\delta > 0$, define
\begin{equation*}
J_{\delta} = \{\zeta\in\C  : \Re\zeta\in I, \ |\Im\zeta| \leq \delta\}.
\end{equation*}
Then there exists $\delta_0 > 0$ small enough such that for any $\zeta\in J_{\delta_0}$, we have
\begin{equation*}
\lim_{N\rightarrow\infty}{\widetilde{R}_{<N}(N\zeta)} = R_{<1}(\zeta^2), \hspace{.1in}\lim_{N\rightarrow\infty}{N^{2\Sigma}\TS_{<N}(N\zeta)} = S_{<1}(\zeta^2),
\end{equation*}
and the convergence is uniform on $J_{\delta_0}$. The functions $\TR_{<N}, \TS_{<N}$ are given in $(\ref{tildeRsmallerN})$ and $(\ref{tildeSsmallerN})$.
\end{prop}

\begin{proof}[Proof of Proposition \ref{psilimit}]
Rewrite the expression for $\psi_{\geq N}(\widehat{x_N})$, coming from $(\ref{eqn:psigreaterN})$, as
\begin{equation*}
\begin{gathered}
\psi_{\geq N}(\widehat{x_N}) = \frac{\sin\pi(z-x_N+N)\sin\pi(z'-x_N+N)}{\pi^2}\cdot(x_N + \epsilon)\cdot\frac{\Gamma(x_N + N + 2\epsilon)^2}{\Gamma(x_N - N + 1)^2}\\
\times\Gamma\left[ \begin{split}
x_N + 1 &,& x_N + a + 1 &,& x_N-z-N+1 &,& x_N-z'-N+1\\
x_N + 2\epsilon &,& x_N + b + 1 &,& x_N+z+N+2\epsilon &,& x_N+z'+N+2\epsilon
\vphantom{\frac12}
\end{split}
\right].
\end{gathered}
\end{equation*}
Since $x_N\in\Zp$, it is clear the constant factor in front of $\psi_{\geq N}(\widehat{x_N})$ is $\sin(\pi z)\sin(\pi z')/\pi^2$. Finally the result for $\psi_{\geq N}$ follows from the following uniform estimates as $N$ tends to infinity, when $x > 1$:
\begin{equation*}
\begin{gathered}
\frac{\Gamma(x_N + 1)}{\Gamma(x_N + b + 1)} = N^{-b}x^{-b}(1 + O(N^{-1})) ,\hspace{.1in} \frac{\Gamma(x_N + a + 1)}{\Gamma(x_N + 2\epsilon)} = N^{-b}x^{-b}(1 + O(N^{-1}))\\
x_N + \epsilon = Nx(1 + O(N^{-1}))\\
\frac{\Gamma(x_N + N + 2\epsilon)}{\Gamma(x_N + u + N + 2\epsilon)} = N^{-u}(x + 1)^{-u}(1 + O(N^{-1})), \textrm{ for }u = z, z',\\
\frac{\Gamma(x_N -  u - N + 1)}{\Gamma(x_N - N + 1)} = N^{-u}(x - 1)^{-u}(1 + O(N^{-1})), \textrm{ for }u = z, z'.
\end{gathered}
\end{equation*}
The expression $(\ref{eqn:psismallerN})$ for $\psi_{<N}(\widehat{x_N})$ can be rewritten as
\begin{equation*}
\begin{gathered}
\psi_{<N}(\widehat{x_N}) = 4(x_N + \epsilon)\cdot\Gamma\left[ \begin{split}
x_N + 2\epsilon &,& x_N + b + 1\\
x_N + 1 &,& x_N + a + 1
\vphantom{\frac12}
\end{split}
\right]\cdot\frac{1}{\Gamma(N - x_N)^2\Gamma(x_N + N + 2\epsilon)^2}\\
\times\Gamma(z - x_N + N)\Gamma(z' - x_N + N)\Gamma(z + x_N + N + 2\epsilon)\Gamma(z' +x_N + N + 2\epsilon).
\end{gathered}
\end{equation*}
The result for $\psi_{<N}$ then follows from the following uniform estimates as $N\rightarrow\infty$  when $0<x<1$:
\begin{equation*}
\begin{gathered}
\frac{\Gamma(x_N + b + 1)}{\Gamma(x_N + 1)} = N^{b}x^{b}(1 + O(N^{-1})) , \hspace{.1in} \frac{\Gamma(x_N + 2\epsilon)}{\Gamma(x_N + a + 1)} = N^{b}x^{b}(1 + O(N^{-1})),\\
x_N + \epsilon = Nx(1 + O(N^{-1})),\\
\frac{\Gamma(u + x_N + N + 2\epsilon)}{\Gamma(x_N + N + 2\epsilon)} = N^u(x+1)^u(1 + O(N^{-1})),  \textrm{ for }u = z, z',\\
\frac{\Gamma(u - x_N + N)}{\Gamma(N - x_N)} = N^u(1 - x)^u(1 + O(N^{-1})),  \textrm{ for }u = z, z'.
\end{gathered}
\end{equation*}
\end{proof}

As a preliminary to the proof of Proposition $\ref{RS1limit}$, we show the following lemma.

\begin{lem}\label{limithypergeom1}
Let $\{A^{(1)}_N\}_{N \geq 1}, \{B^{(1)}_N\}_{N \geq 1},\{ B^{(2)}_N\}_{N \geq 1}$ be sequences of complex numbers and let $A^{(2)}, A^{(3)}, B^{(3)}$ be complex numbers, independent of $N$, such that $B^{(3)} \notin \Z_{\leq 0}$. Assume that the sequences satisfy
\begin{equation}\label{eq:sequences}
\begin{gathered}
\lim_{N\rightarrow\infty}{\frac{NA^{(1)}_N}{B^{(1)}_NB^{(2)}_N}} = \eta\in\{z\in\C : |z| < 1\}, \hspace{.1in}\lim_{N\rightarrow\infty}{\left|A_N^{(1)}\right|} = \infty,\\
\textrm{for large enough $N$: }\Re B_N^{(1)} > Ns \textrm{ for some fixed }s>1, \hspace{.1in} \Re B_N^{(2)} > \left| A_N^{(1)} \right|,\\
-B_N^{(1)} + B_N^{(2)} + B^{(3)} + N - A_N^{(1)} - A^{(2)} - A^{(3)} = 1.
\end{gathered}
\end{equation}
Then
\begin{equation*}
\lim_{N\rightarrow\infty}{\pFq{4}{3}{-N ,,, A^{(1)}_N ,,, A^{(2)} ,,, A^{(3)}}{-B^{(1)}_N ,,, B^{(2)}_N ,,, B^{(3)}}{1}} = \pFq{2}{1}{A^{(2)} ,,, A^{(3)}}{B^{(3)}}{\eta}.
\end{equation*}
Moreover the limit above is uniform over sequences $\{A_N^{(1)}\}, \{B_N^{(1)}\}, \{B_N^{(2)}\}$ such that they satisfy the conditions $(\ref{eq:sequences})$ above and $\left\{\frac{NA_N^{(1)}}{B_N^{(1)}B_N^{(2)}} - \lim_{N\rightarrow\infty}{ \frac{NA_N^{(1)}}{B_N^{(1)}B_N^{(2)}} }\right\}$ converges uniformly to $0$ as $N$ tends to infinity.
\end{lem}
\begin{proof}
The power series expansion of the generalized hypergeometric function is
\begin{equation}\label{eq:4F3expansion}
\pFq{4}{3}{-N ,,, A^{(1)}_N ,,, A^{(2)} ,,, A^{(3)}}{-B^{(1)}_N ,,, B^{(2)}_N ,,, B^{(3)}}{1} = \sum_{k=0}^N{\frac{(-N)_k(A_N^{(1)})_k(A^{(2)})_k(A^{(3)})_k}{(-B_N^{(1)})_k(B_N^{(2)})_k(B^{(3)})_kk!}}.
\end{equation}
The $k$-th term in the sum above can be written as
\begin{equation*}
\frac{(A^{(2)})_k(A^{(3)})_k}{(B^{(3)})_k k!}\left( \frac{NA_N^{(1)}}{B_N^{(1)}B_N^{(2)}} \right)^k\cdot\prod_{i=0}^{k-1}{\frac{(1 - i/N)(1 + i/A_N^{(1)})}{(1 - i/B_N^{(1)})(1 + i/B_N^{(2)})}}
\end{equation*}
The conditions $(\ref{eq:sequences})$ guarantee that $\lim_{N\rightarrow\infty}{\left|A_N^{(1)}\right|} = \lim_{N\rightarrow\infty}{\left|B_N^{(1)}\right|} = \lim_{N\rightarrow\infty}{\left|B_N^{(2)}\right|} = \infty$, so it is clear now that the $k$-th term above ($k$ does not depend on $N$) converges to the $k$-th term of the power series expansion of the hypergeometric function
\begin{equation*}
\pFq{2}{1}{A^{(2)} ,,, A^{(3)}}{B^{(3)}}{\eta} = \sum_{k=0}^{\infty}{\frac{(A^{(2)})_k(A^{(3)})_k}{(B^{(3)})_k k!} \eta^k}.
\end{equation*}
To guarantee that the limit can be taken term-by-term, we need to use the dominated convergence theorem. For large enough $N$, and any $0\leq k < N$, we have the estimates
\begin{equation*}
\begin{gathered}
\left|\frac{(-N)_k(A_N^{(1)})_k}{(-B_N^{(1)})_k(B_N^{(2)})_k}\right| \leq \frac{N(N-1)\cdots(N-k+1)|A_N^{(1)}|(|A_N^{(1)}| + 1)\cdots (|A_N^{(1)}| + k - 1)}{(\Re B_N^{(1)})\cdots (\Re B_N^{(1)} - k + 1)(\Re B_N^{(2)})(\Re B_N^{(2)} + 1)\cdots(\Re B_N^{(2)} + k - 1)}\\
\leq \frac{N(N-1)\cdots(N-k+1)}{(\Re B_N^{(1)})(\Re B_N^{(1)} - 1)\cdots (\Re B_N^{(1)} - k + 1)}\\
\leq \left( \frac{N}{\Re B_N^{(1)}} \right)^k < s^{-k},
\end{gathered}
\end{equation*}
where the inequality between the first and second lines assumes $\Re B_N^{(2)} > \left| A_N^{(1)} \right|$, the one between lines two and three is obvious, and the inequality in the third line assumes $\Re B_N^{(1)} > Ns$. Therefore we have bounded the absolute value of the $k$-th term in the sum $(\ref{eq:4F3expansion})$ above by the $k$-th term in the following series which does not depend on $N$:
\begin{equation}\label{eq:2F1sum}
\sum_{k=0}^{\infty}{\left|\frac{(|A^{(2)}|)_k (|A^{(3)}|)_k}{(\Re B^{(3)})_k}\right|\frac{s^{-k}}{k!}} = \pFq{2}{1}{|A^{(2)}| ,,, |A^{(3)}|}{\Re B^{(3)}}{s^{-1}}.
\end{equation}
Since $(\ref{eq:2F1sum})$ is convergent, we conclude that we can take the limit of the $_4F_3$ terminating series expansion term-by-term, so the limit in the statement of Lemma $\ref{limithypergeom1}$ is proved. To show the uniform convergence, observe that we have bounded the series of absolute values corresponding to $(\ref{eq:4F3expansion})$ by the same convergent series $(\ref{eq:2F1sum})$, as long as the sequences $\{A_N^{(1)}, B_N^{(1)}, B_N^{(2)}\}$ satisfy the conditions stated in the Lemma. Along with the uniform convergence of $\left\{\frac{NA_N^{(1)}}{B_N^{(1)}B_N^{(2)}} - \lim_{N\rightarrow\infty}{ \frac{NA_N^{(1)}}{B_N^{(1)}B_N^{(2)}} }\right\}$, the final statement is clear.
\end{proof}

\begin{proof}[Proof of Proposition $\ref{RS1limit}$]
We need the following new formulas for the functions $\widetilde{R}_{\geq N}$ and $\widetilde{S}_{\geq N}$.

\begin{equation}\label{eq:RgeqN}
\begin{gathered}
\widetilde{R}_{\geq N}(N\zeta) = \Gamma\left[ \begin{split}
N\zeta + 1 - z' &,& N\zeta + z' + N + 2\epsilon &,& N\zeta-N+1 &,& N\zeta + 2\epsilon\\
N\zeta-N+1-z' &,& N\zeta + z' + 2\epsilon &,& N\zeta+1 &,& N\zeta + N + 2\epsilon
\vphantom{\frac12}
\end{split}
\right]\\
\times\pFq{4}{3}{-N,, N + \Sigma + a ,, z' + b ,, z'}{-N\zeta+z' ,, N\zeta + z' + 2\epsilon ,, \Sigma}{1},
\end{gathered}
\end{equation}

\begin{equation}\label{eq:SgeqN}
\begin{gathered}
\widetilde{S}_{\geq N}(N\zeta) = 2\cdot\Gamma\left[ \begin{split}
1+z &,& 1+z' &,& 1+z+b &,& 1+z'+b\\
&& \Sigma+1 &,& \Sigma+2
\vphantom{\frac12}\end{split}
\right]\\
\Gamma\left[ \begin{split}
N+\Sigma+a+1,\hspace{.1in} N+\Sigma+1,\hspace{.1in} N\zeta - z',\hspace{.1in} N\zeta + N + z' + 2\epsilon\\
N+a,\hspace{.1in} N,\hspace{.1in} N\zeta - N - z' + 1,\hspace{.1in} N\zeta + z' + 2\epsilon + 1
\vphantom{\frac12}
\end{split}
\right]\\
\times\Gamma\left[ \begin{split}
N\zeta - N + 1 &,& N\zeta + 2\epsilon\\
N\zeta + 1 &,& N\zeta + N + 2\epsilon
\vphantom{\frac12}
\end{split}
\right]\times\pFq{4}{3}{-N+1,, N+\Sigma+a+1 ,, z'+b+1  ,, z'+1}{1-N\zeta+z' ,, N\zeta+z'+2\epsilon+1 ,, \Sigma+2}{1}.
\end{gathered}
\end{equation}

The proofs of the formulas above are given in the Appendix. Let us next study the limit of the functions $\widetilde{R}_{\geq N}(N\zeta)$. The Gamma functions are easy to handle: for $\Re\zeta > 1$ we have
\begin{equation*}
\begin{gathered}
\frac{\Gamma(N\zeta + 1 - z')}{\Gamma(N\zeta + 1)} = (N\zeta)^{-z'}(1 + O(1/N)), \ \frac{\Gamma(N\zeta + 2\epsilon)}{\Gamma(N\zeta + z' + 2\epsilon)} = (N\zeta)^{-z'}(1 + O(1/N)),\\
\frac{\Gamma(N\zeta + z' + N + 2\epsilon)}{\Gamma(N\zeta + N + 2\epsilon)} = (N(\zeta + 1))^{z'}(1 + O(1/N)), \ \frac{\Gamma(N\zeta - N + 1)}{\Gamma(N\zeta - N + 1 - z')} = (N(\zeta - 1))^{z'}(1 + O(1/N)),
\end{gathered}
\end{equation*}
and all estimates are uniform for $\zeta$ belonging to compact subsets of $\{z\in\C : \Re z > 1\}$.

Next we can apply Lemma $\ref{limithypergeom1}$ above to the sequences $A_N^{(1)} = N + \Sigma + a, B_N^{(1)} = N\zeta - z', B_N^{(2)} = N\zeta + z' + 2\epsilon$ and the complex numbers $A^{(2)} = z' + b, A^{(3)} = z', B^{(3)} = \Sigma$ to obtain the limit
\begin{equation*}
\lim_{N\rightarrow\infty}{\pFq{4}{3}{-N,, N + \Sigma + a ,, z' + b ,, z'}{-N\zeta+z' ,, N\zeta + z' + 2\epsilon ,, \Sigma}{1}} = \pFq{2}{1}{z' + b ,, z'}{\Sigma}{\frac{1}{\zeta^2}}.
\end{equation*}
The lemma moreover implies that the limit above is uniform for $\zeta$ in any compact subset $K\subset\{z\in\C : \Re z > 1\}$ since it is clear that the conditions imposed in the lemma are satisfied, say for $s = \frac{1 + \inf_{\zeta\in K}{\Re\zeta}}{2}$.

Thus we conclude that $\lim_{N\rightarrow\infty}{\widetilde{R}_{\geq N}(N\zeta)} = R_{>1}(\zeta^2)$, uniformly for $\zeta\in I_{\delta_0}$ and any $\delta_0 > 0$. A similar analysis easily shows the second desired limit $\lim_{N\rightarrow\infty}{N^{-2\Sigma}\widetilde{S}_{\geq N}(N\zeta)} = S_{>1}(\zeta^2)$, uniformly for $\zeta\in I_{\delta_0}$ and any $\delta_0 > 0$.
\end{proof}

\begin{proof}[Proof of Proposition $\ref{RS2limit}$]
Let us begin with the first limit in the proposition. The prelimit expression $\widetilde{R}_{<N}(N\zeta)$ is a sum of two terms, and both terms differ from each other by the swap of variables $z\leftrightarrow z'$, so we only need to deal with one of them. The first term is
\begin{equation*}
\begin{gathered}
-\frac{\sin\pi z}{\pi}\cdot\Gamma\left[ \begin{split}
z' - z &,& z+b+1 &,& z+1\\
&& \Sigma+1
\vphantom{\frac12}\end{split}\right]\times\Gamma\left[ \begin{split}
N+\Sigma+a+1 &,& N+\Sigma+1\\
N+z+b+1 &,& N+z+2\epsilon
\vphantom{\frac12}\end{split}\right]\\
\times\Gamma\left[ \begin{split}
N(1+\zeta)+2\epsilon &,& N(1-\zeta)\\
N(1+\zeta)+z'+2\epsilon &,& N(1-\zeta)+z'
\vphantom{\frac12}\end{split}\right]\pFq{4}{3}{1+z+b ,, -z' ,, N(1+\zeta)+z+2\epsilon ,, N(1-\zeta)+z}{1+z-z' ,,, N+z+b+1 ,,, N+z+2\epsilon}{1}.
\end{gathered}
\end{equation*}
 We show that it has a uniform limit to the first term of $R_{<1}(\zeta^2)$, which is
\begin{equation*}
\begin{gathered}
-\frac{\sin\pi z}{\pi}\cdot\Gamma\left[ \begin{split}
z' - z &,& z+b+1 &,& z+1\\
&& \Sigma+1
\vphantom{\frac12}\end{split}\right]\times (1 - \zeta^2)^{-z'}\times \pFq{2}{1}{1+z+b ,, -z'}{1+z-z'}{1 - \zeta^2}.
\end{gathered}
\end{equation*}
By the asymptotics of Gamma functions, we have
\begin{equation*}
\begin{gathered}
\Gamma\left[ \begin{split}
N+\Sigma+a+1 &,& N+\Sigma+1\\
N+z+b+1 &,& N+z+2\epsilon
\vphantom{\frac12}\end{split}\right] = N^{2z'}(1 + O(1/N)),\\
\frac{\Gamma(N(1+\zeta)+2\epsilon)}{\Gamma(N(1+\zeta) + z'+2\epsilon)} = (N(1+\zeta))^{-z'}(1 + O(1/N)), \ \frac{\Gamma(N(1-\zeta))}{\Gamma(N(1-\zeta) + z')} = (N(1-\zeta))^{-z'}(1 + O(1/N)).
\end{gathered}
\end{equation*}
The limits above are all uniform for $\zeta$ belonging to $J_{\delta_0}$, for any $\delta_0 > 0$. We are left to prove the uniform limit
\begin{equation}\label{eqn:limithypergeos}
\lim_{N\rightarrow\infty}{ \pFq{4}{3}{1+z+b ,, -z' ,, N(1+\zeta)+z+2\epsilon ,, N(1-\zeta)+z}{1+z-z' ,,, N+z+b+1 ,,, N+z+2\epsilon}{1} } = \pFq{2}{1}{1+z+b ,, -z'}{1+z-z'}{1 - \zeta^2}
\end{equation}
for $\zeta$ belonging to $J_{\delta_0}$ and some small $\delta_0 > 0$.
The $_4F_3$ generalized hypergeometric function above can be written as the absolutely convergent power series
\begin{equation}\label{eqn:powerseries4F3}
\sum_{k=0}^{\infty}{\frac{(1+z+b)_k(-z')_k(N(1+\zeta) + z + 2\epsilon)_k(N(1 - \zeta) + z)_k}{(1 + z - z')_k(N + z + b + 1)_k(N+z+2\epsilon)_k k!}}.
\end{equation}
For any $k\in\N$ fixed, we have the evident limit
\begin{equation}\label{eqn:productk}
\frac{(N(1+\zeta)+z+2\epsilon)_k (N(1 - \zeta) + z)_k}{(N + z + b + 1)_k(N + z + 2\epsilon)_k} = \prod_{i=0}^{k-1}{\frac{\left( 1 + \zeta + \frac{z + 2\epsilon + i}{N} \right)\left( 1 - \zeta + \frac{z + i}{N} \right) }{\left( 1 + \frac{z + b + 1 + i}{N} \right) \left( 1 + \frac{z + 2\epsilon + i}{N} \right)}} \xrightarrow{N\rightarrow\infty} (1 - \zeta^2)^k.
\end{equation}
Therefore, for each $k\in\N$, the $k$-th term of the series $(\ref{eqn:powerseries4F3})$ converges to the $k$-th term of the power series expansion of $_2F_1(1+z+b, \ -z';\ 1 + z - z';\ 1 - \zeta^2)$. Thus the limit $(\ref{eqn:limithypergeos})$ holds term-by-term in their power series expansion. To complete the proof, we make use of the dominated convergence theorem, but we need some uniform estimates of the terms in the series $(\ref{eqn:powerseries4F3})$.

For any $i\geq 0$, and using that $b+1 > 0, \ 2\epsilon \geq 0$, the absolute value of the $i$-th term in the product $(\ref{eqn:productk})$ is bounded by
\begin{equation}\label{eqn:bound1}
\left| \frac{\left( N(1 + \zeta) + z + 2\epsilon + i \right)\left( N(1 - \zeta) + z + i \right) }{\left( N + z + b + 1 + i \right) \left( N + z + 2\epsilon + i \right)} \right| \leq \frac{(N|1 + \zeta| + |z| + 2\epsilon + i)(N|1 - \zeta| + |z| + i)}{(N + \Re z + i)^2},
\end{equation}
for all $N$ large enough, say for $N > |\Re z|$. Let $c_1 > 0$ and $c_2 > 1$ be arbitrary real numbers. It is easy to see that there exists $N_1\in\N$ large enough so that the right-hand side of $(\ref{eqn:bound1})$ is upper bounded by
\begin{equation}\label{eqn:bound2}
\leq c_2\frac{(N|1 + \zeta| + i)(N|1 - \zeta| + i)}{(N+i)^2} \ \forall \ N \geq N_1, \ \forall \ i \leq c_1N \ \forall \ \zeta\in J_{\delta},
\end{equation}
and any $\delta > 0$. Let $c_3 > 1$ be any real number. There exists $\delta_1 > 0$ small enough such that $|1 - \zeta| < c_3 (1 - |\zeta|)$ for all $\zeta\in J_{\delta_1}$. Therefore we can bound $(\ref{eqn:bound2})$ by
\begin{equation*}
\begin{gathered}
\leq c_2c_3 \frac{(N(1 + |\zeta|) + i)(N(1 - |\zeta|) + i)}{(N+i)^2} \ \forall N \geq N_1, \ \forall \ i \leq c_1N, \ \forall  \ \zeta\in J_{\delta_1}\\
= c_2c_3 \left( 1 - \frac{N^2|\zeta|^2}{(N+i)^2} \right) \leq c_2c_3 \left( 1 - \left(\frac{N\inf_{\zeta\in J}{|\zeta|}}{N+i}\right)^2 \right) \leq c_2c_3 \left( 1 - \left(\frac{\inf_{\zeta\in J}{|\zeta|}}{1 + c_1}\right)^2 \right).
\end{gathered}
\end{equation*}
Thus for any $c_1 > 0$, let $c_2, c_3 > 1$ be any real numbers such that $d \myeq c_2c_3 \left( 1 - \left(\frac{\inf_{\zeta\in J}{|\zeta|}}{1 + c_1}\right)^2 \right)\in(0, 1)$ and fix them. Let $J\subset (0, 1)$ be any compact subset. The estimates above show that there exists $N_1\in\N$ large enough, $\delta_1 > 0$ small enough such that we have the bound
\begin{equation*}
\left| \frac{\left( N(1 + \zeta) + z + 2\epsilon + i \right)\left( N(1 - \zeta) + z + i \right) }{\left( N + z + b + 1 + i \right) \left( N + z + 2\epsilon + i \right)} \right| \leq d, \ \forall \ N \geq N_1, \ \forall \ \zeta\in J_{\delta_1}, \ \forall \ i\leq c_1N.
\end{equation*}

For any $k \leq c_1N$ and $N > N_1$, the estimate above yields the bound
\begin{equation}\label{eqn:bound4}
\begin{gathered}
\left| \frac{(1+z+b)_k(-z')_k(N(1+\zeta) + z + 2\epsilon)_k(N(1 - \zeta) + z)_k}{(1 + z - z')_k(N + z + b + 1)_k(N+z+2\epsilon)_k k!} \right|\\
\leq \left| \frac{(1+z+b)_k(-z')_k}{(1+z-z')_kk!} \right| d^k \ \forall k \leq c_1N, \ \forall N>N_1.
\end{gathered}
\end{equation}

We need also similar estimates for $k>c_1N$ and $N$ large enough. Let us begin by estimating the left-hand side of $(\ref{eqn:bound1})$ for $i \geq c_1N$ and large $N$. We claim that for any $c_4 \in (0, 1)$, there exists $N_2\in\N$ large enough and $\delta_2 > 0$ small enough such that
\begin{equation}\label{eqn:bound3}
\begin{gathered}
\left| \frac{\left( N(1 + \zeta) + z + 2\epsilon + i \right)\left( N(1 - \zeta) + z + i \right) }{( N + z + b + 1 + i ) ( N + z + 2\epsilon + i )} \right| = \left| 1 - \frac{N^2\zeta^2 + 2N\zeta \epsilon + (b+1)(N+z+i+2\epsilon)}{( N + z + b + 1 + i ) ( N + z + 2\epsilon + i )}  \right|\\
\leq 1 - c_4\cdot\Re\left(\frac{N^2\zeta^2 + 2N\zeta \epsilon + (b+1)(N+z+i+2\epsilon)}{( N + z + b + 1 + i ) ( N + z + 2\epsilon + i )}\right) \forall N > N_2, \ \forall \zeta\in J_{\delta_2}, \ \forall i \geq c_1N.
\end{gathered}
\end{equation}
The reader can supply a formal proof, but let us give a sketch. If we let $\frac{N^2\zeta^2 + 2N\zeta \epsilon + (b+1)(N+z+i+2\epsilon)}{( N + z + b + 1 + i ) ( N + z + 2\epsilon + i )} = A + B\sqrt{-1}$, then inequality $(\ref{eqn:bound3})$ is equivalent to $B^2 \leq (1 - c_4)A(2 - (1+c_4)A)$. We can estimate $A \sim \frac{N^2(\Re\zeta)^2 + N^2(\Im\zeta)^2 + (b+1)(N+i)}{(N+i)^2}$ and $B\sim \frac{2N^2(\Re\zeta)(\Im\zeta)}{(N+i)^2}$ when $N$ is very large. Also $\Re\zeta\in J\subset(0, 1)$ and $J$ is compact, so $\inf_{\zeta\in J}{\Re\zeta},\ \sup_{\zeta\in J}{\Im\zeta} \in (0, 1)$. It should then be clear that we can choose $\delta_2>0$ very small such that  $B^2 \leq (1 - c_4)A(2 - (1+c_4)A)$ holds when $A$ and $B$ are replaced by their large $N$ estimates. If we choose $N_2$ large enough, then for $N>N_2$, the inequality itself also holds.

By the same principle, for any $c_5\in (0, 1)$, we can increase the value of $N_2\in\N$ if necessary so we can upper bound the last expression of $(\ref{eqn:bound3})$ by
\begin{equation*}
\begin{gathered}
\leq 1 - c_4 c_5\cdot\Re\left(\frac{N^2\zeta^2 + 2N\zeta \epsilon + (b+1)(N + i)}{( N + i )^2}\right)\\
\leq 1 - c_4 c_5\cdot\frac{b+1}{N + i} \ \forall N > N_2, \ \forall \zeta\in J_{\delta_2}, \ \forall i \geq c_1N.
\end{gathered}
\end{equation*}
Now if $k > c_1N$, we have obtained the following bound
\begin{eqnarray}
\prod_{i=c_1N}^k{\left| \frac{\left( N(1 + \zeta) + z + 2\epsilon + i \right)\left( N(1 - \zeta) + z + i \right) }{\left( N + z + b + 1 + i \right) \left( N + z + 2\epsilon + i \right)} \right|} \leq \prod_{i=c_1N}^{k}{\left( 1 - \frac{c_4c_5(b+1)}{N+i} \right)}\nonumber\\
= \exp\left( \sum_{i=c_1N}^k{\ln{\left( 1 - \frac{c_4c_5(b+1)}{N+i} \right)}} \right) \leq \exp\left( -c_4c_5(b+1) \sum_{i=c_1N}^k{{\frac{1}{N+i} }} \right)\nonumber\\
\leq \exp\left( -c_4c_5(b+1) \left( \ln{(N + k)} - \ln{(N(1 + c_1))} - 1 \right) \right) = const\times\frac{N^{c_4c_5(b+1)}}{(N+k)^{c_4c_5(b+1)}}.\label{eqn:bound5}
\end{eqnarray}
From $(\ref{eqn:bound4})$ applied to $c_1N$, we also have that $N>N_1$ implies
\begin{equation}\label{eqn:bound6}
\left| \frac{( N(1 + \zeta) + z + 2\epsilon)_{c_1N} ( N(1 - \zeta) + z)_{c_1N} }{( N + z + b + 1)_{c_1N} (N + z + 2\epsilon)_{c_1N}} \right| \leq d^{c_1N},
\end{equation}
for some constant $d\in(0, 1)$. Since $d^{c_1N}$ is exponentially decreasing in $N$ and $N^{c_4c_5(b+1)}$ has polynomial growth, there exists $N_3\in\N$ such that $N > N_3$ implies $const\times N^{c_4c_5(b+1)}\times d^{c_1N}\leq 1$, where $const$ is the same constant as in $(\ref{eqn:bound5})$. We conclude from $(\ref{eqn:bound5})$ and $(\ref{eqn:bound6})$ that
\begin{equation}\label{eqn:bound7}
\begin{gathered}
\left| \frac{( N(1 + \zeta) + z + 2\epsilon)_k ( N(1 - \zeta) + z)_k }{( N + z + b + 1)_k (N + z + 2\epsilon)_k} \right|\\
\leq \frac{1}{(N+k)^{c_4c_5(b+1)}} \leq \frac{1}{k^{c_4c_5(b+1)}}, \forall k > c_1N, \ N > \max\{N_1, N_2, N_3\}.
\end{gathered}
\end{equation}
Finally we estimate the remaining factors in the general $k$-th term of the power series $(\ref{eqn:powerseries4F3})$. For large $k$, we can deduce from the asymptotics of Gamma functions that
\begin{equation}\label{eqn:bound8}
\begin{gathered}
\frac{(1+z+b)_k(-z')_k}{(1+z-z')_kk!} = \frac{\Gamma(1+z-z')}{\Gamma(1+z+b)\Gamma(-z')}\frac{\Gamma(k+z+b+1)\Gamma(k-z')}{\Gamma(k+1+z-z')\Gamma(k+1)}\\
= \frac{\Gamma(1+z-z')}{\Gamma(1+z+b)\Gamma(-z')}k^{b-1}(1 + O(1/k)), \ k\rightarrow\infty.
\end{gathered}
\end{equation}
So far, we chose $c_4, c_5\in (0, 1)$ to be arbitrary constants. Let us require them to be any real numbers for which $b - 1 - c_4c_5(b+1) < -\frac{3}{2}$; this is possible if we pick $c_4$ and $c_5$ to be very close to $1$. From the latter estimates $(\ref{eqn:bound7})$ and $(\ref{eqn:bound8})$, we have that the absolute value of the $k$-th term in the power series $(\ref{eqn:powerseries4F3})$ is upper-bounded by
\begin{equation}\label{eqn:bound9}
\begin{gathered}
\left| \frac{(1+z+b)_k(-z')_k(N(1+\zeta) + z + 2\epsilon)_k(N(1 - \zeta) + z)_k}{(1 + z - z')_k(N + z + b + 1)_k(N+z+2\epsilon)_k k!} \right|\\
\leq const\times k^{b - 1 - c_4c_5(b+1)} < const\times k^{-\frac{3}{2}} \ \forall \ k > c_1N, \ \forall \ N > \max\{N_1, N_2, N_3\}, \ \forall \ \zeta\in J_{\delta_2}.
\end{gathered}
\end{equation}
Let $N_0 = \max\{N_1, N_2, N_3\}$ and $\delta_0 = \min\{\delta_1, \delta_2\}$. We have proved above that for $N > N_0$, the absolute value of the $k$-th term in the power series $(\ref{eqn:powerseries4F3})$ is upper bounded by the absolute value of the $k$-th term of the power series of $_2F_1(1+z+b, \ -z'; \ 1 + z - z'; \ d)$, for some $d\in (0, 1)$ when $k \leq c_1N$ and by $const\times k^{-3/2}$ if $k>c_1N$. Since the power series of the hypergeometric function $_2F_1(1+z+b, \ -z'; \ 1 + z - z'; \ d)$ is absolutely convergent and the sum $\sum_{k\geq 1}{k^{-3/2}}$ is convergent, it follows that we can apply the uniform convergence theorem to obtain the desired uniform limit $(\ref{eqn:limithypergeos})$.

Therefore we are finished proving the first limit in the statement of Proposition $\ref{RS2limit}$. The second limit $\lim_{N\rightarrow\infty}{N^{-2\Sigma}\widetilde{S}_{<N}(N\zeta)} = S_{<1}(\zeta^2)$ can be proved similarly, so we leave it as an exercise to the reader. Notice that the complicated part is to prove an statement analogous to $(\ref{eqn:limithypergeos})$. Such statement can in fact be deduced from the proof above, since the $_4F_3$ function appearing in $\widetilde{S}_{<N}(N\zeta)$ can be obtained from the $_4F_3$ function in $\widetilde{R}_{<N}(N\zeta)$ after the change of variables $z\mapsto z-1$, $z'\mapsto z' - 1$, $N\mapsto N+1$. The proof of Proposition $\ref{RS2limit}$ is now complete.
\end{proof}

\subsection{Proof of the main theorem}

As we have mentioned already, the key is to use Proposition $\ref{thm:scalinglimit}$ to approximate the point process $\PP$ on $\X$ by point processes $\LL^{(N)}$ on $\Zpe$. Clearly condition (2) always holds by Theorem $\ref{thm:Lkernel}$.

Now let us assume that $(z, z')\in\U_{\adm}$ is such that $\Sigma = z+z'+b \neq 0$ and $z-z'\notin\Z$. Later we remove this restriction. Clearly each of the functions $\psi_{<1}, \psi_{>1}, R_{<1}, S_{<1}, R_{>1}, S_{>1}$ is continuous on their respective domains. In fact, they are all actually holomorphic on neighborhoods of their respective domains. It follows from $(\ref{KernelP})$ that $K^{\PP}$ is holomorphic on a neighborhood of $\X\times\X\setminus\{(x, x) : x\in\X\}$. From $(\ref{KernelP})$, we also deduce that $K^{\PP}$ admits an analytic continuation to the diagonal $\{(x, x) : x\in\X\}$ and that continuation is given by $(\ref{KernelP2})$. In particular, $K^{\PP}$ is continuous on $\X\times\X$, thus (1) is satisfied.

Let us verify condition (3). For $x, y\in\X$, let $\widehat{x_N}, \widehat{y_N}$ be the points in $\Zpe$ that are closest to $\left((N+\epsilon-\frac{1}{2})x\right)^2$ and $\left((N+\epsilon - \frac{1}{2})y\right)^2$, respectively. Thanks to Propositions $\ref{psilimit}$, $\ref{RS1limit}$ and $\ref{RS2limit}$, we have the uniform limit
\begin{equation}\label{eqn:uniformlimit}
\lim_{N\rightarrow\infty}{N\cdot K^{\LL^{(N)}}(\widehat{x_N}, \widehat{y_N})} = (2x^{1/2}y^{1/2})\cdot K^{\PP}(x^2, y^2)
\end{equation}
on compact subsets of $\X\times\X$. Observe that we made use of the \textit{uniform} convergence of Propositions $\ref{RS1limit}$ and $\ref{RS2limit}$ to deal with the singularities on the diagonal (of $K^{\LL^{(N)}}(\widehat{x_N}, \widehat{x_N})$ and $K^{\PP}(x, x)$). The conditions of Proposition $\ref{thm:scalinglimit}$ have been verified, so we can conclude that $\PP$ is a determinantal process and $K^{\PP}$ is a correlation kernel for it.

The final task is to remove the restrictions $\Sigma \neq 0$ and $z-z' \notin \Z$, by means of analytic continuation. More precisely we show that $K^{\PP}(x, y)$ admits an analytic continuation to $(z, z')\in\U_0$, and that the uniform limit $(\ref{eqn:uniformlimit})$ still holds for any $(z, z')$ in the domain $\U_0$ and in particular for any $(z, z')\in\U_{\adm}\subset\U_0$.

We observe that the limits of Propositions $\ref{psilimit}$, $\ref{RS1limit}$ and $\ref{RS2limit}$ are all uniform on compact subsets of $(z, z')\in\U_0\cap\{\Sigma \neq 0\}\cap\{z-z'\notin\Z\}$. Therefore the limit $(\ref{eqn:uniformlimit})$ holds uniformly on compact subsets of the domain $(x, y, (z, z'))\in\X\times\X\times\left( \U_0\cap\{\Sigma \neq 0\}\cap\{z-z'\notin\Z\} \right)$.
It is implied that the functions $(2x^{1/2}y^{1/2})\cdot K^{\PP}(x^2, y^2)$, for fixed $x, y$, are such that any point $(z, z')$ in one of the singular hyperplanes $\{\Sigma = 0\}$ or $\{z - z' = k\in\Z\}$ has a neighborhood $U\subset\U_0$ such that $(2x^{1/2}y^{1/2})\cdot K^{\PP}(x^2, y^2)$ is bounded on $U\setminus\left( \{\Sigma = 0\}\cup\{z - z'\in\Z\} \right)$. Thus $(2x^{1/2}y^{1/2})\cdot K^{\PP}(x^2, y^2)$ admits an analytic continuation to $\U_0$, e.g. see \cite[Thm. 8]{Rcx}. Therefore $K^{\PP}(x, y)$ also has an analytic continuation to $\U_0$.

We have mentioned that the limit $(\ref{eqn:uniformlimit})$ holds uniformly on compact subsets of $(x, y, (z, z'))\in\X\times\X\times\left( \U_0\cap\{\Sigma \neq 0\}\cap\{z-z'\notin\Z\} \right)$. Moreover we just showed that $K^{\PP}(x, y)$, as a function of $(z, z')$, has an analytic continuation to $\U_0$. In the proof of Theorem $\ref{thm:Lkernel}$ we also showed that $K^{\LL^{(N)}}$ is holomorphic on $\U_0$. Hence the limit $(\ref{eqn:uniformlimit})$ also holds uniformly on compact subsets of $(x, y, (z, z')) \in \X\times\X\times\U_0$ (as an application of Cauchy's integral formula). In particular, if $(z, z')\in\U_{\adm}\subset\U_0$ is fixed, the limit $(\ref{eqn:uniformlimit})$ holds uniformly on compact subsets of $\X\times\X$, concluding the proof.

\begin{appendix}

\section{Identities for generalized hypergeometric functions}

For simplicity in notation, denote $\textrm{Sin}[a_1, a_2, \ldots] = \sin\pi a_1 \sin\pi a_2\cdots$, in addition to the notations of Section $\ref{sec:notation}$. In this appendix, we prove several statements left out in the main body of the text. All the proofs here involve long computations with identities involving generalized hypergeometric functions.

Many of the arguments below involve the following well known identities in several places
\begin{equation*}
\begin{gathered}\Gamma(t)\Gamma(1-t) = \frac{\pi}{\sin\pi t} \ \forall t\notin\Z,\\
\sin \pi(M + \eta) = (-1)^M \sin \pi \eta, \ \sin \pi(M - \eta) = (-1)^{M+1} \sin \pi \eta \ \forall  M\in\Z.
\end{gathered}
\end{equation*}
Whenever we mention `simplifications' or `algebraic manipulations' below, we mean that we use the identities above, possibly many times.

\begin{proof}[Proof of equalities $(\ref{pNfirst}) = (\ref{almostpN})$ and $(\ref{pNfirst1}) = (\ref{almostpN1})$]
Observe that $(\ref{pNfirst1}) = (\ref{almostpN1})$ is equivalent to $(\ref{pNfirst}) = (\ref{almostpN})$ after the change of variables $N\mapsto N-1, z\mapsto z+1, z'\mapsto z'+1$, so it suffices to prove $(\ref{pNfirst}) = (\ref{almostpN})$.

In \cite[Sec. 7]{B}, there is a convenient notation for writing fundamental identities for generalized hypergeometric functions $_4F_3$; let us recall here such notation. Given any $r_1, r_2, \ldots, r_6\in\C$ and $n\in\N$ satisfying $\sum_i{r_i}=0$, define
\begin{eqnarray*}
\phi &=& \frac{1}{3}(n-1),\\
\epsilon_{ij} &=& r_i + r_j - \phi \ \forall i\neq j,
\end{eqnarray*}
and the following expression involving a generalized hypergeometric function
\begin{equation*}
\begin{gathered}
S(1, 2, 3) = (-1)^n\cdot\Gamma\left[ \begin{split}
1-\epsilon_{56} &,& 1-\epsilon_{46} &,& 1-\epsilon_{45} \\
1-n-\epsilon_{56} &,& 1-n-\epsilon_{46} &,& 1-n-\epsilon_{45}
\vphantom{\frac12}
\end{split}\right]\\
\times\pFq{4}{3}{-n ,,, \epsilon_{12} ,,, \epsilon_{23} ,,, \epsilon_{13}}{1-n-\epsilon_{45} ,,, 1-n-\epsilon_{56} ,,, 1-n-\epsilon_{46}}{1}.
\end{gathered}
\end{equation*}
Analogously, we can define $S(a, b, c)$ for any $a, b, c\in\{1, 2, 3, 4, 5, 6\}$. In particular, it is true that $S(a, b, c)$ is symmetric with respect to its arguments $a, b, c$. Less obvious is \cite[Sec. 7.3 (6)]{B}, which states
\begin{equation}\label{123456}
S(1, 2, 3) = S(4, 5, 6).
\end{equation}
For $n = N$ and the choice of parameters
\begin{equation*}
\begin{gathered}
r_1 = -x-b-\frac{a+1+z+z'}{2} - \frac{N-1}{3}, \hspace{.2in} r_2 = x+\frac{a+1-z-z'}{2} - \frac{N-1}{3},\\
r_3 = b + \frac{a+1+z+z'}{2} + \frac{2(N-1)}{3}, \hspace{.2in} r_4 = \frac{z'-z-a-1}{2} - \frac{N-1}{3},\\
r_5 = \frac{z-z'-a-1}{2} - \frac{N-1}{3}, \hspace{.2in} r_6 = \frac{z+z'+a-1}{2} + \frac{2N+1}{3},
\end{gathered}
\end{equation*}
$(\ref{123456})$ yields the identity
\begin{equation*}
\begin{gathered}
\Gamma\left[ \begin{split}
1-z &,& 1-z' &,& N+a+1\\
1-z-N &,& 1-z'-N &,& 1+a
\vphantom{\frac12}
\end{split}\right]\pFq{4}{3}{-N ,, -N+1-\Sigma ,, x+2\epsilon ,, -x}{1+a ,,, 1-N-z ,,, 1-N-z'}{1}\\
= \Gamma\left[ \begin{split}
N+\Sigma &,& -x-a-b &,& x+1\\
\Sigma &,& -x-N-a-b &,& x-N+1
\vphantom{\frac12}
\end{split}\right]\pFq{4}{3}{-N ,,, -N-a ,,, z ,,, z'}{x-N+1 ,, -x-N-a-b ,, \Sigma}{1}.
\end{gathered}
\end{equation*}
Recall that $\Sigma = z + z' + b$ in equations $(\ref{pNfirst})$ and $(\ref{almostpN})$. Then the identity above yields the equality between $(\ref{pNfirst})$ and $(\ref{almostpN})$, after several algebraic manipulations.
\end{proof}

\begin{proof}[Proof of equalities $(\ref{eq:RgeqN})$ and $(\ref{eq:SgeqN})$]

Observe that $(\ref{eq:SgeqN})$ can be deduced from $(\ref{eq:RgeqN})$ (and from the formulas $(\ref{SlargerN}), (\ref{RlargerN})$) after the change of parameters $z\mapsto z+1$, $z'\mapsto z'+1$, $N\mapsto N-1$. Thus we only prove $(\ref{eq:RgeqN})$. Use $(\ref{RlargerN})$ for the definition of $\TR_{\geq N}(\zeta)$; by simple cancellations, the equality is equivalent to (use $x$ instead of $N\zeta$):
\begin{equation}\label{eqn:appendix1}
\begin{gathered}
\Gamma\left[ \begin{split}
1-z &,& 1-z' &,& N+a+1\\
1-z-N &,& 1-z'-N &,& 1+a
\vphantom{\frac12}
\end{split}\right]\pFq{4}{3}{-N ,, -N+1-\Sigma ,, x+2\epsilon ,, -x}{1+a ,,, 1-N-z ,,, 1-N-z'}{1}\\
= \Gamma\left[ \begin{split}
1-\Sigma &,& x+z'+N+2\epsilon &,& x+1-z'\\
1-N-\Sigma &,& x+z'+2\epsilon &,& x-N+1-z'
\vphantom{\frac12}
\end{split}\right]\pFq{4}{3}{-N ,,, N+\Sigma+a ,,, z'+b ,,, z'}{-x+z' ,, x+z'+2\epsilon ,, \Sigma}{1}.
\end{gathered}
\end{equation}
As the previous two proofs in the Appendix, the latter equality is a consequence of one of the identities in \cite[Sec. 5.3]{B}, namely
\begin{equation}\label{1234561}
S(1, 2, 3) = S(3, 5, 6).
\end{equation}
For $n = N$ and the choice of parameters
\begin{equation*}
\begin{gathered}
r_1 = -x-b-\frac{a+1+z+z'}{2} - \frac{N-1}{3}, \hspace{.2in} r_2 = x+\frac{a+1-z-z'}{2} - \frac{N-1}{3},\\
r_3 = b + \frac{a+1+z+z'}{2} + \frac{2(N-1)}{3}, \hspace{.2in} r_4 = \frac{z-z'-a-1}{2} - \frac{N-1}{3},\\
r_5 = \frac{z'-z-a-1}{2} - \frac{N-1}{3}, \hspace{.2in} r_6 = \frac{z+z'+a-1}{2} + \frac{2N+1}{3},
\end{gathered}
\end{equation*}
equality $(\ref{1234561})$ yields the desired $(\ref{eqn:appendix1})$.
\end{proof}

For the proof of Lemma $\ref{R2S2rewritten}$, we shall use several identities of $L$-functions. For example, it is clear that the function $L(A, B, C, D; E; F, G)$ does not depend on the order of its top arguments $A, B, C, D$ or the order of its bottom arguments $F, G$; we use such independence on the ordering of the arguments without further mention. We also need the following propositions, which are proved in \cite{Mi1}.

\begin{prop}[$L$-fundamental transformations; \cite{Mi1}, (4.7) \& (4.8)]\label{prop:fundamental}
If $E + F + G - A - B - C - D = 1$, then
\begin{eqnarray*}
L\left[ \begin{split}
A , \ B &,& C &,& D;\\
E &;& F &,& G
\vphantom{\frac12}
\end{split}\right] &=& L\left[ \begin{split}
1+A-E, \hspace{.2in} A,\hspace{.2in} G-C,\hspace{.2in} F-C;\\
1+A-C;\ 1+A+B-E,\ 1+A+D-E
\vphantom{\frac12}
\end{split}\right]\\
&=& L\left[ \begin{split}
A, \hspace{.3in} B,\hspace{.3in} G-C,\hspace{.3in} G-D;\\
1+A+B-F;\ 1+A+B-E,\ G
\vphantom{\frac12}
\end{split}\right].
\end{eqnarray*}
\end{prop}

\begin{prop}[$L$-incoherent identity (of type $\{6, 5, \overline{6}\}$)]\label{incoherent}
If $E + F + G - A - B - C - D = 1$, then
\begin{equation*}
\begin{gathered}
\frac{\sin \pi(F - G)\sin \pi G \cdot L\left[ \begin{split}
A , \ B ,\ C ,\ D;\\
E ;\hspace{.1in} F, \hspace{.1in} G
\vphantom{\frac12}
\end{split}\right]}{\Gamma\left[1 + A - F, \ 1 + B - F, \ 1 + C - F, \ 1 + D - F, \ E - A, \ E - B, \ E - C, \ E - D\right]}\\
+ \frac{1}{\pi^4}\left\{ \textrm{Sin}[G, \ G - E,  \ F - A. \ F - B, \  F - C, \ F - D] \right.\\
\left. + \textrm{Sin}[F, \ E - F, \ G - A, \ G - B, \ G - C, \ G - D] \right\}\cdot L\left[ \begin{split}
A , \ B ,\ C ,\ D;\\
F ;\hspace{.1in} E ,\hspace{.1in} G
\vphantom{\frac12}
\end{split}\right]\\
+ \frac{\sin \pi(E - F)\sin \pi(F - G)}{\Gamma\left[ A, \ B, \ C, \ D, \ G - A, \ G - B, \ G - C, \ G - D\right]}\cdot L\left[ \begin{split}
1 - A , \ 1 - B ,\ 1 - C ,\ 1 - D;\\
2 - E ;\hspace{.1in} 2 - F ,\hspace{.1in} 2 - G
\vphantom{\frac12}
\end{split}\right] = 0
\end{gathered}
\end{equation*}
\end{prop}

\begin{proof}[Proof of Lemma $\ref{R2S2rewritten}$]

The proof consists of putting together several identities for generalized hypergeometric functions and other special functions. We begin with the expression for $\TR_{<N}(\zeta)$ in Lemma $\ref{R2S2rewritten}$ and after applying several identities, we arrive at the expression for $\TR_{<N}(\zeta)$ in Proposition $\ref{prop:R2S2limit}$.

Because of $(\ref{eqn:4F3L})$ and the first equality of Proposition $\ref{prop:fundamental}$ applied to the parameters $A = 1+z,\ B = -N+1,\ C = 1+z',\ D = -N+1-a,E = \zeta-N+2,\ F = 1-\zeta-N-a-b,\ G = z+z'+b+2$, we have
\begin{equation*}
\begin{gathered}
\frac{\Gamma(N-\zeta-1)\cdot\pFq{4}{3}{-N+1 ,, -N+1-a ,, 1+z ,, 1+z'}{\zeta-N+2 ,, 1-\zeta-N-a-b ,, \Sigma+2}{1}}{\pi\Gamma[-\zeta, -\zeta-a, z-\zeta+N, z'-\zeta+N, 1-\zeta-N-a, \Sigma+2]} = L\left[ \begin{split}
-N+1, -N+1-a, 1+z, 1+z';\\
\zeta-N+2; 1-\zeta-N-a-b, \Sigma+2
\vphantom{\frac12}
\end{split}\right]\\
= L\left[ \begin{split}
1+z, \ -N+1, \ 1+z', \ -N+1-a;\\
\zeta-N+2;\ 1-\zeta-N-a-b,\ \Sigma+2
\vphantom{\frac12}
\end{split}\right] = L\left[ \begin{split}
z-\zeta+N, \ 1+z+b, \ 1+z, \ -\zeta-N-z'-a-b;\\
1+z-z';\ -\zeta+z+1,\ -\zeta+z+1-a
\vphantom{\frac12}
\end{split}\right]
\end{gathered}
\end{equation*}
By using the definition of $h_{N-1}$ in $(\ref{hN1})$, and several algebraic manipulations, it follows that the first term of $\TR_{<N}(\zeta)$ in Lemma $\ref{R2S2rewritten}$ equals
\begin{equation}\label{eqn:firstterm}
\begin{gathered}
\frac{\pi^3 G(\zeta)\cdot\Gamma\left[ 1+z, \ 1+z', \ 1+z+b, \ 1+z'+b, \ N+\Sigma+1, \ N+\Sigma+a+1, \ N - \zeta, \ N+\zeta+2\epsilon \right]}{\textrm{Sin}[\zeta+a, \ \zeta+a+b]\cdot\Gamma\left[ N+a, \ N, \ \Sigma + 1, \ \zeta+b+1, \ \zeta+2\epsilon, \ \zeta+N+z+2\epsilon, \ \zeta+N+z'+2\epsilon \right]}\\
\times L\left[ \begin{split}
-\zeta+z+N ,\ 1+z ,\ 1+z+b ,\  -\zeta-N-z'-a-b;\\
1+z-z' ;\hspace{.2in} -\zeta+z+1 ,\hspace{.2in} -\zeta+z+1-a
\vphantom{\frac12}
\end{split}\right].
\end{gathered}
\end{equation}

As for the second term of $\TR_{<N}(\zeta)$ in Lemma $\ref{R2S2rewritten}$, by the evident symmetries of the top arguments of the $L$-function, it is equal to
\begin{equation}\label{eqn:secondterm}
\begin{gathered}
\frac{\pi\cdot\Gamma[-\Sigma,\ \zeta+1,\ \zeta+a+1,\ \zeta+N+2\epsilon,\ \zeta+1-z-N,\ \zeta+1-z'-N]}{\Gamma(\zeta-N+1)}\\
\times L\left[ \begin{split}
\zeta-N-z+1 ,\ -z ,\ -z-b ,\ \zeta+N+z'+2\epsilon;\\
1+z'-z ;\hspace{.1in} \zeta+1-z ,\hspace{.1in} \zeta+a+1-z
\vphantom{\frac12}
\end{split}\right].
\end{gathered}
\end{equation}

Let us apply the $L$-incoherent identity, Proposition $\ref{incoherent}$, to the parameters
\begin{equation*}
\begin{gathered}
A = -\zeta + z + N, \ B = 1 + z, \ C = 1 + z + b, \ D = -\zeta - N - z' - a - b,\\
E = 1 + z - z', \ F = -\zeta + 1 + z, \ G = -\zeta + 1 + z - a.
\end{gathered}
\end{equation*}
Since $F - A = 1 - N$ implies $\sin \pi (F - A) = 0$, in the second term of Proposition $\ref{incoherent}$ one of the products of six sine functions vanishes. After some simplifications, we deduce the following identity
\begin{equation}\label{eqn:incoherentapp1}
\begin{gathered}
\Gamma\left[ \zeta +1 - z -N,\  \zeta + 1 - z' - N,\  \zeta + a + 1,\  \zeta + 1 \right]\cdot L\left[ \begin{split}
\zeta-N-z+1 ,\ -z ,\ -z-b ,\ \zeta+N+z'+2\epsilon;\\
1+z'-z ;\ \zeta+1-z ,\ \zeta+a+1-z
\vphantom{\frac12}
\end{split}\right]\\
= \frac{\pi^2\cdot \textrm{Sin}[\zeta+a-z, \Sigma+a, z', z'+b]\cdot \Gamma[N+\Sigma+a+1,  N + \Sigma + 1, 1+z, 1+z+b, 1+z', 1+z'+b]}
{\textrm{Sin}[\zeta-z, \ \zeta-z',\ \zeta+z'+a+b, \ a, \ \zeta+a, \ \zeta+a+b]\cdot \Gamma[N+a, \ N]}\\
\times \frac{L\left[ \begin{split}
-\zeta+z+N ,\ 1+z ,\ 1+z+b ,\ -\zeta-N-z'-a-b;\\
1+z-z' ;\hspace{.1in} -\zeta+1+z ,\hspace{.1in} -\zeta-a+1+z
\vphantom{\frac12}
\end{split}\right]}{\Gamma[\zeta+2\epsilon,\ \zeta+b+1,\ \zeta+N+z+2\epsilon,\ \zeta+N+z'+2\epsilon]}\\
+ \frac{(-1)^N\pi\sin(\pi\Sigma)}{\sin\pi(\zeta + z' + a + b)\sin\pi a}\cdot \Gamma\left[ \begin{split}
1+z &,& 1+z+b &,& N+\Sigma+1 &,& \zeta+1-z'-N &,& \zeta+1\\
&& \zeta+2\epsilon &,& N+a &,& \zeta+N+z'+2\epsilon
\vphantom{\frac12}
\end{split}\right]\\
\times L\left[ \begin{split}
-\zeta+z+N ,\ 1+z ,\ 1+z+b ,\ -\zeta-N-z'-a-b;\\
-\zeta+1+z ;\hspace{.1in} 1+z-z' ,\hspace{.1in} -\zeta+1+z-a
\vphantom{\frac12}
\end{split}\right].
\end{gathered}
\end{equation}

Then we can rewrite $(\ref{eqn:secondterm})$ above by using the $L$-incoherent relation above (the $L$-function in $(\ref{eqn:secondterm})$ and that in the left-hand side of $(\ref{eqn:incoherentapp1})$ have all the same arguments). It follows that $\TR_{<N}(\zeta)$ is the sum of the following two terms:

\begin{equation*}
\begin{gathered}
(\textrm{Term }1) \ \frac{\pi^3\cdot\Gamma[N+\Sigma+a+1, \ N+\Sigma+1, \ 1+z, \ 1+z+b, \ 1+z', \ 1+z'+b, \ \zeta+N+2\epsilon]}{\Gamma[N+a, \ N, \ \zeta+2\epsilon, \ \zeta+b+1, \ \zeta+z+N+2\epsilon, \ \zeta+z'+N+2\epsilon]}\\
\times L\left[ \begin{split}
-\zeta+N+z ,\ 1+z ,\ 1+z+b ,\ -\zeta-N-z'-a-b;\\
1+z-z' ;\ -\zeta+1+z ,\ 1-\zeta-a+z
\vphantom{\frac12}
\end{split}\right]
\cdot \left\{ \frac{G(\zeta)\cdot\Gamma(N-\zeta)}{\Gamma(\Sigma+1)\textrm{Sin}[\zeta+a,\ \zeta+a+b]} \right.\\
\left. + \frac{\Gamma(-\Sigma)}{\Gamma(\zeta-N+1)}\frac{\textrm{Sin}[\zeta+a-z, \ \Sigma+a, \ z', \ z'+b]}{\textrm{Sin}[\zeta-z, \ \zeta-z', \ \zeta+z'+a+b, \ a, \ \zeta+a, \ \zeta+a+b]}
\right\},
\end{gathered}
\end{equation*}

\begin{equation*}
\begin{gathered}
(\textrm{Term }2) \ \frac{(-1)^N\pi^2\sin(\pi\Sigma)\Gamma[\zeta+1-z'-N, \zeta+1, 1+z, 1+z+b, N+\Sigma+1, \zeta+N+2\epsilon, -\Sigma]}
{\sin\pi(\zeta+z'+a+b)\sin(\pi a)\cdot\Gamma[\zeta+N+z'+2\epsilon, \ \zeta+2\epsilon, \ \zeta-N+1, \ N+a]}\\
\times L\left[ \begin{split}
-\zeta+N+z ,\ 1+z ,\ 1+z+b ,\ -\zeta-N-z'-a-b;\\
-\zeta+1+z ;\hspace{.15in} 1-z'+z ,\ 1-\zeta-a+z
\vphantom{\frac12}
\end{split}\right].
\end{gathered}
\end{equation*}

By using the definition of $G(\zeta)$ in $(\ref{Gdef})$, the sum inside brackets of (Term 1) equals
\begin{equation}\label{eqn:insideTerm}
\begin{gathered}
\frac{\Gamma(N-\zeta)}{\Gamma(\Sigma+1)}\frac{(-1)^N\sin \pi z'}{\textrm{Sin}[\zeta-z, \ \zeta-z', \ \zeta+a, \ \zeta+a+b, \ \Sigma]}\\
\times\left\{ \sin\pi z \sin \pi(\zeta-\Sigma) - \frac{\sin \pi\zeta\sin\pi(\zeta+a-z)\sin\pi(\Sigma+a)\sin\pi(-z'-b)}{\sin\pi(\zeta+z'+a+b)\sin\pi a} \right\}.
\end{gathered}
\end{equation}

To simplify the trigonometric expression that appears inside the brackets of equation $(\ref{eqn:insideTerm})$, use
\begin{equation}\label{eq:trigonometric}
\begin{gathered}
\sin A_1 \sin B_1 \sin C_1 \sin D_1 - \sin A_2  \sin B_2 \sin C_2 \sin D_2\\
= \sin\left( \frac{X+Y}{2} \right)\sin\left( \frac{Y-X}{2} \right)\sin\left( \frac{Z+W}{2} \right)\sin\left( \frac{W-Z}{2} \right),
\end{gathered}
\end{equation}
valid whenever
\begin{equation*}
\begin{gathered}
X = A_1 + B_1 = A_2 + B_2, \hspace{.1in}Y = C_1 + D_1 = C_2 + D_2,\\
Z = A_1 - B_1 = C_2 - D_2, \hspace{.1in}W = C_1 - D_1 = A_2 - B_2.
\end{gathered}
\end{equation*}
Specifically we can use $(\ref{eq:trigonometric})$ above for the choice of parameters
\begin{equation*}
\begin{gathered}
A_1 = \pi(\zeta-\Sigma), \ B_1 = \pi z, \ C_1 = \pi(\zeta+z'+a+b), \ D_1 = \pi a,\\
A_2 = \pi\zeta, \ B_2 = \pi(-z'-b),, \ C_2 = \pi(\zeta+a-z), \ D_2 = \pi(\Sigma+a),
\end{gathered}
\end{equation*}
which allows us to simplify the expression inside brackets of $(\ref{eqn:insideTerm})$ as $\frac{\textrm{Sin}[\zeta+a, \ z'+a+b, \ \zeta-z, \ \Sigma]}{\textrm{Sin}[\zeta+z'+a+b, \ a]}$ and therefore $(\ref{eqn:insideTerm})$ simplifies to
\begin{equation}\label{eqn:insideTerm2}
\frac{(-1)^N\pi^3 \Gamma[\zeta+N+2\epsilon, \ N-\zeta] \cdot \textrm{Sin}[z', z'+a+b]}{\Gamma(\Sigma+1)\cdot\textrm{Sin}[\zeta-z', \ \zeta+a+b, \ \zeta+z'+a+b, \ a]}.
\end{equation}
It follows that (Term 1) equals
\begin{equation*}
\begin{gathered}
(\textrm{Term }1')\ \frac{(-1)^N\pi^3 \textrm{Sin}[z', \ z'+a+b]}{\textrm{Sin}[\zeta-z', \ \zeta+a+b, \ \zeta+z'+a+b, \ a]}\cdot \Gamma\left[ \begin{split}
N+\Sigma+a+1 ,\ N+\Sigma+1 ,\ 1+z ,\ 1+z'\\
N+a ,\hspace{.4in} N ,\hspace{.4in} \Sigma+1
\vphantom{\frac12}
\end{split}\right]\\
\times \frac{\Gamma[1+z+b ,\ 1+z'+b ,\ \zeta+N+2\epsilon ,\ N-\zeta]}
{\Gamma[\zeta+2\epsilon ,\ \zeta+b+1 ,\ \zeta+N+z+2\epsilon, \ \zeta+N+z'+2\epsilon]}\\
\times\L\left[ \begin{split}
-\zeta+N+z ,\ 1+z ,\ 1+z+b ,\ -\zeta-N-z'-a-b;\\
1+z-z' ;\hspace{.15in} -\zeta+z+1 ,\ -\zeta+z+1-a
\vphantom{\frac12}
\end{split}\right]
\end{gathered}
\end{equation*}
From $(\ref{eqn:4F3L})$ and the first equality of Proposition $\ref{prop:fundamental}$ applied to the parameters $A = 1+z, \ B = -N+1, \ C = 1+z', \ D = -N+1-a, \ E = \zeta - N+2, \ F = \Sigma+2, \ G = 1-\zeta-N-a-b$, we have
\begin{equation*}
\begin{gathered}
\frac{\Gamma(N-\zeta-1)\pFq{4}{3}{1+z ,,, -N+1 ,,, 1+z' ,,, -N+1-a}{\zeta-N+2 ,,, \Sigma+2 ,,, 1-\zeta-N-a-b}{1}}{\pi\cdot\Gamma[-\zeta, \ -\zeta-a, \ N-\zeta+z, \ N-\zeta+z', \ \Sigma+2, \ 1-\zeta-N-a-b]}\\
= L\left[ \begin{split}
1+z,\ -N+1,\ 1+z',\ -N+1-a;\\
\zeta-N+2;\ \Sigma+2,\ 1-\zeta-N-a-b
\vphantom{\frac12}
\end{split}\right]
= L\left[ \begin{split}
N-\zeta+z,\ 1+z,\ 1+z+b,\ -\zeta-N-z'-a-b;\\
1+z-z';\ -\zeta+z+1,\ -\zeta+z+1-a
\vphantom{\frac12}
\end{split}\right]
\end{gathered}
\end{equation*}
By using the same identities, we can obtain
\begin{equation*}
\begin{gathered}
\frac{\Gamma(-\Sigma-1)\pFq{4}{3}{1+z ,,, -N+1 ,,, 1+z' ,,, -N+1-a}{\Sigma+2 ,,, \zeta-N+2 ,,, 1-\zeta-N-a-b}{1}}{\pi\cdot\Gamma[-N-\Sigma, \ -N-\Sigma-a, \ -z-b, \ -z'-b, \ \zeta-N+2, \ 1-\zeta-N-a-b]}\\
= L\left[ \begin{split}
1+z, -N+1, 1+z', -N+1-a;\\
\Sigma+2; \zeta-N+2, 1-\zeta-N-a-b
\vphantom{\frac12}
\end{split}\right]
= L\left[ \begin{split}
-z'-b, 1+z,\ -\zeta-N-z'-a-b,\ \zeta-N+1-z';\\
1+z-z';\ 1-N-z'-b,\ 1-N-z'-a-b
\vphantom{\frac12}
\end{split}\right]
\end{gathered}
\end{equation*}
The $_4F_3$ generalized hypergeometric functions in the identities above have the same parameters, therefore by equating them we obtain the following identity
\begin{equation*}
\begin{gathered}
L\left[ \begin{split}
N-\zeta+z, 1+z, 1+z+b,\ -\zeta-N-z'-a-b;\\
1+z-z';\ -\zeta+z+1,\ -\zeta+z+1-a
\vphantom{\frac12}
\end{split}\right] = \frac{(-1)^N \sin \pi\Sigma}{\sin \pi\zeta}\cdot \Gamma\left[ \begin{split}
-N-\Sigma-a,\ -N-\Sigma\\
-\zeta+z+N,\ -\zeta+z'+N
\vphantom{\frac12}
\end{split}\right]\\
\times \Gamma\left[ \begin{split}
-z-b,\ -z'-b\\
-\zeta,\ -\zeta-a
\vphantom{\frac12}
\end{split}\right]
L\left[ \begin{split}
1+z,\ -z'-b,\ \zeta-N+1-z',\ -\zeta-N-z'-a-b;\\
1+z-z';\ 1-N-z'-b,\ -N-z'-a-b+1
\vphantom{\frac12}
\end{split}\right].
\end{gathered}
\end{equation*}
By using the identity above $(\textrm{Term }1')$ can be rewritten as
\begin{equation*}
\begin{gathered}
(\textrm{Term }1'')\ \frac{\pi^7 \cdot\textrm{Sin}[z', \ z'+a+b]\cdot\Gamma[1+z, \ 1+z', \ \zeta+N+2\epsilon, \ N-\zeta]}{\textrm{Sin}[\zeta-z', \ \zeta+a+b, \ \zeta+z'+a+b, \ a, \ \zeta, \ \Sigma+a, \ z+b, \ z'+b]\cdot\Gamma[N+a, \ N, \ \Sigma+1]}\\
\times\frac{L\left[ \begin{split}
1+z,\ -z'-b,\ \zeta-N+1-z',\ -\zeta-N-z'-a-b;\\
1+z-z';\ 1-N-z'-b,\ -N-z'-a-b+1
\vphantom{\frac12}
\end{split}\right]}
{\Gamma[-\zeta, \ -\zeta-a, \ \zeta+2\epsilon, \zeta+b+1, \ z-\zeta+N, \ z'-\zeta+N, \ \zeta+N+z+2\epsilon, \ \zeta+N+z'+2\epsilon]}.
\end{gathered}
\end{equation*}

We now make some transformations to (Term 2). We have $\sin(\pi\Sigma)\cdot \Gamma(-\Sigma) = -\frac{\pi}{\Gamma(\Sigma+1)}$ and $\sin\pi(N+a) = (-1)^N\sin\pi a$; if we combine these with one of the fundamental transformation for $L$-functions, the second equality of Proposition $\ref{prop:fundamental}$ applied to $A = 1+z, \ B = -\zeta-N-z'-a-b, \ C = -\zeta+N+z, \ D = 1+z+b, E = 1-\zeta+z, \ F = 1 - \zeta - a + z, \ G = 1+z-z'$, we find that $(\textrm{Term }2)$ equals
\begin{equation*}
\begin{gathered}
(\textrm{Term }2')\ \frac{\pi^3 (-1)^{N-1}}{\textrm{Sin}[\zeta+z'+a+b, \ a]}\cdot
\Gamma\left[ \begin{split}
\zeta+1-z'-N,\ \zeta+1,\ \zeta+N+2\epsilon, \ 1+z ,\ 1+z+b, \ N+\Sigma+1\\
\zeta+N+z'+2\epsilon,\hspace{.1in} \zeta-N+1,\hspace{.1in} \zeta+2\epsilon,\hspace{.1in} N+a, \hspace{.1in} \Sigma+1
\vphantom{\frac12}
\end{split}\right]\\
\times L\left[ \begin{split}
-z'-b,\ 1+z,\ -\zeta-N-z'-a-b,\  \zeta-N+1-z';\\
-N-z'-b+1 ;\ -N- z' - a - b + 1,\ 1+z-z'
\vphantom{\frac12}
\end{split}\right].
\end{gathered}
\end{equation*}

Next apply the $L$-incoherent relation, Proposition $\ref{incoherent}$, to $A = -z'-b, \ B = 1+z, \ C = -\zeta-N-z'-a-b, \ D=  \zeta-N+1-z', \ E = 1+z-z', \ F = -N-z'-b+1, \ G = -N-z'-a-b+1$. Again since $F - A = 1 - N$, then $\sin \pi (F - A) = 0$, so in the second term of Proposition $\ref{incoherent}$ one of the products of six sine functions vanishes. As before, after several algebraic manipulations, we obtain
\begin{equation*}
\begin{gathered}
\frac{\sin\pi(z'+a+b)}{\sin\pi(z+b)\cdot\Gamma[N, \ -\zeta-a, \ \zeta+b+1, \ \zeta+N+z+2\epsilon, \ -\zeta+N+z]}\times\\
L\left[ \begin{split}
-z'-b, 1+z, -\zeta-N-z'-a-b, \zeta-N+1-z'\\
1+z-z' ;\ 1-N-z'-b,\ 1-N-z'-a-b
\vphantom{\frac12}
\end{split}\right] = -\frac{\Gamma[N+\Sigma+1, 1+z+b, -z']}{\Gamma[-z'-b, 1+z, 1-N-a, -N-\Sigma-a]}\\
\times\frac{L\left[ \begin{split}
1+z'+b,\ -z,\ \zeta+N+z'+2\epsilon, \ -\zeta+N+z'\\
1+z'-z ;\hspace{.1in} N+z'+b+1,\hspace{.1in} N+ z' + 2\epsilon
\vphantom{\frac12}
\end{split}\right]}{\Gamma[-\zeta-N-z'-a-b, \ \zeta-N+1-z', \ \zeta+1, \ -\zeta-a-b]} + \frac{(-1)^N\Gamma[N+\Sigma+1, 1+z+b, -z']}{\pi^4}\\
\times \textrm{Sin}[z'+b,\ \zeta, \ \zeta+a+b,\ \Sigma+a]\cdot
L\left[ \begin{split}
-z'-b,\ 1+z,\ -\zeta-N-z'-a-b,\  \zeta-N+1-z'\\
-N-z'-b+1 ;\hspace{.1in} 1+z-z',\hspace{.1in} -N- z' - a - b + 1
\vphantom{\frac12}
\end{split}\right].
\end{gathered}
\end{equation*}
By applying the latter identity, we can rewrite (Term 1'') as a linear combination of the two $L$-functions
\begin{equation*}
\begin{gathered}
L\left( 1+z'+b,\ -z, \ \zeta+N+z'+2\epsilon,\ -\zeta+N+z'; \ 1+z'-z;\ N+z'+b+1,\ N+z'+2\epsilon \right),\\
L\left( -z'-b, 1+z, -\zeta-N-z'-a-b, \zeta-N+1-z';\ 1-N-z'-b;\ 1+z-z', 1-N-z'-a-b \right).
\end{gathered}
\end{equation*}
More explicitly, (Term 1'') is the sum of the expressions
\begin{equation*}
\begin{gathered}
(\textrm{Term }a') \ \pi \cdot \frac{\Gamma[N+\Sigma+a+1, \ N+\Sigma+1, \ 1+z+b, \ 1+z'+b, \ \zeta+N+2\epsilon, \ N-\zeta]}{\Gamma(\Sigma+1)}\\
\times L\left[ \begin{split}
1+z'+b,\ -z,\ \zeta+N+z'+2\epsilon,\  N-\zeta+z';\\
1+z'-z ;\hspace{.1in} N+z'+b,\hspace{.1in} N+z'+a+b
\vphantom{\frac12}
\end{split}\right],
\end{gathered}
\end{equation*}
\begin{equation*}
\begin{gathered}
(\textrm{Term }b') \ \frac{(-1)^{N-1}\pi^4}{\textrm{Sin}[\zeta-z', \ \zeta+z'+a+b, \ a]}\cdot\Gamma\left[ \begin{split}
\zeta+N+2\epsilon,\ N-\zeta,\ 1+z,\ N+\Sigma+1, \ 1+z+b\\
\Sigma+1,\ N+a,\ -\zeta,\ \zeta+2\epsilon, \ \zeta+N+z'+2\epsilon, \ -\zeta+N+z'
\vphantom{\frac12}
\end{split}\right]\\
\times L\left[ \begin{split}
-z'-b,\ 1+z,\ -\zeta-N-z'-a-b,\  \zeta-N+1-z';\\
-N-z'-b+1 ;\hspace{.1in} 1+z-z',\hspace{.1in} -N-z'-a-b+1
\vphantom{\frac12}
\end{split}\right].
\end{gathered}
\end{equation*}
We have $\TR_{<N}(\zeta) = (\textrm{Term }1) + (\textrm{Term }2) = (\textrm{Term }1') + (\textrm{Term }2') = (\textrm{Term }1'') + (\textrm{Term }2') = (\textrm{Term }a') + (\textrm{Term }b') + (\textrm{Term }2')$. We claim that $(\textrm{Term }2') + (\textrm{Term }b') = 0$ and $(\textrm{Term }a')$ is the expression of $\TR_{<N}(\zeta)$ given in Proposition $\ref{prop:R2S2limit}$. We will be done once we prove this claim.

The fact that $(\textrm{Term }a')$ is equal to the expression of $\TR_{<N}(\zeta)$ given in Proposition $\ref{prop:R2S2limit}$ follows simply from the definition of the $L$-function as a sum of two $_4F_3$ generalized hypergeometric functions, $(\ref{Lfunction})$.

Now we check that $(\textrm{Term }2') + (\textrm{Term }b') = 0$. Each of the expressions $(\textrm{Term }2'),  (\textrm{Term }b')$ is a product of ratios of Gamma and sine functions times an $L$-function. The $L$-functions in both terms are exactly the same, thus we simply need to check that the prefactors of these terms (consisting of Gamma and sine functions) are negative to each other. The latter task is a simple exercise.
\end{proof}

\end{appendix}

\end{document}